\documentclass[letterpaper,12pt]{article}

\usepackage{amsmath}
\usepackage{amssymb,amsthm,graphicx,multicol,latexsym,amsmath,amssymb,fancyhdr, cancel,enumitem,mathrsfs,xfrac,pgfgantt,ebgaramond,float,tocloft,stmaryrd,changepage}
\usepackage{hyperref} 

\usepackage{tikz} 
\usetikzlibrary{arrows.meta, positioning}

\renewcommand{\labelenumi}{\theenumi}
\renewcommand{\theenumi}{{\it(\roman{enumi})}}
\newcommand*\fixitem {\item[]%
  \refstepcounter{enumi}\hskip-\leftmargin\labelenumi\hskip\labelsep}

\def\AP{\ensuremath{{\bf(AP)}}}
\def\APN{\ensuremath{{\bf(AP_{\mathbb{N}})}}}
\def\API{\ensuremath{{\bf(AP_{\mathbb{I}})}}}
\def\Star{\ensuremath{{\bf(\star)}}}
\def\SStar{\ensuremath{{\bf(\star\star)}}}
\def\SSStar{\ensuremath{{\bf(\star\star\star)}}}
\def\CCC{\ensuremath{{\bf(CCC)}}}
\def\CC{\ensuremath{{\bf(CC)}}}
\def\CCN{\ensuremath{{\bf(CC_{\mathbb{N}})}}}
\def\sNIP{\ensuremath{{\bf(NIP_{\text{s}})}}}
\def\wNIP{\ensuremath{{\bf(NIP_{\text{w}})}}}
\def\MCP{\ensuremath{{\bf(MCP)}}}
\def\BWP{\ensuremath{{\bf(BWP)}}}
\def\ES{\ensuremath{{\bf(ES)}}}
\def\CA{\ensuremath{{\bf(CA)}}}
\def\IVT{\ensuremath{{\bf(IVT)}}}
\def\BVT{\ensuremath{{\bf(BVT)}}}
\def\EVT{\ensuremath{{\bf(EVT)}}}
\def\CFT{\ensuremath{{\bf(CFT)}}}
\def\IFT{\ensuremath{{\bf(IFT)}}}
\def\UCT{\ensuremath{{\bf(UCT)}}}
\def\CVT{\ensuremath{{\bf(CVT)}}}
\def\RT{\ensuremath{{\bf(RT)}}}
\def\PCP{\ensuremath{{\bf(PCP)}}}
\def\UAS{\ensuremath{{\bf(UAS)}}}
\def\ADT{\ensuremath{{\bf(ADT)}}}

\def\DIT{\ensuremath{{\bf(DIT)}}}
\def\RIT{\ensuremath{{\bf(RIT)}}}
\def\IAT{\ensuremath{{\bf(IAT)}}}
\def\FTCi{\ensuremath{{\bf(FTC1)}}}
\def\FTCii{\ensuremath{{\bf(FTC2)}}}

\def\LCL{\ensuremath{{\bf(LCL)}}}

\def\MVT{\ensuremath{{\bf(MVT)}}}
\def\eMVT{\ensuremath{{\bf(eMVT)}}}
\def\iMVT{\ensuremath{{\bf(iMVT)}}}
\def\TT{\ensuremath{{\bf(TT_L)}}}

\def\lpar{\left(}
\def\rpar{\right)}
\def\lbrak{\left[}
\def\rbrak{\right]}

\def\abClosed{\lbrak a,b\rbrak}
\def\abOpen{\lpar a,b\rpar}

\def\N{\mathbb{N}}
\def\R{\mathbb{R}}

\def\Inf{\mathbb{I}}

\def\seqlimOp{\displaystyle\lim_{n\rightarrow\infty}}
\def\seqlimOpIII{\displaystyle\lim_{k\rightarrow\infty}}

\def\aseq{(a_k)_{k\in\N}}
\def\anseq{\lpar a_n\rpar_{n\in\N}}
\def\bseq{(b_k)_{k\in\N}}
\def\cseq{(c_n)_{n\in\N}}
\def\csubseq{\lpar c_{N_k}\rpar_{k\in\N}}
\def\tseq{(t_n)_{n\in\N}}
\def\tsubseq{\lpar t_{N_k}\rpar_{k\in\N}}
\def\Nkseq{(N_k)_{k\in\N}}
\def\nseq{(n)_{n\in\N}}
\def\Iseq{\lpar I_k\rpar_{k\in\N}}
\def\sseq{(s_n)_{n\in\N}}
\def\ssubseq{\lpar s_{N_k}\rpar_{k\in\N}}
\def\xseq{(x_n)_{n\in\N}}
\def\xsubseq{\lpar x_{N_k}\rpar_{k\in\N}}

\def\into{\longrightarrow}

\def\theSet{\mathcal{S}}
\def\sub{\subseteq}
\def\setdiff{\backslash}

\def\UIC{\ensuremath{{\bf(I1)}}}
\def\NTD{\ensuremath{{\bf(I2)}}}
\def\AAC{\ensuremath{{\bf(I3)}}}
\def\RIC{\ensuremath{{\bf(I4)}}}
\def\UIK{\ensuremath{{\bf(I5)}}}

\def\setOpen{\mathcal{O}}

\def\txthalf{\textstyle{\frac{1}{2}}}

\def\Base{\mathscr{B}}
\def\Filter{\mathscr{T}}
\def\lFilter{\mathscr{L}}
\def\rFilter{\mathscr{R}}
\def\IntFilter{\mathscr{I}}

\def\Cover{\mathscr{C}}

\def\Riemann{\mathcal{R}}

\def\xcLim{\displaystyle{\lim_{x\rightarrow c}}\ }
\def\xcLimL{\displaystyle{\lim_{x\rightarrow c-}}}
\def\xcLimR{\displaystyle{\lim_{x\rightarrow c+}}}

\def\BcLim{\displaystyle{\lim_{\Base\rightarrow c}}\ }

\def\tcLim{\displaystyle{\lim_{t\rightarrow c}}\ }

\def\xaLim{\displaystyle{\lim_{x\rightarrow a}}\ }

\def\xaLimR{\displaystyle{\lim_{x\rightarrow a+}}}

\def\xbLim{\displaystyle{\lim_{x\rightarrow b}}\ }
\def\xbLimL{\displaystyle{\lim_{x\rightarrow b-}}}

\def\id{\operatorname{id}}

\def\abIntUp{\overline{\int_a^b}}
\def\abIntLow{\underline{\int_a^b}}
\def\abInt{\int_a^b}

\def\aaIntUp{\overline{\int_a^a}}
\def\aaIntLow{\underline{\int_a^a}}

\def\StepsSet{\mathscr{S}}
\def\abSteps{\StepsSet[a,b]}
\def\acSteps{\StepsSet[a,c]}
\def\cbSteps{\StepsSet[c,b]}
\def\abParts{\mathscr{P}[a,b]}

\def\iff{\Longleftrightarrow}

 \newtheorem{theorem}{Theorem}
 \newtheorem{lemma}[theorem]{Lemma}
  \newtheorem{proposition}[theorem]{Proposition}
 \newtheorem{corollary}[theorem]{Corollary}

\theoremstyle{definition}

\newtheorem{example}[theorem]{Example}

\theoremstyle{plain}

\begin{document}
\hrule \vspace{3pt}\medskip
\noindent{ {\Large\bf \!Five Circles: Real Analysis Theorems equivalent to Completeness} \hfill\\ by: \emph{Rafael Cantuba}\footnote{\label{dms}Associate Professor, Department of Mathematics and Statistics,  De La Salle University, 2401 Taft Ave., Malate, Manila 1004, Philippines, ORCID: 0000-0002-4685-8761}} \medskip\hrule\ 

\begin{quote}  \textsc{{Abstract}}. This is an exposition of the work of O. Riemenschneider about five ``circles'' of implications relating real analysis theorems each equivalent to the Dedekind completeness of the real field. These circles cover five elements of real function theory: convergence, connectedness, differentiability, compactness and integration.\\


\end{quote}

\section{Introduction}

Proofs in elementary real analysis may be categorized into two types. There are proofs that are immediate from set theory and logic, such as those that involve the use of nested quantifiers. The real difficulty comes from the other kind of proofs that make use of the Completeness Axiom, or the Dedekind completeness (least upper bound property) of the real number system. The survey \cite{dev14} lists 72 equivalent forms of the completeness axiom, ranging from the Monotone Convergence Principle, the Intermediate Value Theorem, to the Mean Value Theorem, the Uniform Continuity Theorem, one form of the Fundamental Theorem of Calculus, and two forms of the Arzel\`a-Ascoli Theorem. [In \cite{can24}, we also find that Littlewood's Three Principles from measure theory are recent additions to the list.] Two 2013 papers \cite{pro13,tei13} independently explored this perspective on real analysis, but there is a third one that actually came earlier\textemdash a year 2001 unpublished but online-available work \cite{rie01}, of which this exposition is about. It is said to have been the result of an attempt to organize proofs of fundamental real analysis theorems in such a way that one can see their equivalence to the Completeness Axiom without much extra work, but with the ``price'' of having an early introduction to topological concepts. This we do not necessarily see as a disadvantage. It seems that the ``busy'' nature of the endeavor (proving the equivalences) is a good way to get deep into topology in a pace that usually does not happen when traditional topology texts are used. This is, of course, just an opinion.

The real analysis principles are categorized into five ``circles'' where each ``circle'' is more precisely a finite sequence of implications that begins with one statement \Star\  and ends up back to \Star. Any two of the five circles overlap, and the presence of these statements that are at the intersections of circles shows us that the statements in all five circles are equivalent to Dedekind completeness. The themes in these five circles are important notions in real analysis: convergence, connectedness, differentiation, compactness and integration. The contents of this exposition are organized according to these five themes.

\begin{enumerate}
    \item \emph{Convergence.} In the proof of Theorem~\ref{FirstCircle} we show the equivalence of real analysis principles that explicitly involve sequences of real numbers. These include the Cauchy Convergence Criterion, the Monotone Convergence Principle, the Bolzano-Weierstrass Property of (sequences in) $\R$, and two forms of the Nested Intervals Principle. For the benefit of the real analysis student, all the necessary definitions, and theorems with proofs that rely only on these definitions, may be found in Section~\ref{FirstCirclePrelims}, where we introduce the notions of sequence, monotonically increasing sequence, subsequence, Cauchy sequence, bounded sequence, convergent sequence and nested intervals. This is the order that they were defined, which is in accordance to the flow of implications in the proof of Theorem~\ref{FirstCircle}. This style is also used in the preliminaries for the other circles of real analysis principles. Also, some of these principles each have to be joined by logical conjunction with the Archimedean Principle (or the weaker ``Countable Cofinality of $\R$'') in order for it to be included in the circle. Thus, we also have a short set of preliminaries in Section~\ref{APPrelims} devoted to all proofs that depend on the Archimedean Principle.
    \item \emph{Connectedness.} Theorem~\ref{SecondCircle} deals with real analysis principles related to the topological notion of connectedness. These real analysis principles include the Least Upper Bound Property or Existence of Suprema, the Intermediate Value Theorem, Dedekind's Cut Axiom, and some statements about connectedness of intervals. The corresponding preliminaries may be found in Section~\ref{SecondCirclePrelims} in which we introduce upper bounds, suprema, interior points and open sets, points of closure and closed sets, connected sets, total disconnectedness, continuity at a point, continuity on a set, the Pasting Lemma, and also cuts, gaps and cut points.
    \item \emph{Differentiability.} In Theorem~\ref{ThirdCircle}, we collect some of the fundamental theorems of Differential Calculus, such as the Extreme Value Theorem, Rolle's Theorem, the Mean Value Theorem, its extended form, which also called Cauchy's Mean Value Theorem, and we also have Taylor's Theorem, the characterization of polynomials of degree at most $n$ by a zero $(n+1)$st order derivative, the convexity of functions with nonnegative second derivative, the monotonicity of functions with nonnegative first derivative, and the statement that a function with a zero derivative is constant. In the corresponding preliminaries section, which is Section~\ref{ThirdCirclePrelims}, we develop the machinery: limits defined in terms of the topological notion of a filter base, the notion of differentiability, and the algebra of limits, continuous functions and differentiable functions.
    \item \emph{Compactness.} Essentially all real analysis theorems are about compactness, but for the fourth circle of real analysis principles, which are in Theorem~\ref{FourthCircle}, we have those with statements or proofs that explicitly or superficially involve compactness. These are the Lebesgue Covering Lemma, the Uniform Continuity Theorem, the boundedness of functions continuous on a closed and bounded interval, the uniform approximation of continuous functions by step functions, and of course, the statement that the unit interval $[0,1]$ is compact. The preliminaries are in Section~\ref{FourthCircle}, where we introduce the notions of cover and subcover, compact set, Lebesgue number, uniform continuity, partition and step function. 
    \item \emph{Integration.} The theory culminates in proving equivalence between theorems concerning the elementary integral (Darboux/Riemann), which are collected in Theorem~\ref{FifthCircle}, and these include the Darboux/Riemann integrability of functions continuous on a closed and bounded interval, two forms of the Fundamental Theorem of Calculus, the statement that nonnegative continuous functions on a closed and bounded interval have monotonically increasing antiderivatives, and that the difference of any two antiderivatives is constant. As will be seen in the proofs, in comparison to the previous circles, there is increased dependence here on previously proven equivalences. The preliminaries in Section~\ref{FifthCirclePrelims} include integrals of step functions, lower and upper Darboux integrals, Riemann sums, integrable functions defined by the equality of upper and lower Darboux integrals, the Riemann integral as the limit of a filter base of sets of Riemann sums, boundedness of integrable functions, and the additivity of integrals.
\end{enumerate}

The correct technical setting for the proofs is an arbitrary ordered field, which is indeed the approach in \cite{rie01}. However, we chose instead to use a naive approach of doing the proofs in the setting of just the real field $\R$, and this is with the real analysis student in mind. Practically, one of the real analysis principles may be chosen as the completeness axiom for $\R$, and the rest of the principles in the five circles may be proven within $\R$ anyway. The drawback is that some real analysis principles in \cite{rie01} turned out not to be within the scope of this exposition. These real analysis principles (which are mainly about series) that are in \cite{rie01} but are not covered in this work are the ones that are outside the five circles: theorems that, in order to be proven equivalent to Dedekind completeness, would have to involve the following argument type: an ordered subfield of $\R$ would have to be assumed to satisfy the principle as an axiom, and then the subfield shall be proven to be $\R$ itself. It seems awkward to carry out arguments of this type within the said naive approach. 

This exposition is intended for an advanced undergraduate or a beginning graduate student that knows basics of proving in pure mathematics: nested quantifiers, negation of statements with nested quantifiers; rules of inference; direct proofs, proofs by contradiction and by contraposition; mathematical induction; proofs of set inclusions and set equations; using the definition of set unions and intersections in an argument (not to be underestimated); when something can be assumed ``without loss of generality'' or WLOG; using images and inverse images of functions. These are not obtainable from a calculus course or its natural progression to differential equations or to ``higher'' applied mathematics. True higher mathematics has mathematical sophistication: the level of difficulty increases, the method of study increases in complexity, but not necessarily the complexity of the objects of study. Real analysis is concerned only with sequences, intervals and basic objects in calculus, but the mathematical sophistication is high. [See also the introductory paragraph in Section~\ref{FourthCircleSec} about how the topological notion of compactness is a ``gate-keeper topic.'']

\section{Preliminaries} The thematic structure of this exposition is based on the level of proving needed. If the arguments proceed only from definitions, then these are the preliminaries that can be found in Sections~\ref{FirstCirclePrelims}--\ref{FifthCirclePrelims}. If a real analysis argument should be based on Dedekind completeness, then this is the very point of the exposition, and these are the arguments that are in the proofs for the five circles of real analysis principles in Theorems~\ref{FirstCircle}--\ref{FifthCircle}. As detailed in the introduction, the preliminaries needed for Theorems~\ref{FirstCircle}--\ref{FifthCircle} can be found in Sections~\ref{FirstCirclePrelims}--\ref{FifthCirclePrelims}, respectively. Section~\ref{APPrelims} is somewhat midway between Sections~\ref{FirstCirclePrelims}--\ref{FifthCirclePrelims} and Theorems~\ref{FirstCircle}--\ref{FifthCircle}. The Archimedean Principle is a consequence of Dedekind completeness, but the converse (which is beyond the scope of this exposition) is false. Thus, the proofs in Section~\ref{APPrelims} are technically not immediate from real analysis definitions, but are not as involved as the ``compactness arguments'' for Theorems~\ref{FirstCircle}--\ref{FifthCircle}.

The basic intention is to supplement \cite{rie01} with all the basic concepts and proofs that a beginner needs in order to understand the circles of implications, or equivalently, to show that the theory that can be gleaned from the five circles of real analysis principles may indeed be bridged into introductory level real analysis.

\subsection{Sequences of Real Numbers}\label{FirstCirclePrelims}

The set of all positive integers shall be denoted by $\N$, and a function $\N\into\R$ shall be called a \emph{sequence}. The usual notation for functions, such as $c : \N\into\R$ with $c:n\mapsto c(n)$, is customarily NOT used for sequences. Instead, we call $c_n:=c(n)$ as  the \emph{$n$th term} of the sequence with $n$ as the \emph{index} of the term $c_n$. The sequence itself is denoted symbolically by enclosing the $n$th term in parenthesis and indicating as a further subcript that $n\in\N$ is used to index the terms: $\cseq$ which may be read as ``the sequence with terms $c_n$.'' 

\begin{proposition}\label{monoDefProp} Given a sequence $\cseq$, the following are equivalent.
\begin{enumerate}\item\label{monoDefI} For any $m,n\in\N$, if $m< n$, then $c_m\leq c_{n}$ (respectively, $c_m< c_{n}$).
\item\label{monoDefII} For each $n\in\N$, the inequality $c_n\leq c_{n+1}$ (respectively, $c_n< c_{n+1}$) holds.
\end{enumerate}
[In such a case, $\cseq$ is said to be a \emph{monotonically increasing} (respectively, \emph{strictly increasing\footnote{A strictly increasing sequence is monotonically increasing, but the converse is false: any constant sequence (any two terms are equal) is monotonically increasing but not strictly increasing.}}) sequence.]
\end{proposition}
\begin{proof}\ref{monoDefI} $\implies$ \ref{monoDefII} Use the fact that $n< n+1$.\\

\noindent \ref{monoDefII} $\implies$ \ref{monoDefI} We use induction on $n$. If $n=1$, then there is no positive integer $m<n$ and the desired statement is vacuously true. Suppose that for some $n\in\N$, for all positive integers $m<n$, we have $c_m\leq c_{n}$ (respectively, $c_m<c_n$). Equivalently, $c_1\leq c_2\leq \cdots\leq c_n$ (respectively, $c_1< c_2< \cdots < c_n$). But then, we simply use \ref{monoDefII} to produce $c_1\leq c_2\leq \cdots\leq c_n\leq c_{n+1}$ (respectively, $c_1< c_2< \cdots < c_n<c_{n+1}$), which implies that for all positive integers $m<n+1$, we have $c_m\leq c_{n+1}$ (respectively, $c_m< c_{n+1}$). By induction, we get the desired statement.
\end{proof}

If $\Nkseq$ is a sequence such that, for all $k\in\N$, we have $N_k\in\N$, or equivalently, that the rule of assignment $k\mapsto N_k$ is a function $\N\into\N$, then, given a sequence $\cseq$, the composition of $k\mapsto N_k$ with $n\mapsto c_n$ is again a function $\N\into\R$, or is again a sequence. We may denote it by $\csubseq$. If we have the additional condition that $\Nkseq$ is strictly increasing, then $\csubseq$ is said to be a \emph{subsequence} of $\cseq$.

\begin{proposition}\label{subseqNkkProp} If $\csubseq$ is a subsequence of $\cseq$, then for all $k\in\N$, we have $N_k\geq k$.
\end{proposition}
\begin{proof}     Since $N_1\in\N$, we have $N_1\geq 1$. If, for some $k\in\N$, we have $N_k\geq k$, by the definition of subsequence, $N_{k+1}>N_k\geq k$, so $N_{k+1}>k$. This strict inequality means that the smallest possible value for the integer $N_{k+1}$ is one integer above $k$. That is, $N_{k+1}\geq k+1$. By induction, we obtain the desired conclusion.
\end{proof}

A sequence $\cseq$ is said to be a \emph{Cauchy sequence} if for each $\varepsilon>0$, there exists $N\in\N$ such that for all $m,n\in\N$ if $m,n\geq N$, then $|c_m-c_n|<\varepsilon$.

\begin{lemma}\label{CCCtoMCPLem}
If $(c_n)_{n\in\N}$ is monotonically increasing but is not a Cauchy sequence, then there exists $\varepsilon>0$ and a subsequence $(c_{N_k})_{k\in\N}$ of $(c_n)_{n\in\N}$ such that, for each $k\in\N$, we have $c_{N_{k+1}}\geq c_{N_1}+k\varepsilon$.
\end{lemma}

\begin{proof} Since $(c_n)$ is not a Cauchy sequence, there exists an $\varepsilon > 0$ such that
    \begin{enumerate}
        \item[\Star] for all $N\in\N$, there exist $m,n\in\N$ such that $m,n\geq N$ and $|c_m - c_n| \ge \varepsilon$.
    \end{enumerate}
    We show that the above condition implies
        \begin{enumerate}
        \item[\SStar] for all $N\in\N$, there exist $m,n\in\N$ such that $m>n\geq N$ and $c_m  \ge c_n+\varepsilon$.
    \end{enumerate}
Given $N\in\N$, suppose the integers $m,n$ from \Star\   are equal. Since $k\mapsto c_k$ is a function, we have $c_m=c_n$, which implies $0=|c_m-c_n|\geq \varepsilon>0$, and we get the contradiction $0>0$. Henceforth, $m\neq n$, and we assume, WLOG, that $m>n$. Since $\cseq$ is monotonically increasing, this implies $c_m\geq c_n$, or that $c_m-c_n\geq 0$, so the nonnegative number $c_m-c_n$ must be equal to its absolute value. That is, $c_m-c_n=|c_m-c_n|\geq \varepsilon$, which implies $c_m\geq c_n+\varepsilon$, and this proves \SStar\  .

    We now use \SStar\   to construct a subsequence of $\cseq$. For $N=1$, there exist $m_1,N_1\in\N$ with $m_1>N_1\geq 1$ such that $c_{m_1}  \ge c_{N_1}+\varepsilon$. We use \SStar\   on $N=m_1+1$ to obtain $m_2,N_2\in\N$ with $m_2>N_2\geq N=m_1+1>m_1>N_1\geq 1$ and $c_{m_2}\geq c_{N_2}+\varepsilon$. From the previous inequalities, we obtain $N_2>m_1$. Since $\cseq$ is monotonically increasing, $c_{N_2}\geq c_{m_1} \geq c_{N_1}+\varepsilon$, so $c_{N_2} \geq c_{N_1}+\varepsilon$. Suppose that for some integer $k\geq 2$, we have determined the terms $c_{m_t}$ and $c_{N_t}$ of $\cseq$ for all positive integers $t<k$ with $m_t>N_t\geq t$, with $N_1<N_2<\cdots<N_{k-1}$, and if $t\leq k-2$, then $c_{N_{t+1}}\geq c_{N_t}+\varepsilon$. We use \SStar\   on $N=m_{k-1}+1$ to produce $m_k,N_k\in\N$ such that $m_k>N_k\geq N=m_{k-1}+1>m_{k-1}>N_{k-1}\geq k-1$ (in particular, $N_k>N_{k-1}$ and $N_k>m_{k-1}$) and also $c_{m_k}\geq c_{N_k}+\varepsilon$. From $N_k>m_{k-1}$, we obtain $c_{N_k}\geq c_{m_{k-1}}$ and from the inductive hypothesis, $c_{m_{k-1}}\geq c_{N_{k-1}}+\varepsilon$, so $c_{N_k}\geq c_{k-1}+\varepsilon$. By induction, we have produced a sequence $\csubseq$ where $\Nkseq$ is a strictly increasing sequence, and 
    \begin{flalign}
       && c_{N_{k+1}} &\geq c_{N_k}+\varepsilon, &(\mbox{for all }k\in\N).\label{notCauchyEQ}
    \end{flalign}
    Since $\Nkseq$ is strictly increasing, $\csubseq$ is indeed a subsequence of $\cseq$.

    We can now prove the final inequality $c_{N_{k+1}}\geq c_{N_1}+k\varepsilon$ by induction. Setting $k=1$ in \eqref{notCauchyEQ} we obtain the desired inequality at $k=1$. If the desired inequality holds for some $k\in\N$, then adding $\varepsilon$ to both sides, $c_{N_{k+1}}+\varepsilon\geq c_{N_1}+(k+1)\varepsilon$, where the left-hand side, according to \eqref{notCauchyEQ}, is at most $c_{N_{k+2}}$, and this completes the induction.
\end{proof}

Given $\theSet\sub\R$, a sequence $\cseq$ is said to be a sequence \underline{\emph{in}} $\theSet$ if for any $n\in\N$, we have $c_n\in\theSet$.

\begin{proposition}\label{boundedDefProp} Given a sequence $\cseq$, the following are equivalent.
\begin{enumerate}\item\label{boundedDefI} There exists $M>0$ such that for each $n\in\N$, we have $|c_n|\leq M$.
\item\label{boundedDefII} There exists an interval $\abClosed$ such that $\cseq$ is a sequence in $\abClosed$.
\end{enumerate}
[In such a case, $\cseq$ is said to be a \emph{bounded} sequence.]
\end{proposition}
\begin{proof}\ref{boundedDefI} $\implies$ \ref{boundedDefII} Because of the fact that $|c_n|\leq M$ is equivalent to $-M\leq c_n\leq M$, we may use the interval $\abClosed=\lbrak-M,M\rbrak$.\\

\noindent\ref{boundedDefII} $\implies$ \ref{boundedDefI} If $a=0=b$, then we simply use $M=1>0$. Henceforth, we assume that one of $a$ or $b$ is nonzero, so one of $|a|$ or $|b|$ is positive, which means that $M:=\max\{|a|,|b|\}$ is positive. Given $n\in\N$, by assumption, $a\leq c_n\leq b$, where the first inequality can be rewritten as $-c_n\leq -a$. Thus, we have $c_n\leq b\leq |b|\leq \max\{|a|,|b|\}=M$ and $-c_n\leq -a\leq |a|\leq \max\{|a|,|b|\}=M$, which further simplify to $c_n\leq M$ and $-c_n\leq M$, where the left-hand sides are the only possible values of $|c_n|$, so in any case, $|c_n|\leq M$.
\end{proof}

If the constant $M$ in Proposition~\ref{boundedDefProp} is relevant in a given context, we say that the sequence $\cseq$ is bounded by $M$. This extends to the situation where only some of the terms have absolute value at most $M$, in which case we say that those terms are bounded by $M$.

\begin{proposition}\label{CauchyBoundProp} A Cauchy sequence is bounded.
\end{proposition}
\begin{proof} Suppose $\cseq$ is Cauchy. Since $1>0$, there exists $N\in\N$ such that for all $m,n\in\N$,\linebreak if $m,n\geq N$, then $|c_m-c_n|<1$. In particular, this is true for $n=N$. That is, for all indices $m\geq N$, we have $|c_m-c_N|<1$. The left-hand side may be replaced to get $||c_m|-|c_N||<1$, and this is due to the Reverse Triangle Inequality\footnote{Given $a\in\R$: for any $k\in\{-1,1\}$, we have $ka\leq |a|$; and there exists $k\in\{-1,1\}$ such that $ka=|a|$. Given $a,b\in\R$, there exists $k\in\{-1,1\}$ such that  $|a+b|=k(a+b)=ka+kb\leq |a|+|b|$. Given $x,y,z\in\R$, set $a=x-y$ and $b=y-z$ to get the \emph{Triangle Inequality} $|x-z|\leq |x-y|+|y-z|$. Setting $a=u$ and $b=v-u$, we get $-|u-v|\leq |u|-|v|$, while setting $a=u-v$ and $b=v$, we get $|u|-|v|\leq |u-v|$. Combining: $-|u-v|\leq |u|-|v|\leq |u-v|$, which is equivalent to $||u|-|v||\leq |u-v|$, the \emph{Reverse Triangle Inequality}.}. The new left-hand side is at least $|c_m|-|c_N|$. Thus, $|c_m|-|c_N|<1$, and transposing $|c_N|$ to the right, we have proven at this point that for all indices $m\geq N$, we have $|c_m|<|c_N|+1$, which implies $|c_m|\leq|c_N|+1$. That is, all sequence terms at index $m$ and beyond are bounded by $|c_N|+1$. Then the entire sequence is bounded by $M:=\max\{|c_1|,|c_2|,\ldots,|c_N|+1\}$, which is positive because one of the numbers in the set, $|c_N|+1$, is positive.
\end{proof}

A sequence $\cseq$ is said to be \emph{bounded above} if there exists $M\in\R$ such that for all $n\in\N$, we have $c_n\leq M$. If in a given context the constant $M$ is relevant, then we say that $\cseq$ is \emph{bounded above by $M$}.

\begin{proposition}\label{monoboundProp} If a sequence $\cseq$ is monotonically increasing or strictly increasing, then the following are equivalent:
\begin{enumerate}\item\label{monoboundI} $\cseq$ is bounded;
\item\label{monoboundII} $\cseq$ is bounded above.
\end{enumerate}
\end{proposition}
\begin{proof}\ref{monoboundI} $\implies$ \ref{monoboundII} Use the fact that $c_n\leq |c_n|$.\\

\noindent \ref{monoboundII} $\implies$ \ref{monoboundI} By \ref{monoboundII}, there exists $M\in\R$ such that for all $n\in\N$, we have $c_n\leq M$. Since $\cseq$ is monotonically increasing (respectively, strictly increasing), for any $n\in\N$, if $n>1$, then $c_n\geq c_1$ (respectively, $c_n>c_1$, but this also implies $c_n\geq c_1$). The inequality $c_n\geq c_1$ is also true for $n=1$: since $k\mapsto c_k$ is a function, $n=1$ implies $c_n=c_1$, which further implies $c_n\geq c_1$. At this point, we have proven that $n\in\N$ implies $c_1\leq c_n\leq M$. That is, $\cseq$ is a sequence in $\lbrak c_1,M\rbrak$, and is hence a bounded sequence.
\end{proof}

\begin{proposition}\label{strictsubProp} If $\cseq$ is a strictly increasing sequence that is not bounded above, then so is any subsequence of $\cseq$.
\end{proposition}
\begin{proof} Let $M\in\R$, and let $\csubseq$ be a subsequence of $\cseq$. Since $\cseq$ is not bounded above, there exists $K\in\N$ such that $c_K>M$. By Proposition~\ref{subseqNkkProp}, $N_K>K$, and since $\cseq$ is strictly increasing, $c_{N_K}>c_K$. Thus, $c_{N_K}>M$, and $\csubseq$ is not bounded above. 
\end{proof}

A sequence $\cseq$ \emph{converges} to $c\in\R$  if, for each $\varepsilon>0$, there exists $N\in\N$ such that for all $n\in\N$ if $n\geq N$, then $|c_n-c|<\varepsilon$. If there exists $c\in\R$ such that $\cseq$ converges to $c$, then $\cseq$ \emph{converges}, or is \emph{convergent}. 

If $\cseq$ converges to $c\in\R$, then the classic ``epsilon-over-two'' technique may be used to prove that $c$ is unique, which we refer to as the \emph{limit of (the sequence) $\cseq$} and, in symbols, $\seqlimOp c_n:=c$. The epsilon-over-two technique may also be used to prove that a convergent sequence is Cauchy, so by Proposition~\ref{CauchyBoundProp}, a convergent sequence is bounded. The converse is false: the sequence with terms alternating between the values $-1$ and $1$ is bounded but is not convergent. An additional condition needed for boundedness to lead to convergence is one of the reasons the ``First Circle'' of real analysis theorems is relevant, as shall be explored later.

Back to the fact that ``convergent implies Cauchy,'' the converse (that a Cauchy sequence is convergent) is the direction with a disproportionately longer proof, which is another reason for studying the ``First Circle'' of real analysis theorems. For now, we emphasize some simple but subtle consequences of the notion of convergence followed by a lemma that shows an interrelationship of the notion of Cauchy sequence with the convergence of a subsequence.

\begin{lemma}\label{NIPtoBWPLemInew}\begin{enumerate} \fixitem\label{convII} Given $k\in\R$, if $\seqlimOp c_n=c$, then $\seqlimOp kc_n=kc$.
\item\label{convIV} If $\seqlimOp\lambda_n=0$, and for any $n\in\N$, $0\leq c_n\leq \lambda_n$, then\footnote{This is obviously a weaker version of the Squeeze Theorem for sequence limits, but that is all we need at this point.} $\seqlimOp c_n=0$.
\item\label{convI} The assertions $\seqlimOp|c_n-c|=0$, and $\seqlimOp c_n=c$ are equivalent\footnote{The given proof shows that no rules on operations with convergent sequences are needed.}.
\end{enumerate}
\end{lemma}
\begin{proof}\ref{convII} If $k=0$, then $\cseq$ becomes the zero sequence $(0)_{n\in\N}$ which trivially converges to $0$. Otherwise, given $\varepsilon>0$, we find that $\frac{\varepsilon}{|k|}>0$, so there exists $N\in\N$ such that if $n\in\N$ with $n\geq N$, then $|c_n-c|<\frac{\varepsilon}{|k|}$, which implies $|kc_n-kc|<\varepsilon$.\\

\noindent\ref{convIV} If $\cseq$ does not converge to $0$, then there exists $\varepsilon>0$ such that 

\begin{enumerate}
    \item[\Star]\label{SqueezeI} for any $k\in\N$, there exists $N_k\in\N$ with $N_k\geq k$ and $|c_{N_k}|=|c_{N_k}-0|\geq\varepsilon$.
\end{enumerate} From the assumption that  $\seqlimOp\lambda_n=0$, there exists $K\in\N$ such that
\begin{enumerate}[resume]
    \item[\SStar]   if $n\in\N$ and $n\geq K$, then $|\lambda_n|=|\lambda_n-0|<\varepsilon$. 
\end{enumerate} The trick is to set $k=K$ in \Star\ to obtain $N_K\in\N$, $N_K\geq K$ and $|c_{N_K}|\geq \varepsilon$, where the first two conditions make the hypothesis in \SStar\ true at $n=N_K$, so $|\lambda_{N_K}|<\varepsilon$. At this point, we have $|\lambda_{N_K}|<\varepsilon\leq |c_{N_K}|$, or that $|\lambda_{N_K}|<|c_{N_K}|$. But by assumption, all sequences involved have nonnegative terms so they must be equal to their absolute values, and we have $\lambda_{N_K}<c_{N_K}$. This contradicts the assumption that $c_n\leq \lambda_n$ for all $n\in\N$. Therefore, $\cseq$ converges to $0$. \\

\noindent\ref{convI} In the definition of convergent sequence, $|(c_n-c)-0|<\varepsilon$ is equivalent to $|c_n-c|<\varepsilon$.
\end{proof}

\begin{lemma}\label{limsubLem} If $\cseq$ converges to $c$, then so does any subsequence of $\cseq$.
\end{lemma}
\begin{proof} Let $\varepsilon$. Since $c=\seqlimOp c_n$, there exists $K\in\N$ such that for all $n\in\N$, 
\begin{flalign}
    && |c_n-c|&<\varepsilon, &(\mbox{if }n\geq K).\label{EVTEQ}
\end{flalign} Let $k\in\N$ such that $k\geq K$. By the definition of subsequence the case $k>K$ leads to $N_k>N_K$, while the fact that $k\mapsto N_k$ is a function, the case $k=K$ leads to $N_k=N_K$. Thus, $k\geq K$ implies $N_k\geq N_K$. By Proposition~\ref{subseqNkkProp}, $N_k\geq N_K\geq K$, so $N_k\geq K$. Thus, the hypothesis in \eqref{EVTEQ} is true at $n=N_k$. This gives us $\left|c_{N_k}-c\right|<\varepsilon$. At this point, we have proven that $\seqlimOpIII c_{N_k}=c$.
\end{proof}

\begin{proposition}\label{monostrictsubProp} A monotonically increasing sequence that does not converge has a strictly increasing subsequence that, also, does not converge.
\end{proposition}
\begin{proof} Suppose $\cseq$ is a monotonically increasing sequence that does not converge. Let $c_{N_1}:=c_1$. Suppose that for some $k\in\N$, we have identified terms of the sequence $\cseq$, ordered as:

\noindent $c_{N_1}<c_{N_2}<\cdots<c_{N_{k-1}}$, such that $N_1<N_2<\cdots<N_{k-1}$.

Tending towards a contradiction, suppose that for all $n\geq N_{k-1}+1$, we have $c_{N_{k-1}}{\not<}c_n$. The condition $n\geq N_{k-1}+1$ implies $N_{k-1}<n$, which, because $\cseq$ is monotonically increasing, further implies $c_{N_{k-}1}\leq c_n$. In conjunction with $c_{N_{k-1}}{\not<}c_n$, becomes $c_{N_{k-1}}=c_n$, or that $c_n-c_{N_{k-1}}=0$. Thus, given $\varepsilon>0$, there exists $N_{k-1}+1\in\N$ such that if $n\geq N_{k-1}+1$, we have $|c_n-c_{N_{k-1}}|=|0|=0<\varepsilon$. Hence, $\cseq$ converges.$\lightning$

Henceforth, there exists $N_k\geq N_{k-1}+1$ [so $N_k>N_{k-1}$] such that $c_{N_{k-1}}<c_{N_k}$. By induction, for any $k\in\N$, we have identified terms of the sequence $\cseq$, ordered as $c_{N_1}<c_{N_2}<\cdots<c_{N_{k}}$, such that $N_1<N_2<\cdots<N_{k}$. This implies that $\csubseq$ is a strictly increasing subsequence of $\cseq$.

Suppose $\csubseq$ converges (to some $L\in\R$). Given $\varepsilon>0$, there exists $K\in\N$ such that for any integer $k\geq K$, we have $|c_{N_k}-L|<\varepsilon$, which implies $-\varepsilon<c_{N_k}-L<\varepsilon$, and so, $L-\varepsilon<c_{N_k}<L+\varepsilon$. That is,
\begin{enumerate}
    \item[\Star] for any $k\in\N$, if $N_k\geq N_K$, then $L-\varepsilon<c_{N_k}<L+\varepsilon$.
\end{enumerate}
Let $m\in\N$ such that $m\geq N_K$. By Proposition~\ref{subseqNkkProp}, $N_K\leq m\leq N_m$, and using either the fact that $\cseq$ is monotonically increasing, or the fact that $n\mapsto c_n$ is a function, $c_{N_K}\leq c_m\leq c_{N_m}$. Because $k\geq K$, we may set $k=K$ in \Star\ and using the first inequality in the conclusion, we obtain $L-\varepsilon<c_{N_K}\leq c_m\leq c_{N_m}$. From $N_K\leq m\leq N_m$ and the definition of subsequence, we have $K\leq m$. Thus, we may set $k=m$ in \Star\ and use the second inequality in the conclusion to obtain $L-\varepsilon<c_{N_K}\leq c_m\leq c_{N_m}<L+\varepsilon$. This implies $L-\varepsilon<c_n<L+\varepsilon$. Subtracting $L$ from every member and expressing in terms of absolute value, $|c_n-L|<\varepsilon$. Hence, $\cseq$ converges, contradicting an assumption. Therefore, $\csubseq$ does not converge.
\end{proof}

\begin{lemma}\label{BWPtoCCCLem}
A Cauchy sequence that has a convergent subsequence converges to the limit of said subsequence.
\end{lemma}
\begin{proof} Given $\varepsilon>0$, we have $\frac{\varepsilon}{2}>0$. Given a Cauchy sequence $\cseq$, there exists $C\in\N$ such that
    \begin{enumerate}
        \item[\Star] if $m,n\in\N$ and $m,n\geq C$, then $|c_m-c_n|<\frac{\varepsilon}{2}$.
    \end{enumerate}
    Given a convergent subsequence $\csubseq$ of $\cseq$ with $c:=\seqlimOpIII c_{N_k}$, there exists $S\in\N$ such that
        \begin{enumerate}
        \item[\SStar] if $k\in\N$ and $k\geq S$, then $|c_{N_k}-c|<\frac{\varepsilon}{2}$.
    \end{enumerate}

    Let $H:=\max\{S,C\}$. Since $H$ is either $S$ or $C$ which are both in $\N$, we find that $H\in\N$. Also, $H\geq C$, and by Proposition~\ref{subseqNkkProp}, $N_H\geq H\geq C$. For any $n\geq N_H$, we also have $n\geq C$, and by \Star, $|c_n-c_{N_H}|<\frac{\varepsilon}{2}$. From $H=\max\{S,C\}$, we also have $H\geq S$, which, by \SStar, implies $|c_{N_H}-c|<\frac{\varepsilon}{2}$.

    By the Triangle Inequality, $|c_n-c|\leq |c_n-c_{N_H}|+|c_{N_H}-c|<\frac{\varepsilon}{2}+\frac{\varepsilon}{2}=\varepsilon$. Therefore, $\seqlimOp c_n=c$.
\end{proof}
 
We shall also be considering sequences of intervals, which simply just mean functions that send every $n\in\N$ to an interval. The notation shall be the same as in sequences, such as $\Iseq$ but with emphasis that the terms are intervals. For our purposes, we shall be considering only closed and bounded intervals. For each such interval, say, $I=\abClosed$, the \emph{length} of $\abClosed$ is defined as $\ell(I):=b-a$. 

\begin{proposition}\label{NIPtoBWPProp} If $c,x\in I$, then $|c-x|\leq \ell(I)$.
\end{proposition}
\begin{proof} If $I=\abClosed$, then $a\leq c\leq b$ and $a\leq x\leq b$, where the latter implies $-b\leq -x\leq -a$, which we add to the first system of inequalities to obtain $-(b-a)\leq c-x\leq b-a$. By the definition of length, $-\ell(I)\leq c-x\leq \ell(I)$. Therefore, $|c-x|\leq \ell(I)$.
\end{proof}

\begin{proposition} Given a sequence $\Iseq$ of nonempty closed and bounded intervals, the following are equivalent.
\begin{enumerate}\item\label{nestDefI} For any $h,k\in\N$, if $h< k$, then $I_h\sub I_{k}$.
\item\label{nestDefII} For each $k\in\N$, the inclusion $I_k\leq I_{k+1}$ holds.
\end{enumerate}
[In such a case, $\Iseq$ is said to be a sequence of \emph{nested intervals}.]
\end{proposition}
\begin{proof} The proof is analogous to that of Proposition~\ref{monoDefProp}.
\end{proof}

\begin{lemma}\label{MCPtoNIPLem}\begin{enumerate}\fixitem\label{NIPi} If $\Iseq$ is a sequence of nested intervals, then for each $n\in\N$, the sequence $\aseq$ of left-endpoints of the intervals $I_k$ is monotonically increasing and is bounded above by the right-endpoint $b_n$ of $I_n$.
\item\label{NIPii} If $\aseq$ is monotonically increasing and bounded above by $M$, and if $\seqlimOpIII a_k=c$, then for each $k\in\N$, we have $a_k\leq c\leq M$.
\end{enumerate}
\end{lemma}
\begin{proof}\ref{NIPi} Since the intervals $I_k$ are nested, given $k\in\N$, we have $ I_{k+1}\sub I_k$, which implies\linebreak $a_{k+1}\in\lbrak a_{k+1},b_{k+1}\rbrak\sub\lbrak a_k,b_k\rbrak$. Consequently, $a_k\leq a_{k+1}\leq b_k$. By the first inequality, $\aseq$ is monotonically increasing.

Suppose there exist $K,N\in\N$ such that $a_K>b_N$. Since the $N$th interval $I_N=\lbrak a_N,b_N\rbrak$ is nonempty, we have $a_N\leq b_N$, so $a_N\leq b_N<a_K\in\lbrak a_K,b_K\rbrak=I_K$, which implies $a_K\notin I_N$. That is, $I_K$ has an element that is not in $I_N$, or equivalently, $I_K{\not\sub}I_N$. This means that $K{\not\leq}N$. Otherwise, we get a contradiction from the definition of sequence (if $K=N$) or from the definition of nested intervals (if $K<N$). Henceforth, $N< K$. By the definition of nested intervals, $I_N\sub I_K$. Thus, $b_N\in\lbrak a_N,b_N\rbrak\sub\lbrak a_K,b_K\rbrak$, which implies $a_K\leq b_N\leq b_K$, where the first inequality contradicts the assumption that $a_K>b_N$. Therefore, for any $k,n\in\N$, we have $a_k\leq b_n$.\\

\noindent\ref{NIPii} Suppose there exists $K\in\N$ such that $a_K>c$. This means that $a_K-c>0$. By convergence, there exists $M\in\N$ such that for all $k\in\N$, if $k\geq M$, then $|a_k-c|<a_K-c$, where the left-hand side is at least $a_k-c$. As a consequence $a_k-c<a_K-c$, or that $a_k<a_K$. That is, we have shown that for all indices $k\geq M$, we have $a_k<a_K$. This is true, in particular, for any $k> H:=\max\{M,K\}$, and even more particular, for the case $k=H> K$. That is, we have produced indices $H$ and $K$ such that $K<H$ but $a_K>a_H$, contradicting that assumption that $\aseq$ is monotonically increasing. Therefore, for any $k\in\N$, we have $a_k\leq c$.

Suppose $c>M$. This means that $c-M>0$, and by convergence, there exists $K\in\N$ such that for all $k\in\N$, if $k\geq K$, then $|c-a_k|=|a_k-c|<c-M$, where the left-most quantity is at least $c-a_k$. Thus, $c-a_k<c-M$, which implies $M<a_k$. This is true, in particular, at $k=K$. That is, we have produced and index $K$ such that $a_K>M$, contradicting the fact that $\aseq$ is bounded above by $M$. Therefore, $c\leq M$.
\end{proof}

If $\cseq$ is convergent, then setting $k=-1$ in Lemma~\ref{NIPtoBWPLemInew}\ref{convII}, we find that $\lpar-c_n\rpar_{n\in\N}$ is also convergent, with $\seqlimOp(-c_n)=-\seqlimOp c_n$. If $\cseq$ has the property that $\lpar-c_n\rpar_{n\in\N}$ is monotonically increasing, then we say that $\cseq$ is \emph{monotonically decreasing}, while if $\lpar-c_n\rpar_{n\in\N}$ is bounded above by $-M$, then we say that \emph{$\cseq$ is bounded below by $M$}. By these, we immediately have the following consequence of Lemma~\ref{MCPtoNIPLem}\ref{NIPii}.

\begin{corollary}\label{decCor} If $\bseq$ is monotonically decreasing and bounded below by $M$, and if $\seqlimOpIII b_k=c$, then for each $k\in\N$, we have $M\leq c\leq b_k$.
\end{corollary}

By the number of terms (finite or infinite) a sequence $\cseq$ has in a set $\theSet$, we mean the cardinality of the set $\{n\in\N\  :\  c_n\in\theSet\}$, or the pre-image of $\theSet$ under the function $n\mapsto c_n$. This has nothing to do with the set $\{c_n\  :\  n\in\N\}\cap\theSet$ or the intersection of $\theSet$ with the range of $n\mapsto c_n$. Equivalently, we say that $\cseq$ \emph{has finitely (respectively, infinitely) many terms in $\theSet$} if the set of all \underline{indices} $n$ such that $c_n\in\theSet$ is finite (respectively, infinite). This notion is not seen in the statement, but is used in the proof, of the following.

\begin{lemma}\label{NIPtoBWPLemII} If $\cseq$ is a bounded sequence, then there exists a sequence $\Iseq$ of nested intervals and a subsequence $\csubseq$ of $\cseq$ such that for each $k\in\N$, we have $c_{N_k}\in I_k$ and $\ell(I_k)=\frac{b-a}{2^{k-1}}$.
\end{lemma}
\begin{proof} Suppose $\cseq$ is a sequence in $I_1=\abClosed$. 
Suppose that for some $k\in\N$, closed and bounded intervals $\lbrak a_{k-1},b_{k-1}\rbrak= I_{k-1}\sub I_{k-2}\sub\cdots\sub I_1$ have been found such that $\cseq$ has infinitely many terms in all these intervals, and $\ell(I_n)=\frac{b-a}{2^{n-1}}$ for all positive integers $n<k$. If $L_{k-1}:=\lbrak a_{k-1},\frac{a_{k-1}+b_{k-1}}{2}\rbrak$ and $R_{k-1}:=\lbrak\frac{a_{k-1}+b_{k-1}}{2},b_{k-1}\rbrak$, then $I_{k-1}=L_{k-1}\cup R_{k-1}$. If each of $L_{k-1}$ and $R_{k-1}$ has a finite number of terms of $\cseq$, then there are a finite number of terms of $\cseq$ in $L_{k-1}\cup R_{k-1}=I_{k-1}$.$\lightning$ Hence, there exists $I_{k}=\lbrak a_{k},b_{k}\rbrak\in\{L_{k-1},R_{k-1}\}$ (with $I_{k}\sub I_{k-1}$) such that $\cseq$ has infinitely many terms in $I_{k}$, with $\ell(I_{k})=\frac{b_k-a_k}{2}=\txthalf\ell(I_{k-1})=\frac{b-a}{2^{k}}$. By induction, we have produced a sequence $\Iseq$ of nested intervals such that for each $k\in\N$, we have $\ell(I_k)=\frac{b-a}{2^{k-1}}$ and that there are infinitely many terms of $\cseq$ in $I_k$. For $k=1$, since $\cseq$ is a sequence in $I_1=\abClosed$, we set $N_1=1$ and so $c_{N_1}=I_1$. 
Suppose that for some $k\in\N$, indices $N_1<N_2<\cdots<N_{k-1}$ have been identified such that $c_{N_t}\in I_t$ for all positive integers $t<k$. Since $I_{k}$ has infinitely many terms of $\cseq$,  the set of indices $\{n\in\N\  :\  c_n\in I_{k}\}\setdiff\{1,2,\ldots,N_{k-1}\}$ [that is, excluding all positive integers from $1$ to $N_{k-1}$ and not just the preidentified indices $N_1<N_2<\cdots<N_{k-1}$] is still infinite, so there is still one index $N_{k}$ which is now bigger than $N_{k-1}$ for which $c_{N_{k}}\in I_{k}$. By induction, we have produced a subsequence $\csubseq$ of $\cseq$ such that for each $k\in\N$, we have $c_{N_k}\in I_k$.
\end{proof}

\subsection{Suprema and Connected Sets}\label{SecondCirclePrelims}

Given $\theSet\sub\R$ and $b\in\R$, we say that $b$ is an \emph{upper bound} of $\theSet$, or that \emph{$\theSet$ is bounded above by $b$} if for each $x\in\theSet$, we have $x\leq b$. If $c\in\R$ is an upper bound of $\theSet$ such that for any upper bound $b$ of $\theSet$, we have $c\leq b$, then $c$ is said to be a \emph{supremum} or \emph{least upper bound} of $\theSet$. If indeed $\theSet$ has a supremum, then suppose $c_0$ and $c$ are suprema of $\theSet$. The upper bound $c_0$ and the supremum $c$ are related by $c_0\leq c$, while the supremum $c_0$ and the upper bound $c$ are also related by $c\leq c_0$. Hence, $c_0=c$, or that (if $\theSet$ has a supremum, then) the supremum is unique, which we denote by $\sup\theSet$.

\begin{lemma}\label{NIPtoESLem}\begin{enumerate}\fixitem\label{elemIsSup} If $b$ is an element and an upper bound of $\theSet$, then $b=\sup\theSet$.
\item\label{NIPforES} If $\theSet$ is a nonempty subset of $\R$ and $b_1$ is an upper bound of $\theSet$, then either $\sup\theSet$ exists (as an element of $\R$), or there exists a sequence $\Iseq$ of nested intervals such that for any $k\in\N$, the left endpoint $a_k$ of $I_k$ is an element of $\theSet$, the right endpoint $b_k$ of $I_k$ is an upper bound but not an element of $\theSet$, and $\ell(I_k)=\frac{b_1-a_1}{2^{k-1}}$.
\end{enumerate}
\end{lemma}
\begin{proof}\ref{elemIsSup} Let $\beta$ be an upper bound of $\theSet$. Since $b\in\S$, we have $b\leq \beta$. Therefore, $b=\sup\theSet$.\\

\noindent\ref{NIPforES} Since $\theSet$ is nonempty, there exists $a_1\in\theSet$. If $b_1\in\theSet$, then by part~\ref{elemIsSup}, we are done. Suppose otherwise; that is, $b_1\notin\theSet$. 

Suppose that for some $k\in\N$, intervals $I_{k-1}\sub I_{k-2}\sub\cdots\sub I_1$ have been identified such that for all positive integers $n<k$, the left endpoint $a_n$ of $I_n$ is an element of $\theSet$, the right endpoint $b_n$ of $I_n$ is an upper bound but not an element of $\theSet$, and $\ell(I_n)=\frac{b_1-a_1}{2^{n-1}}$. If the midpoint $\frac{a_{k-1}+b_{k-1}}{2}$ of $I_{k-1}$ is an element of $\theSet$, let $a_{k}:=\frac{a_{k-1}+b_{k-1}}{2}$ and $b_{k}:=b_{k-1}$. If the midpoint $\frac{a_{k-1}+b_{k-1}}{2}$ of $I_{k-1}$ is not in $\theSet$, then let $a_{k}:=a_{k-1}$ and $b_{k}:=\frac{a_{k-1}+b_{k-1}}{2}$. In both cases, $a_{k-1}\leq a_{k}\leq b_{k-1}$ and $a_{k-1}\leq b_{k}\leq b_{k-1}$, so $I_{k}:=\lbrak a_{k},b_{k}\rbrak\sub I_{k-1}$, and also in both cases, $a_k$ is an element of $\theSet$ while $b_k$ is an upper bound of $\theSet$. For the length of $I_k$, we find that in both cases, one of the endpoints is $\frac{a_{k-1}+b_{k-1}}{2}$ so its distance from the other endpoint (which is one of $a_{k-1}$ or $b_{k-1}$) is $\ell(I_k)=\frac{b_{k-1}-a_{k-1}}{2}=\txthalf\ell(I_{k-1})=\frac{b_1-a_1}{2^{k-1}}$. If $b_k\in\theSet$, then by part~\ref{elemIsSup}, we are done. Otherwise, the condition that $b_k\notin\theSet$ completes the induction.
\end{proof}

\begin{proposition}\label{supIntProp} If $\theSet\sub\abClosed$, $a\in\theSet$, and $c=\sup\theSet$, then $c\in\abClosed$.
\end{proposition}
\begin{proof} If $x\in\theSet\sub\abClosed$, then $x\leq b$, so $b$ is an upper bound of $\theSet$, so by the definition of supremum, $c\leq b$. Since $a$ is an element, and $c$ is an upper bound, of $\theSet$, we have $a\leq c$. At this point, we have $a\leq c\leq b$, so $c\in\abClosed$.
\end{proof}

\begin{proposition}\label{xiProp} If $a<c=\sup\theSet$, then there exists $\xi\in\theSet$ such that $a<\xi\leq c$.
\end{proposition}
\begin{proof} From the definition of supremum, if $b$ is an upper bound of $\theSet$, then $c\leq b$. By contraposition, if $b<c$, then $b$ is not an upper bound of $\theSet$. By the assumption in the statement, this applies to $b=a$, so $a$ is not an upper bound of $\theSet$. From the definition of upper bound, this means that there exists $\xi\in\theSet$ such that $\xi{\not\leq} a$, or that $a<\xi$. Since $\xi$ is an element, and $c$ is an upper bound, of $\theSet$, we have $\xi\leq c$.
\end{proof}


\begin{quote}
    \noindent\underline{\emph{A review of the standard topology on $\R$}}. Given a set $A$ of real numbers, a real number $x$ is said to be an \emph{interior point} of $A$ if there exists a (bounded) open interval $\abOpen$ such that $x\in \abOpen$ and $\abOpen\sub A$, while $x$ is said to be a \emph{point of closure} of $A$ if, for any open interval $\abOpen$, if $x\in \abOpen$, then $\abOpen\cap A\neq\emptyset$ (or equivalently, $\abOpen{\not\sub}A^c$, where $A^c:=\R\setdiff A$). The \emph{interior} $A^o$ of $A$ is the set of all interior points of $A$, while the \emph{closure} $A^-$ is the set of all points of closure. The aforementioned definitions of interior point and point of closure have, as immediate consequences, the set inclusions $A^o\sub A$ and $A\sub A^-$. If $A\sub A^o$, then $A$ is said to be an \emph{open set}, while $A$ is a \emph{closed set} if $A^-\sub A$.  Thus, $A$ is open if and only if $A= A^o$, and $A$ is closed if and only if $A^-= A$. All bounded open intervals are trivially open. Since the interval $\R=\lpar-\infty,\infty\rpar$ contains any bounded open interval, $\R$ is trivially open. The empty set $\emptyset$ is vacuously open. 

If $x$ is an element of the union $\setOpen$ of all sets in a collection $\{A_\lambda\  :\  \lambda\in\Lambda \}$ of open sets, then $x$ is in one of these sets $A_\ell\sub A_\ell^o$ for some $\ell\in\Lambda $, so there exists an open interval $I\sub A_\ell\sub\setOpen$ such that $x\in I$. This means $x\in\setOpen^o$, so $\setOpen\sub\setOpen^o$, and $\setOpen$ is open. That is, the union of an arbitrary collection of open sets is open. Given $a,b\in\R$, we have $(-\infty,a)=\bigcup_{x\in(-\infty,a)}(x,a)$ and $(b,\infty)=\bigcup_{x\in(b,\infty)}(b,x)$ where the bounded open intervals that form the union are open. Thus, the unbounded open intervals $(-\infty,a)$ and $(b,\infty)$ are open.

The intersection of two bounded open intervals $I$ and $J$ is also a bounded open interval. Simply take the smaller of the right endpoints, and the bigger of the left endpoints of $I$ and $J$. If $A$ and $B$ are open, and if $x\in A\cap B$, then $x\in A$ implies there exists a bounded open interval $I_1$ such that $x\in I_1\sub A$ and $x\in B$ implies there exists a bounded open interval $I_2$ such that $x\in I_2\sub B$. Consequently, $x\in I_1\cap I_2\sub A\cap B$, where $I_1\cap I_2$ is a bounded open interval. This means $x\in\lpar A\cap B\rpar^o$, so $A\cap B\sub\lpar A\cap B\rpar^o $, and $A\cap B$ is open. Using induction, the intersection of any finite number of open sets is open. 

If $A\sub B$, then the aforementioned statements of the definitions of interior point and point of closure shall now end with ``$\abOpen\sub A\sub B$,'' and ``$\emptyset\neq\abOpen\cap A\sub\abOpen\cap B$,'' respectively. Thus, $A^o\sub B^o$ and $A^-\sub B^-$ are both consequences of $A\sub B$. 

Of the aforementioned definition of point of closure, with the alternative ending ``$I{\cancel\sub}A^c$,'' an immediate consequence is that $x$ is NOT an interior point of $A$ if and only if $x$ is a point of closure of $A^c$, and also that $x$ is an interior point of $A^c$ if and only if $x$ is NOT a point of closure of $\lpar A^c\rpar^c=A$. Consequently, $x\in\lpar A^o\rpar^c$ if and only if $x\in\lpar A^c\rpar^-$, while $x\in\lpar A^c\rpar^o$ if and only if $x\in\lpar A^-\rpar^c$. These give us the identities $\lpar A^o\rpar^c=\lpar A^c\rpar^-$ and $\lpar A^c\rpar^o=\lpar A^-\rpar^c$.

Because of the aforementioned identities, a set $A$ of real numbers is open if and only if\linebreak $A^o=A$ if and only if $(A^o)^c=A^c$ if and only if $(A^c)^-=A^c$ if and only if $A^c$ is closed. That is, $A$ is open if and only if $A^c$ is closed. Since this also applies to $A^c$, we find that $A^c$ is open if and only if $A=(A^c)^c$ is closed. The complement of a closed and bounded interval $\abClosed$ is $\lpar-\infty,a\rpar\cup\lpar b,\infty\rpar$, a union of open sets and is itself open, so $\abClosed$ is closed. Since $(-\infty,a)$ and $(b,\infty)$ are open, their complements $[a,\infty)=(-\infty,a)^c$ and $(-\infty,b]=(b,\infty)^c$ are closed.

If $\{A_\lambda\  :\  \lambda\in\Lambda \}$ is a collection of closed sets, then $\{A_\lambda^c\  :\  \lambda\in\Lambda \}$ is a collection of open sets, so $\bigcup_{\lambda\in\Lambda }A_\lambda^c$ is open, and $\lpar\bigcup_{\lambda\in\Lambda }A_\lambda^c\rpar^c=\bigcap_{\lambda\in\Lambda }A_\lambda$ is closed. That is, the intersection of an arbitrary collection of closed sets is closed. If $A_1$, $A_2$, $\ldots$ , $A_N$ are closed sets, then $A_1^c$, $A_2^c$, $\ldots$ , $A_N^c$ are open sets, and so is $\bigcap_{k=1}^N A_k^c$. Thus, $\bigcup_{k=1}^NA_k=\lpar\bigcap_{k=1}^N A_k^c\rpar^c$ is closed. That is, the union of finitely many closed sets is closed.

At this point, we have reviewed all the topological notions that we need for succeeding proofs. We mention a bit more, which is concerning how the notion of open set relates to the distance function in the real number system, or the function $\R\times\R\into\R$ denoted by $(x,c)\mapsto |x-c|$.

The assertion that $c\in\abOpen$ is equivalent to the condition that there exists $\delta>0$ such that\linebreak $(c-\delta,c+\delta)\sub\abOpen$. [For necessity, use $\delta:=\min\{b-c,c-a\}$, while sufficiency follows from $c-\delta<c<c+\delta$.] This equivalence can be used to prove that $c$ is an interior point of $G$ if and only if there exists $\delta>0$ such that $(c-\delta,c+\delta)\sub G$. Thus, $G$ is open if and only if for each $c\in G$, there exists $\delta>0$ such that $(c-\delta,c+\delta)\sub G$.

For any $x,c\in\R$ the condition that $x\in(c-\delta,c+\delta)$ is true if and only if $c-\delta<x<c+\delta$ if and only if $-\delta<x-c<\delta$ if and only if $|x-c|<\delta$. Henceforth, any assertion of the form $x\in(c-\delta,c+\delta)$ can be used interchangeably with $|x-c|<\delta$. For instance, $(c-\delta,c+\delta)\sub G$, which is equivalent to ``$x\in (c-\delta,c+\delta)$ implies $x\in G$,'' is further equivalent to ``$|x-c|<\delta$ implies $x\in G$.'' Thus, $G$ is open if and only if for each $c\in G$, there exists $\delta>0$ such that $|x-c|<\delta$ implies $x\in G$.
\end{quote}

\begin{proposition}\label{supCloseProp} If $c=\sup\theSet$, then $c\in \theSet^-$.
\end{proposition}
\begin{proof} Given an interval $\abOpen$ such that $c\in\abOpen$, we have $a<c<b$, and by Proposition~\ref{xiProp}, there exists $\xi\in\theSet$ such that $a<\xi\leq c<b$. Equivalently, $\xi\in\theSet$ and $\xi\in\abOpen$. That is, $\xi\in\abOpen\cap\theSet$, so this intersection is not empty. Therefore, $c$ is a point of closure of $\theSet$. 
\end{proof}

\begin{proposition}\label{limCloseProp} If $\cseq$ is a sequence in $\theSet$ that is convergent, then $\seqlimOp c_n\in S^-$.
\end{proposition}
\begin{proof} Let $c:=\seqlimOp c_n$, let $\abOpen$ be an interval such that $c\in\abOpen$. Since $\abOpen$ is an open set, there exists $\delta>0$ such that $|x-c|<\delta$ implies $x\in\abOpen$. Since $\delta>0$ and $c=\seqlimOp c_n$, there exists $N\in\N$ such that for all $n\in\N$, if $n\geq N$ then $|c_n-c|<\delta$. In particular, at $n=N$, we have $|c_N-c|<\delta$, which earlier has been shown to imply $c_N\in\abOpen$. But since $\cseq$ is a sequence in $\theSet$, we have $c_N\in\theSet$, so $c_N\in\abOpen\cap\theSet$, so this intersection is not empty. At this point, we have proven that $c$ is a point of closure of $\theSet$.  
\end{proof}

If $U\sub I\sub\R$, then we say that $U$ is \emph{open relative to $I$} if there exists an open set $G$ such that $U=G\cap I$. This definition of relatively open set and our previous definition of open set are related in the following manner: a set $U\sub\R$ is an open set if and only if $U$ is open relative to $\R$.

\begin{proposition}\label{closedRelProp} If $U\sub I\sub\R$, then the following are equivalent.
\begin{enumerate}
    \item\label{closedRelI} There exists a closed set $F$ such that $U=F\cap I$.
    \item\label{closedRelII} The set $I\setdiff U$ is open relative to $I$.
\end{enumerate}
[If one of the above conditions is true, then we say that \emph{$U$ is closed relative to $I$}.]
\end{proposition}
\begin{proof} \ref{closedRelI} $\implies$\ref{closedRelII} The sets $F^c$ and $F$ have the property that $\R=F\cup F^c$. If we take the intersection of both sides with $I\sub\R$, we obtain $I=(F\cap I)\cup(F^c\cap I)$, where the right-hand side is equal to $U\cup(F^c\cap I)$. Consequently, $I\setdiff U=F^c\cap I$, where $F^c$ is open because $F$ is closed. Therefore, $I\setdiff U$ is open relative to $I$.\\

\noindent \ref{closedRelII} $\implies$\ref{closedRelI} If there exists an open set $G$ such that $G\cap I=I\setdiff U$, then $(G\cap I)^c=(I\cap U^c)^c$, and $G^c\cup I^c=I^c\cup U$. We take the intersection of both sides with $I$ to obtain $(G^c\cap I)\cup\emptyset=\emptyset\cup(U\cap I)$, where $U\cap I$ reduces to $U$ because $U\sub I$. Thus,  $G^c\cap I=U$, where $G^c$ is closed because $G$ is open.
\end{proof}

\begin{corollary}\label{closedRelCor} A set $U$ is open relative to $I$ if and only if $I\setdiff U$ is closed relative to $I$.
\end{corollary}
\begin{proof} A set $V$ is closed relative to $I$ if and only if (from Proposition~\ref{closedRelProp}\ref{closedRelII}) $I\setdiff V$ is open relative to $I$. Set $V=I\setdiff U$.
\end{proof}

\begin{proposition}\label{condefProp} Given $I\sub\R$, the following are equivalent.
\begin{enumerate}

    \item\label{condefI} There exist nonempty disjoint subsets $A$ and $B$ of $I$, both open relative to $I$, such that $I=A\cup B$.
    \item\label{condefII} There exists a nonempty subset $U$ of $I$, both open and closed relative to $I$, such that $U\neq I$.
\end{enumerate}

[A set that does NOT satisfy one of the above statements is said to be a \emph{connected set}.]
\end{proposition}
\begin{proof}\ref{condefI} $\implies$ \ref{condefII} We show that the desired set is $U:=A$. The condition $I=A\cup B$ implies $A=I\setdiff B$. Since $B$ is open relative to $I$, by Proposition~\ref{closedRelCor}, $A$ is closed relative to $I$. Since $A$ and $B$ are disjoint and nonempty, there exist $a\in A$ and $b\in B$ such that $a\neq b$. Thus, $b$ is an element of $B\sub I$ that is not in $A$, so $I$ is not a subset of $A$, and hence cannot be equal to $A$.\\

\noindent\ref{condefII} $\implies$ \ref{condefI} We show that the desired sets are $A:=U$ and $B:=I\setdiff U$, which are disjoint and with union $I$. Since $U$ is nonempty and open relative to $I$, so is $A$, and since $U$ is closed relative to $I$, by Proposition~\ref{closedRelProp}, $B$ is open relative to $I$. Since $A=U\neq I$, we find that $I$ is not a subset of $A=U$, so there exists $b\in I$ such that $b\notin U$, or that $b\in I\setdiff U=B$, so $B\neq\emptyset$.
\end{proof}

\begin{example}\label{disconEx} The set $\{0,1\}$ is not connected. Its subsets $A:=\{0\}$ and $B:=\{1\}$ are nonempty, disjoint and have $\{0,1\}$ as union. The open sets $\lpar -\infty,\txthalf\rpar$ and $\lpar\txthalf,\infty\rpar$ have the properties $A=\lpar-\infty,\txthalf\rpar\cap \{0,1\}$ and $B=\{0,1\}\cap\lpar\txthalf,\infty\rpar$, so $A$ and $B$ are both open relative to $\{0,1\}$.
\end{example}

\begin{lemma}\label{ConnectedLem} Let $I\sub\R$, and let $a\in I$. A necessary and sufficient condition for $I$ to be connected is that if $U$ is a subset of $I$ that contains $a$ and is both open and closed relative to $I$, then $U=I$.
\end{lemma}
\begin{proof} $(\!\!\implies\!\!)$ Suppose $I$ is connected. By our definition of connectedness, this is equivalent to the negation of Proposition~\ref{condefProp}\ref{condefII}, which is that any nonempty subset $U$ of $I$, both open and closed relative to $I$, is equal to $I$. This applies to the case when $U$ contains $a$.\\

\noindent$(\!\!\impliedby\!\!)$ Let $U$ be a nonempty subset of $I$ that is both open and closed relative to $I$. 

We show that $a\in U$. Otherwise, $I\setdiff U$ contains $a$, and is hence nonempty. Since $U$ is both open and closed relative to $I$, by Proposition~\ref{closedRelProp} and Corollary~\ref{closedRelCor}, so is $I\setdiff U$. Since $a\in I\setdiff U$, by assumption, $I\setdiff U=I$, which implies $U=\emptyset$, contradicting the assumption that $U$ is nonempty. Hence, $a\in U$, and by assumption, $U=I$. At this point, we have proven that $I$ satisfies the negation of Proposition~\ref{condefProp}\ref{condefII}. Therefore, $I$ is connected.
\end{proof}

\begin{example}\label{singConEx} A singleton $\{r\}$ is connected. The open set $(r-1,r+1)$ and the closed set $[r-1,r+1]$ have the properties $\{r\}=(r-1,r+1)\cap\{r\}$ and $\{r\}=[r-1,r+1]\cap\{r\}$, so $\{r\}$ is open and closed relative to itself, and the only possible subset  that contains $r$ that is both open and closed relative to $\{r\}$ is itself. By Lemma~\ref{ConnectedLem}, $\{r\}$ is connected.
\end{example}

The conclusion of the characterization of connected set in the statement of Lemma~\ref{ConnectedLem} is always true when the set in question is a singleton (contains only one element). A subset of $\R$ for which any connected subset is a singleton is said to be \emph{totally disconnected}.

\begin{proposition}\label{conNecProp}If $I$ is connected, then for any $a,b\in I$, we have $\abClosed\sub I$.
\end{proposition}
\begin{proof}Suppose $I$ is connected, but there exist $a,b\in I$ such that $\abClosed{\not\sub}I$, so there exists $c\in\abClosed$ such that $c\notin I$. From $c\in\abClosed$, we obtain $a\leq c\leq b$, while from $c\notin I$ and $a,b\in I$, we find that $a\neq c$ and $c\neq b$. Thus, $a<c<b$. The intervals $(-\infty,c)$ and $(c,\infty)$ are open sets, and so, the subsets $A:=\{x\in I\  :\  x<c\}=I\cap(-\infty,c)$ and $B:=\{x\in I\  :\  c<x\}=I\cap(c,\infty)$ of $I$ are both open relative to $I$. From $a<c<b$ and $a,b\in I$, we find that $a\in A$ and $b\in B$, so $A$ and $B$ are nonempty. If $x\in I$, then by the Trichotomy Law, either $x<c$, $x=c$ or $c<x$. We cannot have the second possibility because $c\notin I$. Thus, $x<c$ or $c<x$, and this proves that $I\sub (-\infty,c)\cup(c,\infty)$. Taking the intersection of both sides with $I$, we obtain $I\sub (I\cap (-\infty,c))\cup(I\cap(c,\infty))=A\cup B$. Since $A,B\sub I$, we also have $A\cup B\sub I$. Hence, $I=A\cup B$, and this shows that $I$ is not connected.$\lightning$ Therefore, for any $a,b\in I$, we have $\abClosed\sub I$.
\end{proof}

\begin{proposition}\label{conIntProp} Let $I$ be an interval. If any interval $\abClosed\sub I$ is connected, then $I$ is connected.
\end{proposition}
\begin{proof} Let $I$ be an interval. We prove the rest of the statement by contraposition. Suppose $I$ is not connected. This means that there exist open sets $G_1$ and $G_2$ such that the subsets $A:=G_1\cap I$ and $B:=G_2\cap I$ of $I$ are nonempty, disjoint and both open relative to $I$, with $I=A\cup B$. Since $A$ and $B$ are nonempty, there exist $a\in A$ and $b\in B$. Since $A$ and $B$ are disjoint, we have $a\neq b$, and, WLOG, we assume $a<b$, so $\abClosed\neq\emptyset$. Since $I$ is an interval, the conditions $a,b\in I$ and $a<b$ imply $\abClosed\sub I$. 

The sets $X:=G_1\cap \abClosed$ and $Y:=G_2\cap\abClosed$ are both open relative to $\abClosed$. We take the intersection with $\abClosed$ of both sides of $I=A\cup B$, substitute $A=G_1\cap I$ and $B=G_2\cap I$, and after some set-theoretic manipulation, every instance of $\abClosed\cap I$ may be replaced by $\abClosed$ because $\abClosed\sub I$. The result is $\abClosed=\lpar G_1\cap \abClosed\rpar\cup\lpar G_2\cap\abClosed\rpar =X\cup Y$. 

If there exists $k\in X\cap Y=\lpar G_1\cap \abClosed\rpar\cup\lpar G_2\cap\abClosed\rpar\sub \lpar G_1\cap I\rpar\cup\lpar G_2\cap I\rpar=A\cap B$, then $x\in A\cap B$, contradicting the disjointness of $A$ and $B$. At this point, we have proven that $\abClosed$ is not connected. Therefore, there exists a closed and bounded interval contained in $I$ that is not connected, and this completes the proof.
\end{proof}

Given a function $f:I\into\R$, and given $c\in I$, we say that $f$ is \emph{continuous at $c$} if for each $\varepsilon>0$, there exists $\delta>0$ such that for any $x\in I$, if $|x-c|<\delta$, then $\left| f(x)-f(c)\right|<\varepsilon$.

\begin{example}[\emph{Constant and linear functions are continuous}]\label{contEx} The above definition of ``continuity at a point'' is the so-called ``epsilon-delta'' definition. This assertion is trivially true for a \emph{constant function} $f(x)=k$ for some constant $k$ since the conclusion becomes $|k-k|=0<\varepsilon$. Thus, constant functions are continuous at any real number in the declared domain. If $f(x)=mx+b$ for some constants $m\neq 0$ and $b$ (or that $f$ is a ``linear function''), then  the usual ``epsilon-delta'' proof that makes use of $\delta=\frac{\varepsilon}{|m|}>0$ may be used to show that $f$ is continuous. In general, for any $m,b\in\R$ the function $x\mapsto mx+b$ is continuous at any element of any subset of $\R$. 
\end{example}

\begin{proposition}\label{contDefProp} Given a function $f:I\into\R$, the following are equivalent.
\begin{enumerate}
    \item\label{contDefI} For each $c\in I$, the function $f$ is continuous at $c$.
    \item\label{contDefIII} For any open set $G\sub\R$, the inverse image $f^{-1}[G]:=\{t\in I\  :\  f(t)\in G\}$ (of $G$ under $f$) is open relative to $I$.
       \item\label{contDefIV} For any closed set $F\sub\R$, the inverse image $f^{-1}[F]$ (of $G$ under $f$) is closed relative to $I$.
\end{enumerate}
[A function $f$ that satisfies one of the conditions above is said to be \emph{continuous on $I$}.]
\end{proposition}
\begin{proof}\ref{contDefI} $\implies$ \ref{contDefIII} Let $G$ be an open set, and let $c\in f^{-1}[G]$. Thus, $f(c)\in G$, and since $G$ is open, there exists $\varepsilon>0$ such that $|y-f(c)|=|f(c)-y|<\varepsilon$ implies $y\in G$. By \ref{contDefI}, there exists $\delta>0$ such that $x\in I$ and $|x-c|<\delta$ imply $|f(x)-f(c)|<\varepsilon$. Thus, if $x\in I$ and $|x-c|<\delta$, then $f(x)\in G$, which further implies $x\in f^{-1}[G]$. That is, $x\in I$ and $x\in(c-\delta,c+\delta)$ imply $x\in f^{-1}[G]$. Equivalently, $I\cap\lpar c-\delta,c+\delta\rpar\sub f^{-1}[G]$. If $Q$ is the union of all intervals (which are open sets) $\lpar c-\delta,c+\delta\rpar$ for all $c\in f^{-1}[G]$, then $Q$ is an open set. From $(c-\delta,c+\delta)\sub f^{-1}[G]$, we obtain \begin{equation}
    I\cap Q=I\cap\lpar\bigcup_{c\in f^{-1}[G]} \lpar c-\delta,c+\delta\rpar \rpar=\bigcup_{c\in f^{-1}[G]} \lpar I\cap\lpar c-\delta,c+\delta\rpar\rpar\sub f^{-1}[G],\label{ContEqSet}
\end{equation}
but then, the arbitrary element $c$ of $f^{-1}[G]\sub I$ is contained in $\bigcup_{c\in f^{-1}[G]} \lpar c-\delta,c+\delta\rpar=Q$. Thus, $c\in I$ and $c\in Q$, so $c\in I\cap Q$. This proves $f^{-1}[G]\sub I\cap Q$, which, in conjunction with \eqref{ContEqSet}, gives us $f^{-1}[G]= I\cap Q$, where $Q$ is an open set. Therefore, $f^{-1}[G]$ is open relative to $I$.\\

\noindent\ref{contDefIII} $\implies$ \ref{contDefIV} If $F$ is a closed set, then $F^c$ is open. By \ref{contDefIII}, $f^{-1}[F^c]$ is open relative to $I$. By Corollary~\ref{closedRelCor}, the set $I\setdiff f^{-1}[F^c]=I\cap [f^{-1}[F^c]]^c=\cap [[f^{-1}[F]]^c]^c=I\cap f^{-1}[F]=f^{-1}[F]$ is closed relative to $I$ (where the last set equality is because $f^{-1}[F]\sub I$).\\

\noindent\ref{contDefIV} $\implies$ \ref{contDefI} Let $c\in I$, and let $\varepsilon>0$. The interval $J:=(f(c)-\varepsilon,f(c)+\varepsilon)$ is an open set, so $J^c$ is closed. By \ref{contDefIII}, the inverse image $f^{-1}[J^c]$ is closed relative to $I$. This implies that $$I\setdiff f^{-1}[J^c]=I\cap [f^{-1}[J^c]]^c=\cap [[f^{-1}[J]]^c]^c=I\cap f^{-1}[J]=f^{-1}[J],$$ is open relative to $I$. (The last set equality is because $f^{-1}[J]\sub I$.) Since $f(c)\in(f(c)-\varepsilon,f(c)+\varepsilon)=J$, we have $c\in f^{-1}[J]$. Since $f^{-1}[J]$ is open relative to $I$, there exists an open set $G$ such that $f^{-1}[J]=G\cap I$, so $c\in f^{-1}[J]=G\cap I\sub G$. Since $G$ is open, there exists $\delta>0$ such that $(c-\delta,c+\delta)\sub G$, which further implies $I\cap(c-\delta,c+\delta)\sub G\cap I$, where the right-hand side is equal to $f^{-1}[J]$. That is, $I\cap(c-\delta,c+\delta)\sub f^{-1}[J]$. Equivalently, $x\in I$ and $x\in(c-\delta,c+\delta)$ [which is equivalent to $|x-c|<\delta$] imply $x\in f^{-1}[J]$. If indeed $x\in I$ and $|x-c|<\delta$, then $x\in f^{-1}[J]$, so $f(x)\in J=(f(c)-\varepsilon,f(c)+\varepsilon)$, which implies $|f(x)-f(c)|<\varepsilon$.
\end{proof}

\begin{proposition}\label{contSeqProp} If $f$ is continuous on $I$, and if $\cseq$ is a sequence in $I$ that converges to $c\in I$, then $\seqlimOp f(c_n)=f(c)$.
\end{proposition}
\begin{proof} Let $\varepsilon>0$. Since $f$ is continuous at $c$, there exists $\delta>0$ such that 
\begin{flalign}
&&    |f(x)-f(c)| &<\varepsilon, &(\mbox{if }x\in I, |x-c|<\delta).\label{contseqEq}
\end{flalign}
Since $\delta>0$ and $\cseq$ converges to $c$, there exists $N\in\N$ such that for all $n\in\N$, if $n\geq N$, then $|c_n-c|<\delta$. If indeed $n\in\N$ with $n\geq N$, then $|c_n-c|<\delta$, and, since $\cseq$ is a sequence in $I$, we have $c_n\in I$. Thus, the hypotheses in \eqref{contseqEq} are true at $x=c_n$, so $|f(c_n)-f(c)|<\varepsilon$. Therefore, $\seqlimOp f(c_n)=f(c)$.
\end{proof}

\begin{proposition}\label{contImProp} If $I$ is connected and $f$ is continuous on $I$, then the image set $f[I]:=\{f(x)\  :\ x\in I\}$ (of $I$ under $f$) is connected.
\end{proposition}
\begin{proof} Suppose $I$ is connected but, tending towards a contradiction, that the image set $f[I]$ is not connected. This means that there exist open sets $G_1$ and $G_2$ such that $A:=G_1\cap f[I]$ and $B:=G_2\cap f[I]$ are nonempty and disjoint, with $f[I]=A\cup B$. Since $f$ is continuous, the inverse images $f^{-1}[G_1]$ and $f^{-1}[G_2]$ are both open sets. Since taking inverse images preserves intersections, $f^{-1}[A]=f^{-1}[G_1]\cap f^{-1}[f[I]]$ and $f^{-1}[B]=f^{-1}[G_2]\cap f^{-1}[f[I]]$, on which we use the set-theoretic rule $f^{-1}[f[I]]=I$ (which is valid for the domain $I$ of $f$), and so, we obtain $f^{-1}[A]=f^{-1}[G_1]\cap I$ and $f^{-1}[B]=f^{-1}[G_2]\cap I$, which mean that these inverse images are both open relative to $I$. Since $A$ and $B$ are nonempty, these exist $a\in A=G_1\cap f[I]\sub f[I]$ and $b\in B=G_2\cap f[I]\sub f[I]$, and so there exist $t_1,t_2\in I$ such that $f(t_1)=a\in A$ and $f(t_2)=b\in B$. Thus, $t_1\in f^{-1}[A]$ and $t_2\in f^{-1}[B]$, and these inverse image sets are hence nonempty. If there exists $k\in f^{-1}[A]\cap f^{-1}[B]$, then $f(k)$ is in both $A$ and $B$ contradicting the fact that $A$ and $B$ are disjoint. Thus, $f^{-1}[A]\cap f^{-1}[B]=\emptyset$. Since taking inverse images also preserves unions, taking the inverse images of both sides of $f[I]=A\cup B$ under $f$, and then using the rule $f^{-1}[f[I]]=I$, results to $I=f^{-1}[A]\cup f^{-1}[B]$. At this point, we have completed a proof that $I$ is not connected.$\lightning$ Therefore, $f[I]$ is connected.
\end{proof}

\begin{proposition}\label{homeoProp} Given two closed and bounded intervals $\lbrak a_0,b_0\rbrak$ and $\abClosed$, with $a_0<b_0$, there exists a continuous function $f:\lbrak a_0,b_0\rbrak\into \abClosed$ such that the image of $\lbrak a_0,b_0\rbrak$ under $f$ is equal to $\abClosed$.
\end{proposition}
\begin{proof} Since $a_0<b_0$, $\frac{1}{b_0-a_0}$ exists in $\R$, and the function $f$ defined, for each $x\in I:=\lbrak a_0,b_0\rbrak$, by 
\begin{equation}
    f(x)=\frac{b-a}{b_0-a_0}(x-a_0)+a=\frac{b-a}{b_0-a_0}x+\frac{ab_0-a_0b}{b_0-a_0},\nonumber
\end{equation}
according to Example~\ref{contEx}, is continuous on $I$. If $a_0\leq x\leq b_0$, then subtracting $a_0$ from every member, dividing every member by $b_0-a_0$, multiplying every member by $b-a$, and then adding $a$ to every member leads to $a\leq f(x)\leq b$. Thus, $f^{-1}[I]\sub \abClosed$. If $a\leq y\leq b$, then by similar computations, the number $c=\frac{b_0-a_0}{b-a}(y-a)+a_0$ has the properties $a_0\leq c\leq b_0$ and $y=f(c)$. This completes a proof that $f^{-1}[I]= \abClosed$.
\end{proof}

\begin{lemma}[Pasting Lemma] Suppose $F_1$ and $F_2$ are sets both closed relative to some set $I$. If $f_1$ is continuous on $F_1$ and $f_2$ is continuous on $F_2$ and if, for all $x\in F_1\cap F_2$, we have $f_1(x)=f_2(x)$, then the function $g: F_1\cup F_2\into\R$ defined by
\begin{equation}
    g(x)=\begin{cases} f_1(x), & x\in F_1,\\
        f_2(x), & x\in F_2,
    \end{cases}\nonumber
\end{equation}
is continuous on $F_1\cup F_2$.
\end{lemma}
\begin{proof} Let $F$ be a closed set. Since $f_1$ is continuous on $F_1$, and $f_2$ is continuous on $F_2$, the set $f_1^{-1}[F]$ is closed relative to $F_1$, and $f_2^{-1}[F]$ is closed relative to $F_2$. Thus, there exist closed sets $F_3$ and $F_4$ such that $f_1^{-1}[F]=F_3\cap F_1$ and $f_2^{-1}[F]=F_4\cap F_2$. Since $F_1$ and $F_2$ are both closed relative to $I$, there exist closed sets $F_5$ and $F_6$ such that $F_1=F_5\cap I$ and $F_2=F_6\cap I$, and we further have $f_1^{-1}[F]=F_3\cap F_5\cap I$ and $f_2^{-1}[F]=F_4\cap F_6\cap I$. The set $F_7:=F_3\cap F_5\cup F_4\cap F_6$ is a closed set, and we further have, $f_1^{-1}[F]\cup f_2^{-1}[F]=(F_3\cap F_5\cup F_4\cap F_6)\cap I=F_7\cap I$. Since $F_1,F_2\sub I$, we have $F_1\cup F_2\sub I$. Since $f_1^{-1}[F]\sub F_1\sub F_1\cup F_2$ and $f_2^{-1}[F]\sub F_2\sub F_1\cup F_2$, if we take the intersection with $F_1\cup F_2$ of both sides of $f_1^{-1}[F]\cup f_2^{-1}[F]=F_7\cap I$, we obtain $f_1^{-1}[F]\cup f_2^{-1}[F]=F_7\cap (F_1\cup F_2)$. At this point, we have proven that $f_1^{-1}[F]\cup f_2^{-1}[F]$ is closed relative to $F_1\cup F_2$.

To complete the proof, we show $g^{-1}[F]=f_1^{-1}[F]\cup f_2^{-1}[F]$. The condition $x\in g^{-1}[F]$ is equivalent to $g(x)\in F$, but by the definition of $g$, the number $g(x)$ is equal to either $f_1(x)$ or $f_2(x)$. Thus, $g(x)\in F$ implies $f_1(x)\in F$ or $f_2(x)\in F$. Conversely, if one of $f_1(x)\in F$ or $f_2(x)\in F$ is true, then\linebreak $g(x)=f_1(x)\in F$ if $x\in F_1$, or $g(x)=f_2(x)\in F$ if $x\in F_2$. Thus,  $g(x)\in F$ if and only if $f_1(x)\in F$ or $f_2(x)\in F$. This is equivalent to $x\in g^{-1}[F]$ if and only if $x\in f_1^{-1}[F]$ or $x\in f_2^{-1}[F]$. Furthermore,   $x\in g^{-1}[F]$ if and only if $x\in f_1^{-1}[F]\cup f_2^{-1}[F]$. Therefore, $g^{-1}[F]=f_1^{-1}[F]\cup f_2^{-1}[F]$, and this completes the proof that $g$ is continuous on $I=F_1\cup F_2$.
\end{proof}

\begin{proposition}\label{conNecPropII} If $I$ is connected, then any closed and bounded interval contained in $I$ is connected.
\end{proposition}
\begin{proof} Suppose $I$ is connected, but, tending towards a contradiction, that there exists an interval $\abClosed\sub I$ that is not connected. As a consequence, there exist  nonempty, disjoint subsets $A$ and $B$ of $\abClosed$, both open relative to $\abClosed$, with $\abClosed=A\cup B$. Consequently, $A=\abClosed\setdiff B$ and $B=\abClosed\setdiff A$. By Corollary~\ref{closedRelCor}, $A$ and $B$ are closed relative to $\abClosed$. By Example~\ref{contEx}, the constant function $A\into\R$ defined by $x\mapsto 0$ is continuous on $A$, and the constant function $B\into\R$ given by $x\mapsto 1$ is continuous on $B$. Since $A\cap B=\emptyset$, at this point, we have satisfied all the hypotheses of the Pasting Lemma to deduce that the function $f:\abClosed\into\R$ defined by 
\begin{equation}
    f(x):=\begin{cases}
        0, & x\in A,\\
        1, & x\in B,
    \end{cases}\nonumber
\end{equation}
is continuous on $A\cup B=\abClosed$. 

Because the intervals $(-\infty,a]$, $[b,\infty)$ and $\abClosed$ are closed, the sets $L_a:=(-\infty,a]\cap I$,  $R_b:=I\cap [b,\infty)$ and $\abClosed=\abClosed\cap I$ (since $\abClosed\sub I$) are closed relative to $I$. 

Since, $R=(-\infty,\infty)=(-\infty,a]\cup\abClosed\cup[b,\infty)$, taking the intersection of both sides with $I$ gives us $I\cap\R=L_a\cup (\abClosed\cap I)\cup R_b$. Since $\abClosed\sub I\sub\R$, we further have $I=L_a\cup\abClosed\cup R_b$. We now define $g: I\into\R$ by
\begin{equation}
    g(x):=\begin{cases}
        f(a), & x\in L_a,\\
        f(x), & x\in\abClosed,\\
        f(b), & x\in  R_b.
    \end{cases}\nonumber
\end{equation}

We cannot have $a=b$ which will make $\abClosed$ a singleton, so by Example~\ref{singConEx}, $\abClosed$ is connected.$\lightning$ Hence, $a<b$. This means that $L_a\cap R_b=\emptyset$. For the two other intersections, we have $\{a\}=L_a\cap \abClosed$ and $\{b\}=\abClosed\cap R_b$, but for each of these intersections, the values of $g$ agree: whether we view $a$ as $a\in L_a$ or $a\in\abClosed$, still, $g(a)$ is $f(a)$ in both cases. Similarly, any of  $b\in L_a$ or $b\in\abClosed$ results to $g(b)=f(b)$. By the Pasting Lemma, $g$ is continuous on $L_a\cup\abClosed\cup R_b=I$. Since $I$ is connected, by Proposition~\ref{contImProp}, $g[I]=\{0,1\}$ is connected, contradicting Example~\ref{disconEx}. Therefore, any $\abClosed\sub I$ is connected.
\end{proof}

\begin{lemma}\label{BVTtoMCPLem} Let $\aseq$ be a monotonically increasing sequence that does not converge, and let $\cseq$ be an arbitrary sequence. Define $I_1:=\lpar-\infty,a_1\rbrak$, and for each integer $n\geq 2$, let $I_n:=\lbrak a_{n-1},a_n\rbrak$. Let $A:=\bigcup_{n=1}^\infty I_n$. The function $f:\R\into\R$ defined by
\begin{eqnarray}
    f(x)=\begin{cases} c_1, & x \in I_1, \\
    c_{n-1} + \frac{x-a_{n-1}}{a_n-a_{n-1}}(c_n - c_{n-1}), & x \in I_n, n \geq 2, \\ 0, & x \notin A, \end{cases}\nonumber
\end{eqnarray}
is continuous on $\R$, with $\{c_n\  :\  \in\N\}\sub f[\{a_n\  :\  n\in\N\}]$.
\end{lemma}
\begin{proof} By Proposition~\ref{monostrictsubProp}, we may assume, WLOG, that $\anseq$ is strictly increasing. More precisely, if the original arbitrary monotonically increasing sequence (that does not converge) is not strictly increasing, we can choose $\anseq$ to be a strictly increasing subsequence, which also does not converge.

By Example~\ref{contEx}, the constant function $x\mapsto c_1$ is continuous on $\R$, and in particular, on $I_1=(-\infty,a_1]$, so $f$ is continuous on $I_1=\bigcup_{n=1}^1 I_n$. Suppose that for some $k\in\N$, the function $f$ is continuous on $\bigcup_{n=1}^kI_n$, which is a union of closed intervals and is hence a closed set. Also by Example~\ref{contEx}, the function $x\mapsto c_k + \frac{x - a_k}{a_{k+1} - a_k}(c_{k+1} - c_k)$ is continuous on $\R$, because it is of the form $x\mapsto mx+b$, where\linebreak $m = \frac{c_{k+1} - c_k}{a_{k+1} - a_k}$ and $b = \frac{a_{k+1} c_k - a_k c_{k+1}}{a_{k+1} - a_k}$. In particular, the restriction of this function, which is $f$, on the closed set $[a_{k},a_{k+1}]=I_{k+1}$, is continuous. At the intersection of $\bigcup_{n=1}^kI_n$ and $I_{k+1}$, which is at $a_k$, the value of $f:\bigcup_{n=1}^kI_n\into\R$ is either $0$ [if $k=1$] or $c_k$ [if $k\geq 2$], which is equal to the value of $f:I_{k+1}\into\R$ at $a_k$. By the Pasting Lemma, $f$ is continuous on $\bigcup_{n=1}^{k+1}I_n$, and by induction, we have proven that for any $k\in\N$, the function $f$ is continuous on $\bigcup_{n=1}^kI_n$. If there exists $x\in A=\bigcup_{n=1}^\infty I_n$ such that $f$ is not continuous on $x$, then there exists $N\in\N$ such that $x\in I_N\sub\bigcup_{n=1}^NI_n$, so $f$ is not continuous on $\bigcup_{n=1}^NI_n$.$\lightning$ Hence, $f$ is continuous on $A$. 

We now prove the continuity of $f$ on $A^c$. Let $\varepsilon>0$, and let $x\in A^c$. Since $A$ is a union of closed intervals, it is a closed set, and $A^c$ is an open set, so there exists $\delta>0$ such that $(x-\delta,x+\delta)\sub A^c$. If $|x-t|<\delta$, then $t\in (x-\delta,x+\delta)\sub A^c$, and by the definition of $f$, the condition $t\in A^c$ implies $f(t)=0$. But since $x\in A^c$, we also have $f(x)=0$. Thus, $|f(t)-f(x)|=|0-0|=0<\varepsilon$. Therefore, $f$ is continuous at the arbitrary $x\in A^c$. This completes the proof that $f$ is continuous on $\R$. 

For any $n\in\N$, we have $a_n\in I_n$, and so $f(a_n)=c_{n-1}+1\cdot(c_n-c_{n-1})=c_n$, and this proves that $\{c_n\  :\  \in\N\}\sub f[\{a_n\  :\  n\in\N\}]$.
\end{proof}

\begin{proposition}\label{CutProp} For a nonempty proper subset of $A$ of $\R$, the following are equivalent.
\begin{enumerate}\item\label{CutPairDef} For each $a\in A$ and each $b\in{A^c}$, $a<b$.
\item\label{CutClosedDef} For each $x\in\R$ and each $a\in A$, if $x<a$, then $x\in A$.
\item\label{CutOpenDef} $\bigcup_{a\in A}\lpar-\infty,a\rpar\sub A$.
\item\label{CutOpenDefc} $\bigcup_{b\in {A^c}}\lpar b,\infty\rpar\sub{A^c}$.
\end{enumerate}
[If $A$ satisfies one of the conditions \ref{CutPairDef}--\ref{CutOpenDefc}, we say that $A$ is a \emph{cut}\footnote{The traditional definition of cut involves a pair of subsets, where, for instance, $A$ in \ref{CutPairDef} is paired with ${A^c}$. However, we find the definition \ref{CutClosedDef} from \cite{mon01,mon08} as more expedient. The statement \ref{CutClosedDef} is what we interpret from \cite[p. 195]{mon01} as the meaning of ``downward closed,'' which should not be confused with the previously defined closed subsets of $\R$.} of $\R$.]
\end{proposition}
\begin{proof}\ref{CutPairDef} $\implies$ \ref{CutClosedDef} Let $x\in\R$, and let $a\in A$ such that $x<a$. Tending towards a contradiction, suppose $x\notin A$. By \ref{CutPairDef}, we have $a<x$, contradicting the Trichotomy Law in $\R$. Therefore, \ref{CutClosedDef} is true.\\

\noindent\ref{CutClosedDef} $\implies$ \ref{CutOpenDef} Given $x\in\bigcup_{a\in A}\lpar-\infty,a\rpar$, there exists $\alpha\in A$ such that $x\in\lpar-\infty,\alpha\rpar$, or that $x<\alpha$. By \ref{CutClosedDef}, $x\in A$. Therefore, $\bigcup_{a\in A}\lpar-\infty,a\rpar\sub A$.\\

\noindent\ref{CutOpenDef} $\implies$ \ref{CutOpenDefc} Given $x\in\bigcup_{b\in {A^c}}\lpar b,\infty\rpar$, there exists $\beta\in{A^c}$ such that $x\in\lpar\beta,\infty\rpar$, or that $\beta<x$. Tending towards a contradiction, suppose $x\notin {A^c}$. Equivalently, $x\in A$. By \ref{CutOpenDef}, $\lpar-\infty,x\rpar\sub A$, so from $\beta<x$, we obtain $\beta\in A$.$\lightning$  Hence, $x\in {A^c}$, and so, $\bigcup_{b\in {A^c}}\lpar b,\infty\rpar\sub{A^c}$.\\

\noindent\ref{CutOpenDefc} $\implies$ \ref{CutPairDef} Let $a\in A$ and $b\in{A^c}$. Tending towards a contradiction, suppose $a\geq b$. Because $A$ and ${A^c}$ are disjoint, we cannot have $a=b$. Henceforth, $a>b$, which further implies $a\in\lpar b,\infty\rpar$, where the interval, according to \ref{CutOpenDefc}, is a subset of ${A^c}$. Thus, $b\in A$.$\lightning$ Therefore, $a<b$.
\end{proof}

\begin{proposition}\label{GapProp} Given a cut $A$ of $\R$, the following are equivalent.
\begin{enumerate}\item\label{GapPairDef} For any $c\in\R$, there exist $a\in A$ and $b\in A^c$ such that $c\notin\abClosed$.
\item\label{GapOpenDef} $\bigcup_{a\in A}\lpar-\infty,a\rpar= A$ and $\bigcup_{b\in {A^c}}\lpar b,\infty\rpar={A^c}$.
\end{enumerate}
[If $A$ satisfies one of the conditions \ref{GapPairDef}--\ref{GapOpenDef}, then we say that $A$ is a \emph{gap} in $\R$. If $A$ is not a gap, then the number $c$ in the negation of \ref{GapPairDef} is called a \emph{cut point} of $A$.]
\end{proposition}
\begin{proof}\ref{GapPairDef} $\implies$ \ref{GapOpenDef} We prove this by contraposition. The negation of \ref{GapOpenDef} is equivalent to the assertion that one of the two set equations in \ref{GapOpenDef} is false. In view of Proposition~\ref{CutProp}\ref{CutOpenDef}--\ref{CutOpenDefc}, one of the set inclusions $A\sub\bigcup_{a\in A}\lpar-\infty,a\rpar$ and $A^c\sub\bigcup_{b\in {A^c}}\lpar b,\infty\rpar$ is false. 

If there exists $c\in A$ such that $c\notin\bigcup_{a\in A}\lpar-\infty,a\rpar$, then for any $a\in A$, we have $c\notin\lpar-\infty,a\rpar$, or that $c\in\lbrak a,\infty\rpar$. Thus, $a\leq c$. Given $b\in A^c$, since $c\in A$, by the definition of cut, $c< b$. At this point, we have $a\leq c<b$, so $c\in\abClosed$.

If there exists $c\in A^c$ such that $c\notin\bigcup_{b\in A^c}\lpar b,\infty\rpar$, then for any $b\in A^c$, we have $c\notin\lpar b,\infty\rpar$, which implies $c\in\lpar-\infty,b\rbrak$, and we obtain $c\leq b$. From $b\in A^c$ and the definition of gap, if $a\in A$, then $a<c\leq b$. Hence, $c\in\abClosed$.

In both cases, we obtained the negation of \ref{GapPairDef}.\\

\noindent\ref{GapOpenDef} $\implies$ \ref{GapPairDef} Let $c \in \R=A\cup A^c$.

If $c\in A$, then by the first set equation in \ref{GapOpenDef}, there exists $\alpha\in A$ such that $c\in(-\infty,\alpha)$, or that $c<\alpha$. Since $A$ is a proper subset of $\R$, the complement $A^c$ is nonempty, so there exists $\beta\in A^c$. By the definition of cut, $\alpha<\beta$. At this point, we have $c<\alpha<\beta$, so $c\notin\lbrak\alpha,\beta\rbrak$.

If $c\in A^c$, then the second set equation in \ref{GapOpenDef} implies that there exists $\beta\in A$ such that $c\in(\beta,\infty)$. Since $A$ is nonempty, there exists $\alpha\in A$, and by the definition of cut, $\alpha<\beta<c$, so $c\notin\lbrak\alpha,\beta\rbrak$.
\end{proof}

\subsection{Filter Bases, Limits and Differentiable Functions}\label{ThirdCirclePrelims}

A nonempty collection $\Base$ of nonempty subsets of $\R$ is said to be a \emph{filter base\footnote{This is a standard topological notion, but is rarely emphasized in mainstream topology or real analysis despite its potency. We used \cite[p.~14]{dix84} and \cite[pp.~121--122]{sie92}.}} if for every $A,B\in\Base$, there exists $C\in\Base$ such that $C\sub A\cap B$. If $c\in\R$, then we say that a filter base \emph{$\Base$ approaches $c$}, or in symbols $\Base\into c$, if for each $\delta>0$, there exists $I\in\Base$ such that $I\sub(c-\delta,c+\delta)$. 


Given $L\in\R$, $X\sub\R$, a function $f:X\into\R$, and a filter base $\Base$ that approaches $c\in\R$, we say that \emph{$f$ approaches $L$ along $\Base$}, or that \emph{$f$ approaches $L$ as $\Base$ approaches $c$}, or in symbols, \emph{$f\into L$ as $\Base\into c$}, if for each $\varepsilon>0$, there exists $I\in\Base$ such that $f[I]\sub\lpar L-\varepsilon,L+\varepsilon\rpar$.

\begin{proposition}\label{BaseLimUniqueProp} As a consequence of the Trichotomy Law, if $f\into L$ as $\Base\into c$, then $L$ is unique. [In such a case, we define $\BcLim f:=L$, which is called the \emph{limit of $f$ along $\Base$}, or the \emph{limit of $f$ as $\Base\into c$}. If there exists $L\in\R$ such that $L=\BcLim f$, then ``\emph{the  $\BcLim f$ exists}.'']
\end{proposition}
\begin{proof} Tending towards a contradiction, suppose there exists a real number $K$, distinct from $L$, such that $f$ approaches both $K$ and $L$ along the filter base $\Base$. WLOG, we suppose further that $K<L$. Thus, $\varepsilon:=\frac{L-K}{2}>0$. Since $f$ approaches $K$ and $L$ along $\Base$, there exist $I\in\Base$ and $J\in\Base$ such that\linebreak $f[I]\sub\lpar L-\varepsilon,L+\varepsilon\rpar$, and $f[J]\sub\lpar K-\varepsilon,K+\varepsilon\rpar$. Since $\Base$ is a filter base, there exists $X\in\Base$ such that $X\sub I\cap J$, and so $f[X]\sub f[I\cap J]\sub f[I]$, and $f[X]\sub f[I\cap J]\sub f[J]$. At this point, we have $f[X]\sub\lpar L-\varepsilon,L+\varepsilon\rpar$ and $f[X]\sub\lpar K-\varepsilon,K+\varepsilon\rpar$, which further imply $$f[X]\sub \lpar L-\varepsilon,L+\varepsilon\rpar\cap \lpar K-\varepsilon,K+\varepsilon\rpar.$$ Since $\Base$ is a filter base, $X\in\Base$ implies $X$ is nonempty and hence so is $f[X]$. Thus, there exists\linebreak $y\in f[X]\sub \lpar L-\varepsilon,L+\varepsilon\rpar\cap \lpar K-\varepsilon,K+\varepsilon\rpar$. Consequently, $L-\varepsilon<y<K+\varepsilon$, to which we substitute $\varepsilon=\frac{L-K}{2}$, and we obtain $0<y-\frac{L+K}{2}<0$, which contradicts the Trichotomy Law. Therefore, for any real number $K$ distinct from $L$, the function $f$ cannot approach $K$ along $\Base$. This completes the proof.
\end{proof}

\begin{proposition}\label{AlgLimProp} Let $X\sub\R$, let $c\in\R$, and let $\Base$ be a filter base that approaches $c$.
\begin{enumerate}\item\label{AlgConst} Given $C\in\R$, for the constant function $F:X\into\R$ defined by $F:x\mapsto C$, we have $\BcLim F=C$.
\item\label{AlgId} For the identity function $\id:X\into\R$ defined by $\id: x\mapsto x$, we have $\BcLim \id=c$. As a consequence, if $\Base\into c$, then $c$ is unique.
\item\label{AlgIFF0} Given $L\in\R$, and given a function $f:X\into\R$, if the function $X\into\R$ defined by $x\mapsto f(x)-L$ is denoted by $f-L$, then $\BcLim f=L$ if and only if $\BcLim(f-L)=0$.
\item\label{AlgProd0} Given functions $f,g:X\into\R$, if $\BcLim f=0=\BcLim g$, then $\BcLim fg=0$.
\item\label{AlgLinear} Given $C\in\R$ and functions $f,g:X\into\R$, if $\BcLim f$ and $\BcLim g$ exist, then
\begin{eqnarray}
    \BcLim (f+g) &=& \BcLim f+\BcLim g,\nonumber\\
    \BcLim Cf &=& C\BcLim f.\nonumber
\end{eqnarray}
\item\label{AlgProd} Given functions $f,g:X\into\R$, if $\BcLim f$ and $\BcLim g$ exist, then
\begin{eqnarray}
    \BcLim fg &=& \BcLim f\  \BcLim g.\nonumber
\end{eqnarray}
\end{enumerate}
\end{proposition}
\begin{proof}\ref{AlgConst} Let $\varepsilon>0$. Since the filter base $\Base$ is nonempty, there exists $I\in\Base$, and since $F$ is a constant function with value $C$, the set $F[I]$ is either the singleton $\{C\}$ or the empty set, both of which are subsets of $(C-\varepsilon,C+\varepsilon)$. Therefore,  $\BcLim F=C$.\\

\noindent\ref{AlgId} Let $\varepsilon>0$, and define $\delta:=\varepsilon$, so $\delta>0$. Since $\Base\into c$, there exists $I\in\Base$ such that $I\sub(c-\delta,c+\delta)$, where the left-hand side, because $\id$ is the identity function, is equal to $\id [I]$. Therefore, $\BcLim \id=c$. This equation is true if and only if for each $\delta>0$ there exists $I\in\Base$ such that $I=\id[I]\sub(c-\delta,c+\delta)$, if and only if $\Base\into c$. Thus, if $\Base\into c$ and $\Base\into k$, then both $c$ and $k$ are equal to $\BcLim\id$, which, by Proposition~\ref{BaseLimUniqueProp}, is unique. Therefore, $c=k$.\\

\noindent\ref{AlgIFF0} With reference to our definition of limit, proving the equivalence of $f[I]\sub\lpar L-\varepsilon,L+\varepsilon\rpar$ and\linebreak $(f-L)[I]\sub(-\varepsilon,\varepsilon)$ shall suffice. If $y\in (f-L)[I]$, then there exists $x\in I$ such that $y=f(x)-L$, so $y+L=f(x)\in f[I]\sub(L-\varepsilon,L+\varepsilon)$, which implies $|y+L-L|<\varepsilon$, or that $|y|<\varepsilon$, which is equivalent to $y\in\lpar -\varepsilon,\varepsilon\rpar$. This proves that $(f-L)[I]\sub(-\varepsilon,\varepsilon)$. Conversely, if $y\in f[I]$, then there exists $x\in I$ such that $y=f(x)$, so $y-L=f(x)-L\in F[I]\sub (-\varepsilon,\varepsilon)$, or equivalently, $|y-L+0|<\varepsilon$, or that $|y-L|<\varepsilon$, which is equivalent to $y\in \lpar L-\varepsilon,L+\varepsilon\rpar$. Therefore, $f[I]\sub\lpar L-\varepsilon,L+\varepsilon\rpar$.\\

\noindent\ref{AlgProd0} Let $\varepsilon>0$, and let $\eta:=\min\left\{\txthalf\varepsilon,\txthalf\right\}>0$. Thus, $\eta<\varepsilon$ and $\eta<1$. Furthermore, $(-\eta,\eta)\sub(-\varepsilon,\varepsilon)$. Using basic properties of inequalities, the condition $0<\eta<1$ implies that $(-\eta,\eta)$ is ``closed under multiplication.'' That is, if $x,y\in(-\eta,\eta)$, then $xy\in(-\eta,\eta)$. If $\BcLim f=0=\BcLim g$, then there exist $I,J\in\Base$ such that $f[I]\sub(-\eta,\eta)$ and $g[J]\sub(-\eta,\eta)$. Since $\Base$ is a filter base, there exists $K\in\Base$ such that  $K\sub I\cap J$. Given $y\in (fg)[K]$, there exists $x\in K\sub I\cap J\sub I$ (and also, $x\in I\cap J\sub J$) such that $y=f(x)g(x)$. But $x\in I$ and $x\in J$ imply $f(x)\in f[I]\sub(-\eta,\eta)$ and $g(x)\in g[J]\sub(-\eta,\eta)$. Thus, $y=f(x)g(x)\in(-\eta,\eta)$. Hence, $(fg)[I\cap J]\sub(-\eta,\eta)\sub(-\varepsilon,\varepsilon)$, and therefore, $\xcLim[f(x)g(x)]=0$.\\

\noindent\ref{AlgLinear} Given $\varepsilon>0$, we have $\txthalf\varepsilon>0$, so if $L:=\BcLim f$ and $K:=\BcLim g$, then there exist $I,J\in\Base$ such that $f[I]\sub \lpar L-\txthalf\varepsilon,L+\txthalf\varepsilon\rpar$ and $g[J]\sub \lpar K-\txthalf\varepsilon,K+\txthalf\varepsilon\rpar$. Since $\Base$ is a filter base, there exists $T\in\Base$ such that $T\sub I\cap J$. Given $y\in(f+g)[T]$, there exists $x\in T$ such that $y=f(x)+g(x)$. But $x\in T\sub I\cap J\sub I$ [and also $x\in I\cap J\sub J$] imply $f(x)\in f[I]\sub\lpar L-\txthalf\varepsilon,L+\txthalf\varepsilon\rpar $ and $g(x)\in g[J]\sub \lpar K-\txthalf\varepsilon,K+\txthalf\varepsilon\rpar$, which respectively imply $|f(x)-L|<\txthalf\varepsilon$ and $|g(x)-K|<\txthalf\varepsilon$. By the Triangle Inequality, 
\begin{eqnarray}
    |y-(L+K)| &=& |f(x)+g(x)-(L+K)|=|[f(x)-L]+[g(x)-K]|,\nonumber\\
    &\leq& |f(x)-L|+|g(x)-K| < \txthalf\varepsilon+\txthalf\varepsilon=\varepsilon,\nonumber
\end{eqnarray}
but $|y-(L+K)|<\varepsilon$ is equivalent to $y\in\lpar L+K-\varepsilon,L+K\varepsilon\rpar$. 

We have proven that $(f+g)[T]\sub\lpar L+K-\varepsilon,L+K\varepsilon\rpar$. Therefore, $\BcLim(f+g)=L+K$.

Let $\varepsilon>0$. If $C=0$, then the function $Cf$ is the zero constant function, and by \ref{AlgConst}, we are done. Supposing henceforth that $C\neq 0$, we have $\frac{\varepsilon}{|C|}>0$. If $L:=\BcLim f$, then there exists $I\in\Base$ such that $f[I]\sub\lpar L-\frac{\varepsilon}{|C|},L+\frac{\varepsilon}{|C|} \rpar$. Given $y\in(Cf)[I]$, there exists $x\in I$ such that $y=Cf(x)$, so $\frac{y}{C}=f(x)\in f[I]\sub\lpar L-\frac{\varepsilon}{|C|},L+\frac{\varepsilon}{|C|} \rpar$, which implies $\left|\frac{y}{C}-L\right|<\frac{\varepsilon}{|C|}$, both sides of which, we multiply by the positive number $|C|$ to obtain $|y-CL|<\varepsilon$. Equivalently, $y\in\lpar CL-\varepsilon,CL+\varepsilon\rpar$, and we have proven $(Cf)[I]\sub\lpar CL-\varepsilon,CL+\varepsilon\rpar$. Therefore, $\BcLim Cf=CL$.\\

\noindent\ref{AlgProd} If $L:=\BcLim f$ and $K:=\BcLim g$, then denote the functions $x\mapsto f(x)-L$, $x\mapsto g(x)-K$ and $x\mapsto f(x)g(x)-LK$ simply by $f-L$, $g-K$ and $fg-LK$, respectively. For any $x\in X$, we have the identity $f(x)g(x)-LK=[f(x)-L][g(x)-K]+K[f(x)-L]+L[g(x)-K]$. As an equation of functions, $fg-LK=(f-L)(g-K)+K(f-L)+L(g-K)$. The limit as $\Base\into c$ of the right-hand side may be obtained from previous parts of the proposition as shall be explained, and because the limit of the right-hand side exists, so does the limit of the left-hand side. According to the uniqueness theorem, Proposition~\ref{BaseLimUniqueProp}, we may indeed take the limit of both sides of the function equation. We now proceed with the details of the computation. From \ref{AlgIFF0}, $L=\BcLim f$ and $K=\BcLim g$ imply $\BcLim (f-L)=0=\BcLim(g-K)$. By \ref{AlgProd0}, $\BcLim(f-L)(g-K)=0$, and by \ref{AlgLinear}, $\BcLim(fg-LK)=0+K\cdot 0+L\cdot 0=0$, so using \ref{AlgIFF0} again, $\BcLim fg=LK$.
\end{proof}

\begin{proposition} Given a (nonempty) $X\sub\R$, and given $c\in X$, the sets
\begin{eqnarray}
\lFilter(X,c) &:=& \{X\cap\lpar c-\delta,c\rpar\  :\  \delta>0\},\nonumber\\
\rFilter(X,c) &:=& \{X\cap\lpar c,c+\delta\rpar\  :\  \delta>0\},\nonumber\\
\Filter(X,c) &:=& \{X\cap\lbrak\lpar c-\delta,c\rpar\cup \lpar c,c+\delta\rpar\rbrak\  :\  \delta>0\},\nonumber
\end{eqnarray}
are filter bases that approach $c$.
\end{proposition}
\begin{proof} Let $\Base$ be one of $\lFilter(X,c)$, $\rFilter(X,c)$ or $\Filter(X,c)$. Given $I\in\Base$, let $\delta_I$ be the positive real number such that $I=X\cap (c-\delta_I,c)$ [if $\Base=\lFilter(X,c)$; such that $I=X\cap (c, c+\delta_I)$ if $\Base=\rFilter(X,c)$; or such that $I=X\cap\lbrak(c-\delta_I)\cup(c+\delta_I)\rbrak$ if $\Base=\Filter(X,c)$]. We first prove that 
\begin{enumerate}
    \item[\Star] given $I,J\in\Base$, if $\delta_I<\delta_J$, then $I\sub J$.
\end{enumerate}
Thus, $-\varepsilon<-\delta$, and we further have $c-\varepsilon<c-\delta<c$ for the case $\Base=\lFilter(X,c)$ [respectively, $c<c+\delta<c+\varepsilon$ for the case $\Base=\rFilter(X,c)$, while we need both systems of inequalities for the case for the case $\Base=\Filter(X,c)$]. These inequalities imply that $(c-\delta,c)\sub(c-\varepsilon,c)$ for the case $\Base=\lFilter(X,c)$ [respectively, $(c,c+\delta)\sub(c,c+\varepsilon)$ for the case $\Base=\rFilter(X,c)$, while for the case for the case $\Base=\Filter(X,c)$, we need both set inclusions], and taking the intersection of both sides with $X$, we obtain $I\sub J$. This proves \Star.

Let $I,J\in\Base$. If $\delta_I=\delta_J$, then $I=J$, which implies $I\sub J$. If $\delta_I$ and $\delta_J$ are distinct, we may assume, WLOG, that $\delta_I<\delta_J$, and by \Star, $I\sub J$. In any case, $I\sub J$, which implies $I=I\cap J$, which further implies $I\sub I\cap J$, where $I\in\Base$. Therefore, $\Base$ is a filter base.

Given $\delta>0$, the sets $\lpar c-\delta,c\rpar$, $\lpar c,c+\delta\rpar$ and $\lpar c-\delta,c\rpar\cup \lpar c,c+\delta\rpar$ are all subsets of $(c-\delta,c+\delta)$, and hence, so is the intersection of any of these sets with $X$. Therefore, $\Base\into c$.
\end{proof}

Given a function $f:X\into\R$ and given $c\in X$, we define
\begin{eqnarray}
    \xcLim f(x) &:=&\displaystyle\lim_{\Filter(X,c)\rightarrow c} f,\label{xclimDEF}\\
    \xcLimL f(x) &:=&\displaystyle\lim_{\lFilter(X,c)\rightarrow c} f,\nonumber\\
    \xcLimR f(x) &:=&\displaystyle\lim_{\rFilter(X,c)\rightarrow c} f.\nonumber
\end{eqnarray}
Any limit along the filter base $\Filter(X,c)$ [respectively, along $\lFilter(X,c)$, and along $\rFilter(X,c)$] shall be referred to as a limit \emph{as $x\into c$} [respectively, \emph{as $x\into c-$}, and \emph{as $x\into c+$}].

\begin{proposition}\label{twoSidedLimProp} The following are equivalent.
\begin{enumerate}
\item  $\xcLim f(x)$ exists.
\item  $\xcLimL f(x)$ and $\xcLimR f(x)$ exist and are equal.
\end{enumerate}
In particular, if one of the above statements is true, then $\xcLimL f(x)=\xcLim f(x)=\xcLimR f(x)$.
\end{proposition}
\begin{proof} Suppose $f$ is a function $X\into\R$, and let $c\in X$.

\noindent$(\!\!\implies\!\!)$ Let $L:=\xcLim f(x)$, and let $\varepsilon>0$. 

Since $f$ approaches $L$ along $\Filter_c$, there exists $I\in\Filter_c$ such that $f[I]\sub\lpar L-\varepsilon,L+\varepsilon\rpar$. From $I\in\Filter_c$, we find that there exists $\delta>0$ such that $I=X\cap\lbrak(c-\delta,c)\cup(c,c+\delta)\rbrak$. Consequently, $X\cap(c-\delta,c)\sub I$ is one of the sets in $\lFilter_c$, and $X\cap(c,c+\delta)\sub I$ is in $\rFilter_c$. Furthermore, 
\begin{eqnarray}
    f[X\cap(c-\delta,c)] &\sub& f[I]\sub\lpar L-\varepsilon,L+\varepsilon\rpar,\nonumber\\
    f[X\cap(c,c+\delta)] &\sub& f[I]\sub\lpar L-\varepsilon,L+\varepsilon\rpar.
\end{eqnarray} Therefore, $\xcLimL f(x)=L$ and $\xcLimR f(x)=L$.\\

\noindent $(\!\!\impliedby\!\!)$ Suppose $\xcLimL f(x)=L=\xcLimR f(x)$, and let $\varepsilon>0$. From $\xcLimL f(x)=L$, there exists $I\in\lFilter_c$ such that $f[I]\sub\lpar L-\varepsilon,L+\varepsilon\rpar$, while from $L=\xcLimR f(x)$, there exists $J\in\rFilter_c$ such that\linebreak $f[J]\sub\lpar L-\varepsilon,L+\varepsilon\rpar$. From $I\in\lFilter_c$ and $J\in\rFilter_c$, we find that there exist $\delta_1>0$ and $\delta_2>0$ such that $I=X\cap(c-\delta_1,c)$ and $J=X\cap(c,c+\delta_2)$. The number $\delta:=\min\{\delta_1,\delta_2\}$ is either one of two positive numbers and is hence positive, with the further properties $\delta\leq \delta_1$ and $\delta\leq \delta_2$, which imply\linebreak $c-\delta_1\leq c-\delta <c$ and $c<c+\delta\leq c+\delta_2$. 

These inequalities may be used to prove $(c-\delta,c)\sub (c-\delta_1,c)\cup(c,c+\delta_1)$ and\linebreak $(c,c+\delta)\sub (c-\delta_2,c)\cup(c,c+\delta_2)$. Taking the intersection of both sides (of both set inclusions) with $X$, and then taking unions, we find that $K:=X\cap\lbrak(c-\delta,c)\cup(c,c+\delta)\rbrak\sub I\cup J$ where $K\in\Filter_c$, and furthermore, $f[K]\sub f[I\cup J]=f[I]\cup f[J]\sub \lpar L-\varepsilon,L+\varepsilon\rpar$. Therefore, $\xcLim f(x)=L$.
\end{proof}

We have a special case, which is when $X=\abClosed$ or $X=\abOpen$. Here, we have the set equalities
\begin{eqnarray}
    \rFilter(\abClosed,a)=\Filter(\abClosed,a),&& \lFilter(\abClosed,b)=\Filter(\abClosed,b),\nonumber\\
        \rFilter(\abOpen,a)=\Filter(\abOpen,a),&& \lFilter(\abOpen,b)=\Filter(\abOpen,b),\nonumber
\end{eqnarray}
and so, 
\begin{flalign}
  && \xaLimR f(x)=\xaLim f(x),&\qquad \xbLimL f(x)=\xbLim f(x), &(\mbox{if }f:\abClosed\into\R,\mbox{ or }f:\abOpen\into\R).\label{LimShortCut} 
\end{flalign}

Given $X\sub\R$ and a real number $c$ that is both an element and an interior point of $X$, a function $f:X\into\R$ is said to be \emph{differentiable at $c$} if $\xcLim\frac{f(x)-f(c)}{x-c}$ exists. By Proposition~\ref{BaseLimUniqueProp}, the rule of assignment $f':c\mapsto\xcLim\frac{f(x)-f(c)}{x-c}$ is a function, which is called the \emph{derivative} of $f$. [Since $f$ is a function $X\into\R$, the implied domain of $f'$ is the subset of $X$ consisting of all $c$ at which $f$ is differentiable.] Given an expression for $f(x)$, if the corresponding expression for $f'(x)$ may also be known, then the \emph{differentiation operator $\frac{d}{dx}$ (with respect to the variable $x$)} is defined by the condition $\frac{d}{dx}f(x)=f'(x)$.


\begin{proposition}\label{AlgDiffProp}\begin{enumerate}\fixitem\label{AlgConti} A function $f$ is continuous at $c$ if and only if $\xcLim f(x)=f(c)$.
\item\label{DiffConst} Given a constant function with value $C$, we have $\frac{d}{dx}C=0$.
\item\label{DiffId} For the identity function $x\mapsto x$, we have $\frac{d}{dx}x=1$.
\item\label{DiffConti} If $f$ is differentiable at $c$, then $f$ is continuous at $c$.
\item\label{DiffLinear} Given a constant $C\in\R$, if $f$ and $g$ are differentiable (on some given set) then so are $f+g$ and $Cf$. Furthermore, $(f+g)'=f'+g'$ and $(Cf)'=C\cdot f'$.
\item\label{DiffProd} If $f$ and $g$ are differentiable (on some given set), then so is $fg$. Furthermore, $(fg)'=f'g+g'f$
\end{enumerate}
\end{proposition}
\begin{proof} \ref{AlgConti} $(\!\!\implies\!\!)$ Suppose $f$ is a function $X\into\R$. Given $\varepsilon>0$, there exists $\delta>0$ such that $|x-c|<\delta$ implies $|f(x)-f(c)|<\varepsilon$. The set $I:=X\cap\lbrak(c-\delta,c)\cup(c,c+\delta)\rbrak$ is in $\Filter_c$. Given $y\in f[I]$, there exists $x\in I\sub(c-\delta,c)\cup(c,c+\delta)\sub (c-\delta,c+\delta)$, which implies $|x-c|<\delta$, such that $y=f(x)$. By assumption, $|f(x)-f(c)|<\varepsilon$, or equivalently, $y=f(x)\in(f(c)-\varepsilon,f(c)+\varepsilon)$, and this proves $f[I]\sub(f(c)-\varepsilon,f(c)+\varepsilon)$. Therefore, $\xcLim f(x)=f(c)$.\\

\noindent \ref{AlgConti} $(\!\!\!\impliedby\!\!\!)$ Again, suppose that $f$ is a function $X\into\R$. Let $\varepsilon>0$. If $\xcLim f(x)=f(c)$, then there exists $I\in\Filter_c$ such that $f[I]\sub(f(c)-\varepsilon,f(c)+\varepsilon)$. Since $I\in\Filter_c$, there exists $\delta>0$ such that\linebreak $I=X\cap\lbrak(c-\delta,c)\cup(c,c+\delta)\rbrak$. 

If $x\in X$ with $|x-c|<\delta$, then $x\in(c-\delta,c+\delta)=(c-\delta,c)\cup\{c\}\cup(c,c+\delta)=I\cup\{c\}$. Thus, $f(x)\in f[I]\cup\{f(c)\}\sub (f(c)-\varepsilon,f(c)+\varepsilon)\cup\{f(c)\}$, where the right-hand side, because\linebreak $f(c)\in(f(c)-\varepsilon,f(c)+\varepsilon)$, can be further reduced into $(f(c)-\varepsilon,f(c)+\varepsilon)$. 

That is, $f(x)\in (f(c)-\varepsilon,f(c)+\varepsilon)$, or equivalently, $|f(x)-f(c)|<\varepsilon$. Therefore, $f$ is continuous at $c$.\\

\noindent\ref{DiffConst} and \ref{DiffId} By Example~\ref{contEx}, the constant functions $x\mapsto 0$ and $x\mapsto 1$ are continuous at any $c\in\R$. [Thus, for such $c$, any domain $X\sub\R$ for the said constant functions may be chosen according to the requirements in our definition of differentiability, which is that $c$ is both an element and an interior point of $X$, like when $X=(c-1,c+1)$.] By \ref{AlgConti}, $\xcLim 0=0$ and $\xcLim 1=1$, where the right-hand sides are equal to $\xcLim\frac{C-C}{x-c}$ and $\xcLim\frac{x-c}{x-c}$, respectively. Therefore, $x\mapsto 0$ and $x\mapsto 1$ are differentiable at $c$, where the derivatives are $c\mapsto 0$ and $c\mapsto 1$, respectively. Using our notation for the differentiation operator, $\frac{d}{dx}C=0$ and $\frac{d}{dx}x=1$.\\

\noindent\ref{DiffConti} Suppose $f:X\into\R$ is differentiable at $c\in X$. By Example~\ref{contEx}, the function $x\mapsto x-c$ is continuous at $c$, so by \ref{AlgConti} $\xcLim(x-c)=c-c=0$. We use Proposition~\ref{AlgLimProp}\ref{AlgProd} [with the filter base given in \eqref{xclimDEF}] on the trivial identity $f(x)-f(c)=(x-c)\frac{f(x)-f(c)}{x-c}$, to obtain:

\noindent$\xcLim[f(x)-f(c)]=\xcLim(x-c)\  \xcLim\frac{f(x)-f(c)}{x-c}=0\cdot f'(c)=0$. By Proposition~\ref{AlgLimProp}\ref{AlgIFF0}, $\xcLim f(x)=f(c)$, and by part \ref{AlgConti} of this proposition, $f$ is continuous at $c$.\\

\noindent\ref{DiffLinear} Evaluate $\xcLim\lbrak\frac{f(x)-f(c)}{x-c}+\frac{g(x)-g(c)}{x-c}\rbrak$ and $\xcLim\frac{Cf(x)-Cf(c)}{x-c}$ using Proposition~\ref{AlgLimProp}\ref{AlgLinear} [with the filter base given in \eqref{xclimDEF}].\\

\noindent\ref{DiffProd} Suppose $f,g:X\into\R$ are differentiable at $c\in X$. The limit as $x\rightarrow c$ of the right-hand side of the identity
\begin{eqnarray}
    \frac{f(x)g(x)-f(c)g(c)}{x-c} = g(x)\frac{f(x)-f(c)}{x-c}+f(c)\frac{g(x)-g(c)}{x-c},\nonumber
\end{eqnarray}
may be computed from Proposition~\ref{AlgLimProp}\ref{AlgLinear}--\ref{AlgProd} [with the filter base given in \eqref{xclimDEF}] because of the following reasons: the limits $\xcLim \frac{f(x)-f(c)}{x-c}$ and $\xcLim \frac{g(x)-g(c)}{x-c}$ exist because of the differentiability of $f$ and $g$ at $c$; the number $f(c)$ works as a ``constant'' in using Proposition~\ref{AlgLimProp}\ref{AlgLinear}, while $\xcLim g(x)=g(c)$ is by parts \ref{AlgConti} and \ref{DiffConti} of this proposition. The result is $\xcLim\frac{f(x)g(x)-f(c)g(c)}{x-c}=g(c)f'(c)+f(c)g'(c)$. That is, the derivative of $fg$ sends the arbitrary $c\in\R$ to $g(c)f'(c)+f(c)g'(c)=f(c)g'(c)+f'(c)g(c)$. Therefore, $(fg)'=fg'+f'g$.
\end{proof}

If $f'$ is differentiable (on some subset of $\R$), then the derivative of $f'$, or the \emph{second-order derivative} of $f$, is $f''$. We also use the notation $f^{(0)}:=f$, $f^{(1)}:=f'$ and $f^{(2)}:=f''$. If, for some nonnegative integer $n$, the function $f^{(n)}$ is differentiable (on the appropriate subset of $\R$), then we define $f^{(n+1)}$ as the derivative of $f^{(n)}$. By induction, for each $n\in\N$, we have defined the \emph{$n$th-order derivative} $f^{(n)}$ of $f$.


\begin{lemma}\label{negfunctionLem} 

If $f$ is continuous [respectively, differentiable] on $\theSet\sub\R$, then so is $-f$.
\end{lemma}
\begin{proof} 

Let $c\in\theSet$. If $f$ is continuous at $c$, then by Proposition~\ref{AlgDiffProp}\ref{AlgConti}, $\xcLim f(x)=f(c)$, and by Proposition~\ref{AlgLimProp}\ref{AlgLinear}, $\xcLim(-f)(x)=-\xcLim f(x)=-f(c)$, so by Proposition~\ref{AlgDiffProp}\ref{AlgConti} again, $-f$ is continuous at $c$. For the differentiability case, use Propositions~\ref{AlgDiffProp}\ref{DiffLinear}.
\end{proof}

\begin{lemma}\label{MVTtoTTLem} Let $n\in\N$, and let $f$ be a function $\abClosed\into\R$. Suppose all of $f^{(0)}$, $f^{(1)}$, $f^{(2)}$, $\ldots$~, $f^{(n)}$ are continuous on $\abClosed$, and that $f^{(n)}$ is differentiable on $\abOpen$. Given $x\in\lpar a,b\rbrak$, there exists a unique $\rho(x)\in\R$ such that the function $F:\lbrak a,x\rbrak\into\R$ defined by $$F(t)=f(x)-\sum_{k=0}^n\frac{f^{(k)}(t)}{k!}(x-t)^k- \rho(x) \frac{(x-t)^{n+1}}{(n+1)!},$$ is continuous on $[a,x]$ and differentiable on $(a,x)$. Furthermore, $F(a)=0$ and $$F'(t) = \frac{(x-t)^n}{n!} \lbrak \rho(x) - f^{(n+1)}(t) \rbrak.$$
\end{lemma}
\begin{proof} The assumption that $x\in\lpar a,b\rbrak$ implies $a<x$, so $x-a$ is nonzero, and the number $$\rho(x) = \frac{(n+1)!}{(x-a)^{n+1}}\lbrak f(x) - \sum_{k=0}^n \frac{f^{(k)}(a)}{k!} (x-a)^k \rbrak,$$ exists in $\R$. Define $F:\lbrak a,x\rbrak\into\R$ by
\begin{eqnarray}
    F(t) = f(x) - \sum_{k=0}^n \frac{f^{(k)}(t)}{k!} (x-t)^k - \rho(x) \frac{(x-t)^{n+1}}{(n+1)!}.\label{TTLeq}
\end{eqnarray}
We show that the right-hand side of \eqref{TTLeq} is continuous and differentiable on the desired sets, and we do this by building on the algebraic rules for continuity and differentiability established in Propositions~\ref{AlgLimProp} and \ref{AlgDiffProp}.

Let $c\in\abClosed$. By Proposition~\ref{AlgLimProp}\ref{AlgConst}--\ref{AlgId}, we have $\tcLim t^0=\tcLim 1=1$ and $\tcLim t^1=\tcLim t=c$. If, for some nonnegative integer $m$, we have $\tcLim t^m=c^m$, then by $\tcLim t^1=c$ and Proposition~\ref{AlgLimProp}\ref{AlgProd}, $\tcLim t^{n+1}=\tcLim t\  \tcLim t^n=c^{n+1}$. By induction, for each $m\in\N$, the function $t\mapsto t^m$ is continuous at $c$, so by Proposition~\ref{AlgDiffProp}\ref{AlgConti}, $\tcLim t^m=c^m$. By Proposition~\ref{AlgLimProp}\ref{AlgLinear}, the function $$t\mapsto \sum_{m=0}^{k} (-1)^m \binom{k}{m} x^{k-m} t^m=(x-t)^k,$$ is continuous at the arbitrary $c\in\abClosed$, where, by Proposition~\ref{AlgDiffProp}\ref{AlgConti}, $\tcLim\sum_{k=0}^mc^{m-1-k}t^k=mc^{m-1}$, where the the right-hand side is $\tcLim\frac{t^m-c^m}{t-c}$. This is true, in particular, for any $c\in\abOpen$. This proves that the function $t\mapsto t^m$ is differentiable at the arbitrary $c\in\abOpen$, with derivative $t\mapsto mt^{m-1}$. Using Proposition~\ref{AlgDiffProp}\ref{DiffLinear},
\begin{eqnarray}
    \frac{d}{dt}(x-t)^k &=&\sum_{m=0}^{k} (-1)^m \binom{k}{m} x^{k-m}\frac{d}{dt} t^m= 0+\sum_{m=1}^{k} (-1)^m \binom{k}{m} x^{k-m} mt^{m-1},\nonumber\\
&=&-k \sum_{m=1}^{k} (-1)^{m-1} \binom{k-1}{m-1} x^{k-m} t^{m-1}= -k \sum_{m=0}^{k-1} (-1)^{m} \binom{k-1}{m} x^{(k-1)-m} t^{m},\nonumber\\
 &=&-k(x-t)^{k-1}.\label{TTLeq2}
\end{eqnarray}

Given $k\in\{1,2,\ldots,n\}$, the assumed continuity of $f^{(k)}$ on $\abClosed$ implies that $f^{(k)}$ is defined on all of $\abOpen$, so $f^{(k-1)}$ is differentiable on $\abOpen$. By Propositions~\ref{AlgLimProp}\ref{AlgLinear}, \ref{AlgDiffProp}\ref{AlgConti} and \ref{AlgDiffProp}\ref{DiffLinear}, the right-hand side of \eqref{TTLeq}, as a function of $t$, is continuous on $\abClosed$ and differentiable on $\abOpen$, where the derivative may be computed from \eqref{TTLeq2} and Proposition~\ref{AlgDiffProp}\ref{AlgConst},\ref{DiffLinear},\ref{DiffProd} to be $F'(t) = \frac{(x-t)^n}{n!} \lbrak \rho(x) - f^{(n+1)}(t) \rbrak$.
\end{proof}

\begin{lemma}\label{EVTtoRTLem} Let $f$ be a function $\abOpen\into\R$, and let $c\in\abOpen$. 
\begin{enumerate}
    \item\label{DiffCal01a} If $\xcLim f(x)>0$, then there exist nonempty open intervals $I\sub(a,c)$ and $J\sub(c,b)$ such that $f[I],f[J]\sub(0,\infty)$.
    \item\label{DiffCal02a} If $f'(c)>0$, then there exists a nonempty open interval $I\sub(c,b)$ such that $f[I]\sub\lpar f(c),\infty\rpar$.
       \item\label{DiffCal04a} If $f'(c)<0$, then there exists a nonempty open interval $I\sub(a,c)$ such that $f[I]\sub\lpar f(c),\infty\rpar$.
    \item\label{DiffCal03a} If $f[\abOpen]\sub\lpar-\infty, f(c)\rbrak$, and if $f$ is differentiable at $c$, then $f'(c)= 0$.
\end{enumerate}
\end{lemma}
\begin{proof}\ref{DiffCal01a} Since $L:=\xcLim f(x)>0$, there exists $X\in\Filter_c$ such that $f[X]\sub\lpar L-L,L+L\rpar=\lpar 0,2L\rpar$. From $X\in\Filter_c$, there exists $\eta>0$ such that $X=\abOpen\cap\lbrak(c-\eta,c)\cup(c,c+\eta)\rbrak$. From $c\in\abOpen$, we have $a<c<b$, so $(a,c)\sub\abOpen$ and $(c,b)\sub\abOpen$, and we further have
\begin{eqnarray}
    I :=(a,c)\cap\lbrak(c-\eta,c)\cap\abOpen\rbrak   &=& (c-\eta,c)\cap \lbrak(a,c)\cap\abOpen\rbrak,\nonumber\\
    &=& (c-\eta,c)\cap (a,c),\label{DiffCalSub3}\\
    & \sub & (a,c), \mbox{ as desired},\nonumber\\
    I =(a,c)\cap\lbrak(c-\eta,c)\cap\abOpen\rbrak &\sub & (c-\eta,c)\cap\abOpen,\nonumber\\
    &\sub & \lbrak(c-\eta,c)\cap\abOpen\rbrak\cup \lbrak(c,c+\eta)\cap\abOpen\rbrak=X,\label{DiffCalSub1}\\
    J :=(c,b)\cap\lbrak(c,c+\eta)\cap\abOpen\rbrak &=& (c,c+\eta)\cap \lbrak(c,b)\cap\abOpen\rbrak,\nonumber\\
    &=& (c,c+\eta)\cap (c,b),\label{DiffCalSub4}\\
    & \sub & (c,b), \mbox{ as desired},\nonumber\\
    J =(c,b)\cap\lbrak(c,c+\eta)\cap\abOpen\rbrak &\sub& (c,c+\eta)\cap\abOpen,\nonumber\\
     &\sub & \lbrak(c,c+\eta)\cap\abOpen\rbrak\cup\lbrak(c-\eta,c)\cap\abOpen\rbrak=X.\label{DiffCalSub2}
\end{eqnarray}
By \eqref{DiffCalSub3} and \eqref{DiffCalSub4}, each of $I$ and $J$ is an intersection of two bounded open intervals, and is hence a bounded open interval. From \eqref{DiffCalSub1} and \eqref{DiffCalSub2}, we find that $f[I]$ and $f[J]$ are both subsets of $f[X]\sub\lpar 0,2L\rpar\sub(0,\infty)$. 

To complete the proof, we show that $I$ and $J$ are nonempty. From $a<c<b$ and $\eta>0$, we obtain $a<c$ and $c-\eta<c$ [respectively, $c<b$ and $c<c+\eta$]. If $x:=\frac{c+\max\{a,c-\eta\}}{2}$ [respectively, $y:=\frac{c+\min\{b,c+\eta\}}{2}$], then $a<x<c$ and $c-\eta<x<c$ [respectively, $c<y<b$ and $c<y<c+\eta$]. Therefore, $x\in (c-\eta,c)\cap (a,c)=I\neq \emptyset$ [respectively, $y\in (c,c+\eta)\cap (c,b)=J\neq \emptyset$].\\

\noindent\ref{DiffCal02a} Define $F:\abOpen\setdiff\{c\}\into\R$ by $F(x)=\frac{f(x)-f(c)}{x-c}$. By assumption, $f'(c)=\xcLim F(x)>0$, so by \ref{DiffCal01a}, there exists a nonempty open interval $I\sub(c,b)\sub\abOpen$ such that $F[I]\sub(0,\infty)$. Let $y\in f[I]$. Since $I=(c,b)\sub\abOpen$, or that $I$ is a subset of the domain of $f$, there exists $x\in (c,b)$ such that $y=f(x)$. Also, for this $x$, since  $F[I]\sub(0,\infty)$, we have $\frac{f(x)-f(c)}{x-c}=F(x)>0$. Since $x\in(c,b)$, we have $c<x<b$, so $x-c$ is positive, and multiplying by it both sides of $\frac{f(x)-f(c)}{x-c}>0$, we get $f(x)-f(c)>0$, which implies $f(c)<f(x)=y$. Therefore, $f[I]\sub(f(c),\infty)$.\\

\noindent\ref{DiffCal04a} We also use the function $F:\abOpen\setdiff\{c\}\into\R$ defined by $F(x)=\frac{f(x)-f(c)}{x-c}$. From the assumption that $f'(c)=\xcLim F(x)<0$, we obtain $-\xcLim F(x)>0$. By Proposition~\ref{AlgLimProp}\ref{AlgLinear}, $\xcLim (-F)(x)>0$, so by \ref{DiffCal01a}, there exists a nonempty open interval $I\sub(a,c)\sub\abOpen$ such that $(-F)[I]\sub(0,\infty)$. Since $I$ is a subset of the domain of $f$, for an arbitrary $y\in f[I]$, there exists $x\in I=(a,c)$ such that $y=f(x)$, and $-\frac{f(x)-f(c)}{x-c}>0$. Since $x\in(a,c)$, we have $a<x<c$, so $x-c$ is negative, and multiplying by it both sides of $-\frac{f(x)-f(c)}{x-c}>0$ results to $-[f(x)-f(c)]<0$, which implies $-f(x)+f(c)<0$, so $f(c)<f(x)=y$. This proves $f[I]\sub(f(c),\infty)$.\\

\noindent\ref{DiffCal03a} If $f'(c)>0$, then by \ref{DiffCal02a}, there exists a nonempty open interval $I\sub(c,b)\sub\abOpen$ such that\linebreak $f[I]\sub (f(c),\infty)$. From $I\neq\emptyset$, there exists $\xi\in I$, and from $f[I]\sub (f(c),\infty)$, we have $f(c)<f(\xi)$. But since $\xi\in I\sub\abOpen$, by the assumption that $f[\abOpen]\sub\lpar -\infty,f(c)\rbrak$, we have $f(\xi)\leq f(c)$.$\lightning$

For the case $f'(c)<0$, we use \ref{DiffCal04a}, to obtain a nonempty open interval $I\sub(a,c)\sub\abOpen$ with the property $f[I]\sub (f(c),\infty)$. Thus, there exists $\xi\in I$, and from $f[I]\sub (f(c),\infty)$, we have $f(c)<f(\xi)$, but from $f[\abOpen]\sub\lpar -\infty,f(c)\rbrak$, we have $f(\xi)\leq f(c)$.$\lightning$

By assumption $f'(c)$ exists in $\R$, but since $f'(c)>0$ and $f'(c)<0$ have led to contradictions, by the Triangle Inequality, the only possibility is $f'(c)=0$.
\end{proof}

\subsection{Compact Sets and Uniform Continuity}\label{FourthCirclePrelims}

Given $\theSet\sub\R$ and given a collection $\Cover$ of subsets of $\R$, we say that $\Cover$ is a \emph{cover} of $\theSet$, or that \emph{$\Cover$ covers $\theSet$}, or that \emph{$\theSet$ can be covered by sets from $\Cover$}, if $\theSet\sub\bigcup_{U\in\Cover} U$. 

\begin{proposition}\label{UnionFiniteProp} If $\Cover$ is a cover of $\theSet$, and if each of $A,B\sub\theSet$ may be covered by a finite number of sets from $\Cover$, then so does $A\cup B$.
\end{proposition}
\begin{proof} If there exist $U_1,U_2,\ldots, U_m\in\Cover$ such that $A\sub\bigcup_{k=1}^mU_k$ and if there exist $U_{m+1},U_{m+2},\ldots,$ $U_n\in\Cover$ such that $B\sub\bigcup_{k=m}^nU_k$, then $A\cup B\sub\bigcup_{k=1}^nU_k$, where $U_1$, $U_2$, \ldots , $U_n$ are finitely many sets from $\Cover$.
\end{proof}

\begin{proposition}\label{CoverLengthProp} If the intervals $(a_1,b_1)$, $(a_2,b_2)$, \ldots , $(a_n,b_n)$ cover $\abClosed$, then $b-a<\sum_{k=1}^n(b_k-a_k)$.
\end{proposition}
\begin{proof} We use induction on $n$. If $n=1$, then $\abClosed\sub(a_1,b_1)$, so $a_1<a$ and $b<b_1$, where the former implies $-a<-a_1$. Thus, $b-a<b_1-a_1$. Suppose that for some $n\in\N$, any closed and bounded interval $I$ covered by $n$ bounded open intervals $I_k$ has the property $\ell(I)< \sum_{k=1}^n\ell(I_k)$. Suppose $(a_1,b_1)$, $(a_2,b_2)$, \ldots , $(a_{n+1},b_{n+1})$ cover $\abClosed$. Since $a\in\abClosed\sub\bigcup_{k=1}^{n+1}(a_k,b_k)$, there exists $K\in\{1,2,\ldots,n\}$ such that $a\in(a_K,b_K)$, so $a_K<a<b_K$. If $b<b_K$, then $\abClosed\sub(a_K,b_K)$, which falls under the base case of the induction, and we are done. Suppose henceforth that $b_K\leq b$. Thus, $a_K<a<b_K\leq b$, and so, $\lbrak b_K,b\rbrak\sub\abClosed\sub\bigcup_{k=1}^{n+1}(a_k,b_k)$. This implies
\begin{eqnarray}
    \lbrak b_K,b\rbrak\sub\bigcup_{k=1}^{n+1}(a_k,b_k).\label{CoverCoverEq}
\end{eqnarray}
From $a_K<a<b_K\leq b$, we find that $\lbrak b_K,b\rbrak$ and $(a_K,b_K)$ are disjoint. Thus, if $x\in\lbrak b_K,b\rbrak$, then $x\notin (a_K,b_K)$, and we may remove $(a_K,b_K)$ from the union in the right-hand side of \eqref{CoverCoverEq}. Thus, $\lbrak b_K,b\rbrak$ is covered by $n$ sets, and by the inductive hypothesis, the length $b-b_k$ of $[b_K,b]$ is less than the sum of the lengths of $(a_1,b_1)$, $(a_2,b_2)$, \ldots , $(a_{n+1},b_{n+1})$ except $(a_K,b_K)$, which is the sum of the lengths of $(a_1,b_1)$, $(a_2,b_2)$, \ldots , $(a_{n+1},b_{n+1})$ minus the length of $(a_K,b_K)$. That is, $b-b_K\leq a_K-b_K+\sum_{k=1}^{n+1}(b_k-a_k)$, which implies
\begin{eqnarray}
    b-a_K< \sum_{k=1}^{n+1}(b_k-a_k).\label{CoverCoverEq2}
\end{eqnarray}
But from $a_K<a<b_K\leq b$, we obtain $-a<-a_K$, and so $b-a<b-a_K$, and by \eqref{CoverCoverEq2}, we further have $b-a< \sum_{k=1}^{n+1}(b_k-a_k)$. By induction, we obtain the desired statement.
\end{proof}

Suppose $\Cover$ is a cover of $\theSet$. If $\Base\sub\Cover$ is also a cover of $\theSet$, then $\Base$ is said to be a \emph{subcover} of $\theSet$, or that the cover \emph{$\Cover$ (of $\theSet$) has a subcover} (which in this case is $\Base$). We follow the traditional but often questioned practice of attaching the modifier ``open'' to ``cover'' to refer to the condition that the sets in the cover are open sets, while attaching ``finite'' to ``cover'' (or ``subcover'') to refer to the cardinality of the cover or subcover (and not of the sets in it). A set $\theSet\sub\R$ is a \emph{compact} set if any open cover of $\theSet$ has a finite subcover.

\begin{proposition}\label{contImPropII} If $I$ is compact and $f$ is continuous on $I$, then $f[I]$ is compact.
\end{proposition}
\begin{proof} Let $\Cover$ be an open cover of $f[I]$. Thus, $f[I]\sub\bigcup_{U\in\Cover}U$, and taking the inverse image of both sides under $f$, we have $I\sub \bigcup_{U\in\Cover}f^{-1}[U]$, but since each set $U\in\Cover$ is open and $f$ is continuous, each inverse image $f^{-1}[U]$ is open, so the previous set inclusion implies that $\Base:=\{f^{-1}[U]\  :\  U\in\Cover\}$ is an open cover of $I$. Since $I$ is compact, there exist finitely many sets $f^{-1}[U_1]$, $f^{-1}[U_2]$, \ldots , $f^{-1}[U_n]$ from $\Base$, which implies there exist finitely many sets $U_1$, $U_2$, \ldots , $U_n$ from $\Cover$, such that $I\sub\bigcup_{k=1}^nf^{-1}[U_k]$. Taking the image of both sides under $f$, we obtain $f[I]\sub\bigcup_{k=1}^nU_k$. Therefore, $f[I]$ is compact.
\end{proof}

\begin{proposition}\label{conNecPropIII} If there exists a closed and bounded interval (with positive length) that is compact, then any closed and bounded interval is compact.
\end{proposition}
\begin{proof} Suppose some interval $[a_0,b_0]$ is compact with positive length (so $a_0<b_0$), and let $\abClosed$ be an arbitrary closed and bounded interval. By Propostion~\ref{homeoProp}, there exists a function $f:\lbrak a_0,b_0\rbrak\into \abClosed$, continuous on $[a_0,b_0]$, such that the image of $\lbrak a_0,b_0\rbrak$ under $f$ is equal to $\abClosed$, and such an image set, according to Proposition~\ref{contImPropII}, is compact.
\end{proof}

\begin{lemma}\label{UIKtoLCLLem} If $\theSet$ is a nonempty compact set, and if $\Cover$ is an open cover of $\theSet$, then there exists $\delta>0$, which is called a \emph{Lebesgue number} of the open cover $\Cover$, such that for any $x,c\in\theSet$, if $|x-c|<\delta$, then there exists $G\in\Cover$ such that $x,c\in G$.
\end{lemma}
\begin{proof} Since $\Cover$ covers $\theSet$, we have $\theSet\sub\bigcup_{G\in\Cover} G$, so given $t\in\theSet$, there exists $G_t\in\Cover$ such that $t\in G_t$. Since $G_t$ is open, there exists $\delta_t>0$ such that $|x-t|<\delta_t$ implies $x\in G_t$. Since $\txthalf\delta_t<\delta_t$, we find that $|x-t|<\txthalf\delta_t$ implies $|x-t|<\delta_t$. Henceforth, $|x-t|<\txthalf\delta_t$ implies $x\in G_t$, or equivalently, $x\in\lpar t-\txthalf\delta_t,t+\txthalf\delta_t\rpar$ implies $x\in G_t$. At this point, we have proven that
\begin{enumerate}
    \item[\Star] for each $t\in\theSet$, we have $\lpar t-\txthalf\delta_t,t+\txthalf\delta_t\rpar\sub G_t$.
\end{enumerate}
Also, given an arbitrary $t\in\theSet$, since $t\in \lpar t-\txthalf\delta_t,t+\txthalf\delta_t\rpar\sub \bigcup_{s\in\theSet}\lpar s-\txthalf\delta_s,s+\txthalf\delta_s\rpar$, we have $\theSet\sub \bigcup_{s\in\theSet}\lpar s-\txthalf\delta_s,s+\txthalf\delta_s\rpar$, so the collection $\Base:=\left\{\lpar s-\txthalf\delta_s,s+\txthalf\delta_s\rpar\  :\  s\in\theSet\right\}$ is an open cover of $\theSet$. Since $\theSet$ is compact, $\Base$ has a finite subcover, the sets in which, we denote by $\lpar s_1-\txthalf\delta_{s_1},s_1+\txthalf\delta_{s_1}\rpar$, $\lpar s_2-\txthalf\delta_{s_2},s_2+\txthalf\delta_{s_2}\rpar$, \ldots , $\lpar s_n-\txthalf\delta_{s_n},s_n+\txthalf\delta_{s_n}\rpar$. From \Star,
\begin{enumerate}
    \item[\SStar] for each $k\in\{1,2,\ldots,n\}$, we have $\lpar s_k-\txthalf\delta_{s_k},s_k+\txthalf\delta_{s_k}\rpar\sub G_{s_k}$.
\end{enumerate}
Since each $\delta_{s_k}$ is positive, so is
\begin{eqnarray}
    \delta:=\min\left\{\txthalf\delta_{s_1}, \txthalf\delta_{s_2},\ldots,\txthalf\delta_{s_n}\right\}.\label{LebesgueNumber}
\end{eqnarray}
Suppose $x,c\in\theSet$ such that $|x-c|<\delta$. Since $\Base$ covers $\theSet$, we have $c\in\theSet\sub\bigcup_{k=1}^n\lpar s_k-\txthalf\delta_{s_k},s_k+\txthalf\delta_{s_k}\rpar$, so there exists $K\in\{1,2,\ldots,n\}$ such that $c\in\lpar s_K-\txthalf\delta_{s_K},\delta+\txthalf\delta_{s_K}\rpar$, or that $|c-s_K|<\txthalf\delta_{s_K}$. From \eqref{LebesgueNumber}, $|x-c|<\delta\leq\txthalf\delta_{s_K}$, and by the Triangle Inequality:

\noindent$|x-s_K|\leq |x-c|+|c-s_K|<\txthalf\delta_{s_K}+\txthalf\delta_{s_K}=\delta_{s_K}$, but the inequality $|x-s_K|<\delta_{s_K}$ implies $x\in\lpar s_K-\delta_{s_K},s_K+\delta_{s_K}\rpar$, which, in conjunction with $c\in\lpar s_K-\txthalf\delta_{s_K},\delta+\txthalf\delta_{s_K}\rpar$ and \SStar, implies $x,c\in G_{s_K}\in\Cover$. This completes the proof.
\end{proof}

\begin{lemma}\label{LCLtoUCTLem} If $f$ is continuous on $\theSet$, then for each $\varepsilon>0$, there exists an open cover $\Cover$ of $\theSet$ such that for each $G\in\Cover$, if $x,c\in G$, then $|f(x)-f(c)|<\varepsilon$.
\end{lemma}
\begin{proof} Let $t\in\theSet$. Given $\varepsilon>0$, we have $\txthalf\varepsilon>0$, so by the continuity of $f$ at $t$, there exists $\delta_t>0$ such that $|x-t|<\delta$ implies $|f(x)-f(t)|<\txthalf\varepsilon$, or equivalently, 
\begin{enumerate}
    \item[\Star] if $x\in\lpar t-\delta_t,t+\delta_t\rpar$, then $|f(x)-f(t)|<\txthalf\varepsilon$.
\end{enumerate}
Since $\delta_t>0$, we have $t-\delta_t<t<t+\delta_t$, or that $t\in\lpar t-\delta_t,t+\delta_t\rpar\sub\bigcup_{s\in\theSet}\lpar s-\delta_s,s+\delta_s\rpar$. Since $t\in\theSet$ is arbitrary, we have $\theSet\sub\bigcup_{s\in\theSet}\lpar s-\delta_s,s+\delta_s\rpar$, where each interval in the union is an open set. Hence, $\Cover:=\{\lpar s-\delta_s,s+\delta_s\rpar\  :\  s\in\theSet\}$ is an open cover of $\theSet$. If $x,c\in\lpar s-\delta_s,s+\delta_s\rpar$, then by \Star, we have $|f(x)-f(s)|<\txthalf\varepsilon$ and $|f(s)-f(c)|=|f(c)-f(s)|<\txthalf\varepsilon$. By the Triangle Inequality, $|f(x)-f(c)|\leq|f(x)-f(s)|+|f(s)-f(c)|<\txthalf\varepsilon+\txthalf\varepsilon=\varepsilon$.
\end{proof}

Given $I\sub\R$, a function $f:I\into\R$ is \emph{uniformly continuous} if, for each $\varepsilon>0$, there exists $\delta>0$ such that for any $x,c\in I$, if $|x-c|<\delta$, then $|f(x)-f(c)|<\varepsilon$. If indeed $f:I\into\R$ is uniformly continuous, then for each $t\in I$, setting $c=t$ in the aforementioned epsilon-delta definition of uniform continuity, we find that $f$ is continuous at $t$, and so, uniform continuity implies continuity. To imply uniform continuity, an additional condition has to be combined with an assumption of continuity, and this is part of the rationale for the circle of real analysis principles later in Theorem~\ref{FourthCircle}. The remainder of this section, Section~\ref{FourthCirclePrelims}, should have been placed in the next section, Section~\ref{FifthCirclePrelims}, as these belong more rightfully to the elements of integration theory. However, these have been used in the proof of Theorem~\ref{FourthCircle}.

Given $a,b\in\R$, with $a\leq b$, by a \emph{partition} of $\abClosed$, we mean a finite subset of $\abClosed$ that contains the endpoints $a$ and $b$. If $a=b$, then this definition implies that $\{a\}$ is the only possible partition of $\abClosed$. If $a<b$, then the usual notation is to index the elements of a partition of $\abClosed$ by subscripts in strictly ascending order, such as: $a=x_0<x_1<\cdots<x_n=b$ for the elements of a partition of $\abClosed$ with cardinality $n$. 

\begin{proposition}\label{PartitionCupProp} If $\Delta=\{x_0,x_1,\ldots,x_n\}$ is a partition of $\abClosed$, then $\abClosed=\bigcup_{k=1}^n[x_{k-1},x_k]$.
\end{proposition}
\begin{proof} For each $k\in\{1,2,\ldots,n\}$, we have $x_{k-1},x_k\in\abClosed$, and so $a\leq x_{k-1},x_k\leq b$, so\linebreak $[x_{k-1},x_k]\sub\abClosed$, and we further have $\bigcup_{k=1}^n[x_{k-1},x_k]\sub\abClosed$. 

Let $x\in\abClosed$. If $x>x_k$ for all $k\in\{1,2,\ldots,n\}$, then, in particular at $k=n$, we have $b=x_n<x$, so $x\notin\abClosed$.$\lightning$ Hence, there exists $K\in\{1,2,\ldots,n\}$ such that $x\leq x_K$. That is, $K$ is an element of $\theSet:=\{m\in\N\  :\  x\leq x_m,\  1\leq m\leq n\}$, which is hence a nonempty subset of $\N$. The \emph{Well Ordering Principle} states that every nonempty subset of $\N$ has a least element, or that there exists $L\in\theSet$ (which means that $x\leq x_L$) such that $m\in\theSet$ implies $L\leq m$. By contraposition, $m<L$ implies $m\notin\theSet$ or that $x_m< x$. This is true, in particular, at $m=L-1$. Thus, $x_{L-1}< x\leq x_{L}$, which implies $x\in\lbrak x_{L-1},x_L\rbrak$, where $L\geq 1$ implies $L-1\geq 0$. Thus, $x\in\lbrak x_{L-1},x_L\rbrak\sub \bigcup_{k=1}^n[x_{k-1},x_k]$, and therefore $\abClosed\sub\bigcup_{k=1}^n[x_{k-1},x_k]$. This completes the proof.
\end{proof}

A function $\varphi:\abClosed\into\R$ is a \emph{step function} if there exists a partition $\Delta=\{x_0,x_1,\ldots,x_n\}$ of $\abClosed$ such that for each $k\in\{1,2,\ldots,n\}$, the function $\varphi$ is constant on $(x_{k-1},x_k)$. Under this definition, the values of the step function $\varphi$ at the partition elements $x_0,x_1,\ldots,x_n$ does not matter.

\begin{proposition}\label{StepImageProp} If $\varphi:\abClosed\into\R$ is a step function, then there exist $c_1$, $c_2$, \ldots , $c_N\in\R$ such that\linebreak $\{\varphi(x)\  :\  x\in\abClosed\}=\{c_1,c_2,\ldots,c_N\}$.
\end{proposition}
\begin{proof}
    Since $\varphi$ is a step function on $\abClosed$, there exists a partition $\{x_0,x_1,\ldots,x_n\}$ of $\abClosed$ such that for each $k\in\{1,2,\ldots,n\}$, the function $\varphi$ has the constant value $C_k$ on $(x_{k-1},x_{k})$. At the partition elements, let $y_k:=\varphi(x_k)$ for all $k\in\{0,1,\ldots, n\}$. Thus the image of $\abClosed$ under $\varphi$ contains only the finitely many real numbers $y_0$, $y_1$, \ldots , $y_n$, $C_1$, $C_2$, \ldots , $C_n$. Rename these $N:=2n+1$ numbers as $c_1$, $c_2$, \ldots , $c_N$.
\end{proof}  

\subsection{Some Elements of Integration Theory}\label{FifthCirclePrelims}

In this section, we give some minimal machinery needed to develop the elementary notion of integral and some consequences such as boundedness of integrable functions and additivity of integrals. We start by showing that the integrals of step functions have some sort of independence from the partition chosen. 

\begin{proposition}\label{anypartitionProp} Let $\varphi:\abClosed\into\R$ be a step function, and, given a partition $\Delta=\{x_0,x_1,\ldots,x_n\}$ of $\abClosed$ such that $\varphi$ has the value $c_i$ on $(x_{k-1},x_{k})$ for all $k\in\{1,2,\ldots,n\}$. Let $\int_a^b\varphi=\sum_{i=1}^nc_i(x_i-x_{i-1})$. If $\Omega=\{u_0,u_1,\ldots,u_m\}$ is a partition of $\abClosed$ that contains $\Delta$, then for each $j\in\{1,2,\ldots,m\}$, there exists $C_j\in\R$ such that $\{C_j\}=\varphi\lbrak\lpar u_{j-1},u_j\rpar\rbrak$, and $\int_a^b\varphi=\sum_{j=1}^mC_j(u_j-u_{j-1})$. [Hence, the number $\int_a^b\varphi$ is uniquely determined by $\Delta$, and is called the \emph{integral of (the step function) $\varphi$}.]
\end{proposition}
\begin{proof}We use induction on $m$. Suppose that for any partition $\Omega=\{v_0,v_1,\ldots,v_t\}$ [where $t<m$] of $\abClosed$ that contains $\Delta$, the statement holds. Then this inductive hypothesis applies to the partition\linebreak $\{u_0,u_1,\ldots,u_m\}\setdiff\{u_{m-1}\}$ of $\abClosed$. That is, for each\ $j\in\{1,2,\ldots,m-2,m\}$, there exists $C_j\in\R$ such that
\begin{eqnarray}
\{C_j\} & = & \varphi\lbrak\lpar u_{j-1},u_j\rpar\rbrak,\nonumber\\
\{C_m\} & = & \varphi\lbrak\lpar u_{m-2},u_m\rpar\rbrak,\label{refinement1a}\\
\int_a^b\varphi & = & \lbrak\sum_{j=1}^{m-2}C_j(u_j-u_{j-1})\rbrak+C_{m}(u_m-u_{m-2}).\label{refinement2}
\end{eqnarray}
Since $u_{m-2}<u_{m-1}<u_m$, $\lpar u_{m-2},u_{m-1}\rpar\sub\lpar u_{m-2},u_m\rpar$ and $\lpar u_{m-1},u_{m}\rpar\sub\lpar u_{m-2},u_m\rpar$, which imply $\varphi\lbrak\lpar u_{m-2},u_{m-1}\rpar\rbrak\sub\varphi\lbrak\lpar u_{m-2},u_m\rpar\rbrak$ and $\varphi\lbrak\lpar u_{m-1},u_{m}\rpar\rbrak\sub\varphi\lpar u_{m-2},u_m\rpar$ and by \eqref{refinement1a},
\begin{eqnarray}
\varphi\lbrak\lpar u_{m-2},u_{m-1}\rpar\rbrak \sub  \{C_m\},\qquad \varphi\lbrak\lpar u_{m-1},u_m\rpar\rbrak\sub  \{C_m\},\nonumber
\end{eqnarray}
but since $C_m$ is contained in each of the above image sets, we further get set equality. Thus, we may define $C_{m-1}:=C_m$, and from \eqref{refinement2},
\begin{eqnarray}
\int_a^b\varphi & = & \lbrak\sum_{j=1}^{m-2}C_j(u_j-u_{j-1})\rbrak+C_{m}(u_m-u_{m-1}+u_{m-1}-u_{m-2}),\nonumber\\
& = & \lbrak\sum_{j=1}^{m-2}C_j(u_j-u_{j-1})+C_{m}(u_m-u_{m-1})\rbrak+C_m(u_{m-1}-u_{m-2}),\nonumber\\
& = & \lbrak\sum_{j=1}^{m-2}C_j(u_j-u_{j-1})+C_{m}(u_m-u_{m-1})\rbrak+C_{m-1}(u_{m-1}-u_{m-2}),\nonumber\\
& = & \sum_{j=1}^mC_j(u_j-u_{j-1}).\nonumber
\end{eqnarray}
By induction, the desired result follows.
\end{proof}

\begin{proposition}\label{negStepIntProp} If $\abSteps$ is defined as the collection of all step functions $\abClosed\into\R$, then, given a constant function $C:x\mapsto C$ and given $\varphi,\psi\in\abSteps$, we have $C,C\varphi,\varphi+\psi\in\abSteps$. Furthermore,\linebreak $\int_a^bC=C(b-a)$, $\int_a^b C\varphi=C\int_a^b\varphi$ and $\int_a^b(\varphi+\psi)=\int_a^b\varphi+\int_a^b\psi$.
\end{proposition}
\begin{proof} Using Proposition~\ref{anypartitionProp}, there exists a partition $\Delta=\{x_0,x_1,\ldots,x_n\}$ of $\abClosed$ such that $\varphi$ has the value $C_k$, and $\psi$ has value $G_k$, on $(x_{k-1},x_{k})$ for all $k\in\{1,2,\ldots,n\}$. The functions $C$, $C\varphi$ and $\varphi+\psi$ have the values $C$, $C\cdot C_k$ and $C_k+G_k$, respectively, and are hence constant, on $(x_{k-1},x_{k})$ for all $k\in\{1,2,\ldots,n\}$, so $C,C\varphi,\varphi+\psi\in\abSteps$. Furthermore, 
\begin{eqnarray}
\int_a^bC &=&\sum_{k=1}^nC(x_k-x_{k-1})=C\sum_{k=1}^nx_{k}-C\sum_{k=1}^nx_{k-1},\nonumber\\
    &=&C\sum_{k=1}^nx_{k}-C\sum_{k=0}^{n-1}x_{k}=Cx_n+C\sum_{k=1}^{n-1}x_{k}-C\sum_{k=1}^{n-1}x_{k}-Cx_0,\nonumber\\
    &=&Cx_n-Cx_0=Cb-Ca=C(b-a),\nonumber\\
    \int_a^b C\varphi&=&\sum_{k=1}^nC\cdot C_k(x_{k}-x_{k-1})=C\sum_{k=1}^nC_k(x_{k}-x_{k-1})=C\int_a^b\varphi,\nonumber\\
    \int_a^b(\varphi+\psi)&=&\sum_{k=1}^n(C_k+G_k)(x_{k}-x_{k-1})=\sum_{k=1}^nC_k(x_{k}-x_{k-1})+\sum_{k=1}^nG_k(x_{k}-x_{k-1}),\nonumber\\
    &=&\int_a^b\varphi+\int_a^b\psi.\qedhere\nonumber
\end{eqnarray}
\end{proof}

Given $X\sub\R$ and functions $f,g,h:X\into\R$, consider the assertions
\begin{eqnarray}
\forall x\in X && f(x)\leq g(x),\nonumber\\
\forall x\in X && f(x)<g(x),\nonumber\\
\forall x\in X && f(x)=g(x),\nonumber\\
 \forall x\in X && f(x)\leq g(x)\leq h(x) \implies f(x)\leq h(x),\nonumber\\
 \forall x\in X && f(x)< g(x)< h(x) \implies f(x)< h(x),\nonumber\\
 \forall x\in X && f(x)<g(x) \implies f(x)\leq g(x),\nonumber\\
\forall x\in X && f(x)\leq g(x) \iff  -g(x)\leq -f(x),\nonumber\\
\forall x\in X && |f(x)-g(x)|<\varepsilon \iff g(x)-\varepsilon<f(x)<g(x)+\varepsilon.\nonumber
\end{eqnarray}
Since each of these is true for any $x$ in the common domain $X$ of the functions $f$, $g$ and $h$, we will simply remove any reference to the element $x$, and write the above into what may be called ``pointwise function notation,'' and so each of the above assertions shall simply be written as
\begin{eqnarray}
    && f\leq g,\nonumber\\
&& f<g,\nonumber\\
&& f=g,\nonumber\\
&&  f\leq g\leq h \implies f\leq h,\nonumber\\
&&  f< g< h \implies f< h,\nonumber\\
&&  f<g \implies f\leq g,\nonumber\\
&& f\leq g \iff -g\leq -f,\nonumber\\
&& |f-g|<\varepsilon \iff g-\varepsilon<f<g+\varepsilon,\nonumber
\end{eqnarray}
respectively.

\begin{proposition}\label{monotoneProp} Let $\varphi,\psi\in\abSteps$. If $\varphi\leq\psi$, then $\int_a^b\varphi\leq\int_a^b\psi$. If $\varphi<\psi$, then $\int_a^b\varphi<\int_a^b\psi$.
\end{proposition}
\begin{proof} By the definition of step function, there exist partitions $\{x_0,x_1,\ldots,x_n\}$ and $\{y_0,y_1,\ldots,y_m\}$ of $\abClosed$ such that for any $i\in\{1,2,\ldots,n\}$ and any $j\in\{1,2,\ldots,m\}$ there exist $A_i,B_j\in\R$ such that $\int_a^b\varphi=\sum_{i=1}^nA_i(x_i-x_{i-1})$ and $\int_a^b\psi=\sum_{j=1}^mB_j(y_j-y_{j-1})$. Let
\begin{eqnarray}
\{z_0,z_1,\ldots,z_h\} := \{x_0,x_1,\ldots,x_n\}\cup\{y_0,y_1,\ldots,y_m\},\nonumber
\end{eqnarray}
with $a=z_0<z_1<\cdots<z_n=b$. By Proposition~\ref{anypartitionProp}, for each $k\in\{1,2,\ldots,h\}$, there exist $U_k,V_k\in\R$ such that
\begin{eqnarray}
\{U_k\} & = & \varphi\lbrak\lpar z_{k-1},z_k\rpar\rbrak,\label{monostep1} \\
\{V_k\} & = & \psi\lbrak\lpar z_{k-1},z_k\rpar\rbrak,\label{monostep2}\\
\int_a^b\varphi & = &\sum_{k=1}^hU_k(z_k-z_{k-1}),\label{monostep3}\\
\int_a^b\psi & = &\sum_{k=1}^hV_k(z_k-z_{k-1}).\label{monostep4}
\end{eqnarray}
Let $k\in\{1,2,\ldots,h\}$, and let $x\in\lpar z_{k-1},z_k\rpar$. Since $\varphi\leq\psi$ [respectively, $\varphi<\psi$], we have $\varphi(x)\leq\psi(x)$ [respectively, $\varphi(x)<\psi(x)$], and by \eqref{monostep1}, \eqref{monostep2}, $U_k\leq V_k$ [respectively, $U_k<V_k$], both sides of which, we multiply by $z_k-z_{k-1}$, which is positive because $z_{k-1}<z_k$. Thus, $U_k(z_k-z_{k-1})\leq V_k(z_k-z_{k-1})$ [respectively, $U_k(z_k-z_{k-1})< V_k(z_k-z_{k-1})$]. Taking the summation for all $k\in\{1,2,\ldots,h\}$, and making substitutions using \eqref{monostep3}, \eqref{monostep4}, we get $\int_a^b\varphi\leq\int_a^b\psi$ [respectively, $\int_a^b\varphi<\int_a^b\psi$].
\end{proof}

The \emph{lower Darboux integral} of $f$ over $\abClosed$ is defined as $$\abIntLow f:=\sup\left\{\int_a^b\varphi\  :\  \varphi\in\abSteps,\quad \varphi\leq f\right\},$$
while the \emph{upper Darboux integral} of $f$ over $\abClosed$ is $$\abIntUp f:=-\sup\left\{\int_a^b\varphi\  :\  \varphi\in\abSteps,\quad \varphi\leq -f\right\}.$$

For each of these definitions, replacing $f$ by $-f$ and comparing with the other definition results to 
\begin{eqnarray}
    \abIntLow(-f)=-\abIntUp f,\qquad \abIntUp(-f)=-\abIntLow f,\label{DarbouxNeg}
\end{eqnarray}
or alternatively, if one of the identities in \eqref{DarbouxNeg} has been obtained from some premise, replacing $f$ by $-f$ results to the other.

\begin{proposition}\label{SumtoStepProp} If the lower and upper Darboux integrals exist, then $\abIntLow f\leq \abIntUp f$.
\end{proposition}
\begin{proof} Let $\theSet_1:=\left\{\int_a^b\varphi\  :\  \varphi\in\abSteps,\  \varphi\leq f\right\}$ and $\theSet_2:=\left\{\int_a^b\varphi\  :\  \varphi\in\abSteps,\  \varphi\leq -f\right\}$. By assumption, the suprema $\abIntLow f=\sup\theSet_1$ and $-\abIntUp f=\sup\theSet_2$ exist. Let $\phi,\psi\in\abSteps$ such that $\phi\leq f$ and $\psi\leq-f$, where the latter implies $f\leq -\psi$. Thus, $\phi\leq f\leq -\psi$, which implies $\phi\leq-\psi$. By Proposition~\ref{monotoneProp}, $\int_a^b\phi\leq\int_a^b(-\psi)$, and by Proposition~\ref{negStepIntProp}, $\int_a^b\phi\leq-\int_a^b\psi$, where the left-hand side is an arbitrary element of $\theSet_1$, so the right-hand side is an upper bound of $\theSet_1$, and its relationship to the supremum $\abIntLow f=\sup\theSet_1$ is given by the inequality $\abIntLow f\leq -\int_a^b\psi$, so $\int_a^b\psi\leq -\abIntLow f$. This time, the left-hand side is an arbitrary element of $\theSet_2$, so the right-hand side is an upper bound, and is related to the supremum by $-\abIntUp f\leq -\abIntLow f$. Therefore, $\abIntLow f\leq \abIntUp f$.
\end{proof}

We say that a function $f$ is \emph{bounded above on $X$ by $M\in\R$} if the image set $f[X]$ is bounded above by $M$, or that $M$ is an upper bound of $f[X]$.

\begin{proposition}\label{fboundCupProp} If $f$ is bounded above on each of the sets $X_1$, $X_2$, \ldots , $X_n$, then $f$ is bounded above on $\bigcup_{k=1}^nX_k$.
\end{proposition}
\begin{proof} Using induction on $n$, the desired statement is trivially true at $n=1$. Suppose that for some $n\in\N$ if $X$ is bounded above on $n$ sets, then $f$ is bounded above on the union of these $n$ sets. If $f$ is bounded above on each of the sets $X_1$, $X_2$, \ldots , $X_{n+1}$, then by the inductive hypothesis, $f$ is bounded above by some $M_1$ on the union $Y$ of the first $n$ sets, and by assumption is bounded above by $M_2$ on the last set $X_{n+1}$. If $x\in Y\cup X_{n+1}$, then $f(x)\leq M_1\leq\max\{M_1,M_2\}$ or $f(x)\leq M_2\leq\max\{M_1,M_2\}$, so $f$ is bounded above by $\max\{M_1,M_2\}$ on $Y\cup X_{n+1}$. By induction, we get the desired statement.
\end{proof}

\begin{proposition}\label{stepBVTProp} Any step function $\abClosed\into\R$ is bounded above on $\abClosed$.
\end{proposition}
\begin{proof} If $\varphi:\abClosed\into\R$ is a step function, then there exists a partition $\Delta=\{x_0,x_1,\ldots,x_n\}$ of $\abClosed$ such that for each $k\in\{1,2,\ldots,n\}$, the function $\varphi$ is constant, and hence bounded above, on $(x_{k-1},x_{k})$, and also on each the singletons $\{x_{k-1}\}$ and $\{x_{k}\}$. By Proposition~\ref{fboundCupProp}, $\varphi$ is bounded above on\linebreak $[x_{k-1},x_{k}]=(x_{k-1},x_{k})\cup\{x_{k-1}\}\cup\{x_{k}\}$ for any $k\in\{1,2,\ldots,n\}$. Using Proposition~\ref{fboundCupProp} again, and also Proposition~\ref{PartitionCupProp}, $\varphi$ is bounded above on $\bigcup_{k=1}^n[x_{k-1},x_k]=\abClosed$.
\end{proof}

\begin{lemma}\label{DITtoBVTLem} If $\abIntUp f$ exists, then $f$ is bounded above on $\abClosed$.
\end{lemma}
\begin{proof} Since $-\abIntUp f=\sup\left\{\int_a^b\varphi\  :\  \varphi\in\abSteps,\  \varphi\leq -f\right\}$ and $1>0$, by Proposition~\ref{xiProp}, there exists $\psi\in\abSteps$ such that $\psi\leq-f$ and $-\abIntUp f-1<\int_a^b\psi\leq-\abIntUp f$. From $\psi\leq-f$, we obtain\linebreak $f\leq-\psi$, where $-\psi\in\abSteps$ according to Proposition~\ref{negStepIntProp}, and by Proposition~\ref{stepBVTProp}, is bounded above by some\linebreak $M\in\R$. Thus, for any $x\in\abClosed$, we have $f(x)\leq -\psi(x)\leq M$. Therefore, $f$ is bounded above on $\abClosed$.
\end{proof}

Given a collection $\Cover$ of sets, a function $f:\Cover\into\bigcup_{S\in\Cover} S$ is a \emph{choice function} if, for any $S\in\Cover$, we have $f(S)\in S$. Given a function $f:\abClosed\into\R$, a partition $\Delta=\{x_0,x_1,\ldots,x_n\}$ of $\abClosed$, and a choice function $\xi:\lbrak x_{k-1},x_k\rbrak\mapsto\xi_k$, where $k\in\{1,2,\ldots,n\}$ [which we shall refer to as a \emph{choice function associated to the partition $\Delta$}], the \emph{Riemann sum} of $f$ (over $\abClosed$ for the partition $\Delta$ and the choice function $\xi$) is defined as $$\Riemann(f,\Delta,\xi):=\sum_{k=1}^nf(\xi_k)\  (x_k-x_{k-1}).$$
Let $\abParts$ be the collection of all partitions of $\abClosed$. For each $\Delta=\{x_0,x_1,\ldots,x_n\}\in\abParts$, let $\R^\Delta$ be the collection of all choice functions $\xi:\lbrak x_{k-1},x_k\rbrak\mapsto\xi_k$, where $k\in\{1,2,\ldots,n\}$, and let\linebreak $\Sigma_\Delta=\Sigma_\Delta(f):=\{\Riemann(f,\Omega,\xi)\  :\  \Omega\in\abParts,\  \Delta\sub\Omega,\  \xi\in\R^\Omega\}$. 


\begin{proposition}\label{RiemannConvergesProp} The collection $\IntFilter(a,b;f):=\{\Sigma_\Delta\  :\  \Delta\in\abParts\}$ is a filter base\footnote{This statement is from \cite[p.~122]{sie92}.}. If $\IntFilter(a,b;f)$ approaches some $I\in\R$, then $I$ is unique [and the \emph{Riemann integral} of $f$ over $\abClosed$ is defined as $\int_a^bf:=I$].

\end{proposition}
\begin{proof} If $\Sigma_\Delta,\Sigma_\Lambda\in\IntFilter(a,b;f)$, then $\Delta,\Lambda\in\abParts$ are finite subsets of $\abClosed$ that contain $a$ and $b$, and so is $\Delta\cup\Lambda$. Thus, $\Delta\cup\Lambda\in\abParts$ and $\Sigma_{\Delta\cup\Lambda}\in\IntFilter(a,b;f)$. If $\Riemann(f,\Omega,\xi)\in\Sigma_{\Delta\cup\Lambda}$, then\linebreak $\Delta\cup\Lambda\sub\Omega$, where the left-hand side contains both $\Delta$ and $\Lambda$. Thus, $\Riemann(f,\Omega,\xi)\in\Sigma_{\Delta}$ and $\Riemann(f,\Omega,\xi)\in\Sigma_{\Lambda}$, or that $\Riemann(f,\Omega,\xi)\in\Sigma_{\Delta}\cap\Sigma_{\Lambda}$, which proves that $\Sigma_{\Delta\cup\Lambda}\sub\Sigma_{\Delta}\cap\Sigma_{\Lambda}$. The arbitrary $\Sigma_\Delta$ is nonempty for $\Delta=\{x_0,x_1,\ldots,x_n\}$ contains itself, and given the choice function $\xi$ that sends an interval to the left endpoint, the Riemann sum $\sum_{k=1}^nf(x_{k-1})\  (x_k-x_{k-1})$ is an element of $\Sigma_\Delta$. Therefore, $\IntFilter(a,b;f)$ is a filter base. If it approaches some $I\in\R$, then by Proposition~\ref{BaseLimUniqueProp}, the number $I$ is unique.
\end{proof}


\begin{lemma}\label{RITtoBVTLem} If $\int_a^bf$ exists, then $f$ is bounded above on $\abClosed$.
\end{lemma}
\begin{proof} If $I:=\int_a^bf$ exists, then the filter base $\IntFilter(a,b;f)$ approaches $I$. Since $1>0$, there exists\linebreak $\Sigma_\Delta\in\IntFilter(a,b;f)$ such that $\Sigma_\Delta\sub\lpar I-1,I+1\rpar$. Let $\xi\in\R^\Delta$. Since $\Delta$ contains itself, $\Riemann(f,\Delta,\xi)\in\Sigma_\Delta$, so $I-1<\Riemann(f,\Delta,\xi)<I+1$. If $\Delta=\{x_0,x_1,\ldots,x_n\}$, then 
\begin{flalign}
   && \sum_{k=1}^nf(\xi_k)\  (x_k-x_{k-1}) &<I+1, &\mbox{for all }\xi\in\R^\Delta.\label{RiemannBoundEq}
\end{flalign}
Tending towards a contradiction, suppose $f$ is not bounded above on $\abClosed$. If, for all $k\in\{1,2,\ldots,n\}$, the function $f$ is bounded above on $[x_{k-1},x_k]$, then by Propositions~\ref{PartitionCupProp} and \ref{fboundCupProp}, $f$ is bounded above on $\bigcup_{k=1}^n[x_{k-1},x_k]=\abClosed$.$\lightning$ Thus, there exists $K\in\{1,2,\ldots,n\}$ such that $f$ is not bounded above on $[x_{K-1},x_K]$. Let $Y:=\sum_{k=1}^nf(\xi_k)\  (x_k-x_{k-1})-f(\xi_K)(x_K-x_{K-1})$. From our notation for partitions, $x_K-x_{K-1}>0$, so substituting $Y$ into \eqref{RiemannBoundEq} and dividing both sides by $x_K-x_{K-1}$ results to
\begin{flalign}
   && f(\xi_K) &<\frac{I-Y+1}{x_K-x_{K-1}}, &\mbox{for all }\xi\in\R^\Delta.\label{RiemannBoundEq2}
\end{flalign}
Since $f$ is not bounded above on $\lbrak x_{K-1},x_K\rbrak$, the number $\frac{I-Y+1}{x_K-x_{K-1}}$ is not an upper bound of the image of $\lbrak x_{K-1},x_K\rbrak$ under $f$, so there exists $\gamma\in\lbrak x_{K-1},x_K\rbrak$ such that 
\begin{eqnarray}
    f(\gamma)>\frac{I-Y+1}{x_K-x_{K-1}}.\label{RiemannBoundEq3}
\end{eqnarray}
Let $\zeta\in\R^\Delta$ be the choice function $\zeta:\lbrak x_{k-1},x_k\rbrak\mapsto\zeta_k$ where $\zeta_k=x_{k-1}$ if $k\neq K$, and $\zeta_K=\gamma$. By \eqref{RiemannBoundEq3}, $f(\zeta_K)>\frac{I-Y+1}{x_K-x_{K-1}}$, but setting $\xi=\zeta$ in \eqref{RiemannBoundEq2}, $f(\zeta_K) <\frac{I-Y+1}{x_K-x_{K-1}}$.$\lightning$ Therefore, $f$ is bounded above on $\abClosed$.
\end{proof}


If both $\abIntUp f$ and $\abIntLow f$ exist and $\abIntUp f= \abIntLow f$, then $f$ is said to be \emph{Riemann integrable} over $\abClosed$.

\begin{lemma}\label{UAStoRITLem}If $f$ is Riemann integrable over $\abClosed$, then $\int_a^bf$ exists. Furthermore, $\abIntUp f=\int_a^bf= \abIntLow f$.
\end{lemma}
\begin{proof} Let $I:=\abIntUp f = \abIntLow f$, let
\begin{eqnarray}
    \theSet_1:=\left\{\int_a^b\varphi\  :\  \varphi\in\abSteps,\  \varphi\leq f\right\},\qquad\theSet_2:=\left\{\int_a^b\varphi\  :\  \varphi\in\abSteps,\  \varphi\leq -f\right\},\nonumber
\end{eqnarray}
and let $\varepsilon > 0$. Since $I=\abIntLow f=\sup\theSet_1$, and since $\varepsilon>0$ implies $I-\varepsilon<I$, by Proposition~\ref{xiProp}, there exists $\phi\in\abSteps$ with $\phi\leq f$ such that 
\begin{eqnarray}
    I-\varepsilon<\int_a^b\phi\leq I.\label{Riemann1}
\end{eqnarray}
From $-I=-\abIntUp f=\sup\theSet_2$, the inequalities $(-I)-\varepsilon<-I$, and Proposition~\ref{xiProp}, there exists\linebreak $\psi_0\in\abSteps$ with $\psi_0\leq -f$ such that
\begin{eqnarray}
   (-I)-\varepsilon<\int_a^b\psi_0\leq -I.\label{Riemann2}
\end{eqnarray}
By Proposition~\ref{negStepIntProp}, we have $\psi:=-\psi_0\in\abSteps$ with $-\psi\leq -f$, which implies $f\leq \psi$, and \eqref{Riemann2} becomes $(-I)-\varepsilon<-\int_a^b\psi\leq -I$, or equivalently, $I\leq\int_a^b\psi<I+\varepsilon$, which, in conjunction with \eqref{Riemann1} results to
\begin{eqnarray}
I - \varepsilon < \int_a^b \phi \leq I \leq \int_a^b \psi < I + \varepsilon. \label{Riemann3}
\end{eqnarray}
By Proposition~\ref{anypartitionProp}, there exists a partition $\Lambda = \{t_0, t_1, \ldots, t_m\}$ of $\abClosed$ such that, if $\zeta:[t_{k-1},t_k]\mapsto\zeta_k$ is a choice function associated to $\Delta$, then  $\int_a^b\phi=\sum_{k=1}^m\phi(\zeta_k)\  (t_k-t_{k-1})$ and $\int_a^b\psi=\sum_{k=1}^m\psi(\zeta_k)\  (t_k-t_{k-1})$. If $\Delta = \{x_0, x_1, \ldots, x_n\}$ is any partition of $\abClosed$ that contains $\Lambda$, then we may use Proposition~\ref{anypartitionProp} again to deduce that for any choice function $\xi:[x_{k-1},x_k]\mapsto\xi$ associated to $\Delta$, we have the integrals\linebreak $\int_a^b\phi=\sum_{k=1}^n\phi(\xi_k)\  (x_k-x_{k-1})$ and $\int_a^b\psi=\sum_{k=1}^n\psi(\xi_k)\  (x_k-x_{k-1})$. At this point, we have also proven that $\phi \leq f \leq \psi$, which implies $\phi(\xi_k) \leq f(\xi_k) \leq \psi(\xi_k)$ for all $k\in\{1,2,\ldots,n\}$. Multiplying every member by the positive number $(x_k - x_{k-1})$, and taking the summation over all $k$, we obtain
\begin{eqnarray}
\int_a^b \phi \leq \Riemann(f, \Delta, \xi) \leq \int_a^b \psi, \nonumber
\end{eqnarray}
which, in conjunction with \eqref{Riemann3}, results to $I - \varepsilon < \Riemann(f, \Delta, \xi) < I + \varepsilon$, or that $\Riemann(f, \Delta, \xi)\in\lpar I-\varepsilon,I+\varepsilon\rpar$, where $\Riemann(f, \Delta, \xi)$ is an arbitrary element of $\Sigma_\Delta\in\IntFilter(a,b;f)$, so $\Sigma_\Delta\sub\lpar I-\varepsilon,I+\varepsilon\rpar$. Therefore, $\IntFilter(a,b;f)$ approaches $I$.
\end{proof}

\begin{proposition}\label{DegenerateIntProp} $\int_a^af=0$.
\end{proposition}
\begin{proof} On the degenerate interval $[a,a]$, the integral of any step function is zero, so the sets in the definition of the upper and lower Darboux integrals is just the singleton $\{0\}$, with suprema zero. 

Thus, $\aaIntLow f=0=\aaIntUp f$ for any function $f$ on the degenerated interval. By Lemma~\ref{UAStoRITLem}, $\int_a^af=0$.
\end{proof}

Given a function $f:X\into\R$ and $A\sub X$, the function $\left.f\right|_{A}:A\into\R$ given by $x\mapsto f(x)$ is the \emph{restriction} of $f$ to $A$.

\begin{proposition}\label{StepIntAdditivityProp} If $c\in\abClosed$ and $\varphi\in\abSteps$, then $\left.\varphi\right|_{[a,c]}\in\acSteps$ and $\left.\varphi\right|_{[c,b]}\in\cbSteps$. [We henceforth define $\int_a^c\varphi:=\int_a^c\left.\varphi\right|_{[a,c]}$ and $\int_c^b\varphi:=\int_c^b\left.\varphi\right|_{[c,b]}$.] Furthermore, $\int_a^b\varphi=\int_a^c\varphi+\int_c^b\varphi$.
\end{proposition}
\begin{proof}
If $c=a$ or $c=b$, then one of $\int_a^c\varphi$ or $\int_c^b\varphi$ is zero and the other is $\int_a^b\varphi$, and we are done. We assume henceforth that $c\in(a,b)$. By the definition of step function, there exists a partition $\Delta=\{x_0,x_1,\ldots,x_n\}$ of $[a,b]$ and there exist real numbers $M_1,M_2,\ldots, M_n$ such that for any $k\in\{1,2,\ldots,n\}$,
\begin{align}
\{M_k\} &= \varphi((x_{k-1},x_k)), \label{additivity1} \\
\int_a^b\varphi &= \sum_{k=1}^n M_k(x_k-x_{k-1}). \label{additivity2}
\end{align}
By Proposition~\ref{PartitionCupProp}, $[a,b]=\bigcup_{k=1}^n[x_{k-1},x_k]=\Delta\cup\bigcup_{k=1}^n(x_{k-1},x_k)$, so there exists $N\in\{1,2,\ldots,n\}$ such that either $c=x_N$ or $c\in(x_{k-1},x_k)$. Since $c\neq b=x_n$, for the case $c=x_N$, the integer $N$ is at most $n-1$, so $N+1$ is one of the integers $\{1,2,\ldots,n\}$. Also, for the case $c=x_N$,
\begin{align}
\{M_N\} &= \varphi((x_{N-1},x_{N})) = \left.\varphi\right|_{[a,c]}((x_{N-1},c)), \label{additivity3} \\
\{M_{N+1}\} &= \varphi((x_{N},x_{N+1})) = \left.\varphi\right|_{[c,b]}((c,x_{N+1})), \label{additivity4}
\end{align}
while for the case $c\in(x_{N-1},x_N)$, we have $x_{N-1}<c<x_N$, which implies $(x_{N-1},c)\subset(x_{N-1},x_N)$ and $(c,x_N)\subset(x_{N-1},x_N)$. Thus,
\begin{align*}
\left.\varphi\right|_{[a,c]}((x_{N-1},c)) &= \varphi((x_{N-1},c)) \sub \varphi((x_{N-1},x_N)) = \{M_N\}, \\
\left.\varphi\right|_{[c,b]}((c,x_N)) &= \varphi((c,x_N)) \sub \varphi((x_{N-1},x_N)) = \{M_N\},
\end{align*}
but since $\varphi$ is constant on the interval $(x_{N-1},x_N)$, we find that
\begin{equation}
\left.\varphi\right|_{[a,c]}((x_{N-1},c)) = \{M_N\} = \left.\varphi\right|_{[c,b]}((c,x_N)). \label{additivity3A}
\end{equation}
The set $\{x_0,x_1,\ldots,x_{N-1},c\}$ is a partition of $[a,c]$ where all $k\in\{1,2,\ldots,N-1\}$ satisfy \eqref{additivity1}, and at $k=N$, either \eqref{additivity3} or \eqref{additivity3A} is true. Thus, $\left.\varphi\right|_{[a,c]}$ is a step function, with
\begin{equation}
\int_a^c\left.\varphi\right|_{[a,c]} = \sum_{k=1}^{N-1}M_k(x_k-x_{k-1}) + M_N(c-x_{N-1}). \label{additivity5}
\end{equation}
We also have the partition $\{c,x_N,x_{N+1},\ldots,x_n\}$ of $[c,b]$, for which, at the first subinterval, \eqref{additivity4} or \eqref{additivity3A} is true, while all $k\in\{N+1,N+2,\ldots,n\}$ satisfy \eqref{additivity1}. Consequently, $\left.\varphi\right|_{[c,b]}$ is a step function, and
\begin{equation}
\int_c^b\left.\varphi\right|_{[c,b]} = M_N(x_N-c) + \sum_{k=N+1}^{n}M_k(x_k-x_{k-1}). \label{additivity6}
\end{equation}
Adding \eqref{additivity5} and \eqref{additivity6}, and noting that $M_N(c-x_{N-1}) + M_N(x_N-c) = M_N(x_N-x_{N-1})$, we obtain by substitution into \eqref{additivity2} the equation $\int_a^b\varphi = \int_a^c\left.\varphi\right|_{[a,c]} + \int_c^b\left.\varphi\right|_{[c,b]}$.
\end{proof}

\begin{proposition}\label{RiemannUniqueProp} If $f$ is Riemann integrable over $\abClosed$, then $\int_a^bf$ is the unique real number such that for any $\varphi,\psi\in\abSteps$, if $\varphi\leq f\leq \psi$, then $\int_a^b\varphi\leq \int_a^bf\leq \int_a^b\psi$.
\end{proposition}
\begin{proof} The Riemann integrability of $f$ over $\abClosed$ means that $\abIntUp f=\abIntLow f$, so by Lemma~\ref{UAStoRITLem}, $\int_a^bf$ exists and $\int_a^bf=\abIntUp f=\abIntLow f$. Suppose $I\in\R$ with the property that
\begin{enumerate}
    \item[\Star] for any $\varphi,\psi\in\abSteps$, if $\varphi\leq f\leq \psi$, then $\int_a^b\varphi\leq I\leq \int_a^b\psi$,
\end{enumerate} 
in which, $\int_a^b\varphi$ may be considered as an arbitrary element of $\theSet_1:=\left\{\int_a^b\phi\  :\  \phi\in\abSteps,\  \phi\leq f\right\}$, and so, the number $I$ is an upper bound of $\theSet_1$, and is related to the supremum $\abIntLow f=\sup\theSet_1$ by the inequality $\abIntLow f\leq I$. Let $\theta\in
\abSteps$ such that $\theta\leq -f$, or that $f\leq -\theta$. By \Star, $I\leq \int_a^b(-\theta)$, and by Proposition~\ref{negStepIntProp}, $I\leq-\int_a^b\theta$, so $\int_a^b\theta\leq -I$, where the left-hand side may be considered as an arbitrary element of\linebreak $\theSet_2:=\left\{\int_a^b\phi\  :\  \phi\in\abSteps,\  \phi\leq -f\right\}$, so $-I$ is an upper bound of $\theSet_2$, and is related to the supremum $-\abIntUp f=\sup\theSet_2$ by the inequality $-\abIntUp f\leq -I$, or that $I\leq \abIntUp f$. At this point, we have shown that $\abIntLow f\leq I\leq \abIntUp f=\abIntLow f=\int_a^bf$. Therefore, $I=\int_a^bf$.
\end{proof}

\begin{proposition}\label{mfMIntProp} If $f$ is Riemann integrable over $\abClosed$, and if there exist constants $m,M\in\R$ such that $m\leq f\leq M$ on $\abClosed$, then $m(b-a)\leq\int_a^bf\leq M(b-a)$.
\end{proposition}
\begin{proof} By Proposition~\ref{negStepIntProp}, the constant functions $m,M$ are in $\abSteps$. Using the hypothesis $m\leq f\leq M$ and Proposition~\ref{RiemannUniqueProp}, $\int_a^bm\leq\int_a^bf\leq \int_a^bM$, where the integrals of the step functions may be computed using Proposition~\ref{RiemannUniqueProp}. The result is $m(b-a)\leq\int_a^bf\leq M(b-a)$.
\end{proof}

\begin{lemma}\label{IntAdditivityLem} If $f$ is Riemann integrable over any closed subinterval of $\abClosed$, then for any $c\in\abClosed$, we have $\int_a^bf=\int_a^cf+\int_c^bf$.
\end{lemma}
\begin{proof} Let $\varphi,\psi:\abSteps$ such that $\varphi\leq f\leq \psi$ on $\abClosed$, and these function inequalities hold, in particular, on $[a,c]$ and on $[c,b]$. By assumption, the upper and lower Darboux integrals of $f$ exist and are equal in each of these intervals, so by Proposition~\ref{RiemannUniqueProp}, we have $\int_a^c\varphi\leq \int_a^cf\leq\int_a^c\psi$ and $\int_c^b\varphi\leq \int_c^bf\leq\int_c^b\psi$. Adding these inequalities, $$\int_a^c\varphi+\int_c^b\varphi\leq \int_a^cf+\int_c^bf\leq\int_a^c\psi+\int_c^b\psi,$$
where the left-most member and the right-most member may be replaced according to Proposition~\ref{StepIntAdditivityProp} to obtain $\int_a^b\varphi\leq \int_a^cf+\int_c^bf\leq\int_a^b\psi$. At this point, we have proven that the real number $\int_a^cf+\int_c^bf$ has the property that, for any $\varphi,\psi:\abSteps$, the condition $\varphi\leq f\leq \psi$ implies $\int_a^b\varphi\leq \int_a^cf+\int_c^bf\leq\int_a^b\psi$. According to Proposition~\ref{RiemannUniqueProp}, the only real number with this property is $\int_a^bf$. 

Therefore $\int_a^bf=\int_a^cf+\int_c^bf$.
\end{proof}

To complete the preliminaries we need for the fifth circle of real analysis principles in Theorem~\ref{FifthCircle}, we define here the notion of antiderivative, which is simply just the dual of derivative. The assertion that $f$ is the derivative of $F$ is equivalent to the assertion that $F$ is an \emph{antiderivative} of $f$. The articles (the definite ``the'' versus the indefinite ``an'') are not preserved in this duality: the derivative is unique, but antiderivatives are not. A very simple example from calculus is that $x\mapsto 2x$ is the derivative of $x\mapsto x^2$, but both $x\mapsto x^2$ and $x\mapsto x^2+1$ are antiderivatives of $x\mapsto 2x$.

\subsection{The Archimedean Principle}\label{APPrelims}

The Archimedean Principle has its origins from geometry: as an axiom in Euclidean geometry. One intuitive depiction of the statement is the following. If a straight race track is $x>0$ units long, and a runner's step is only $y>0$ units long, then there is an integer number $N>\frac{x}{y}$ of steps the runner can make to reach the end of the track. [We do not need the equality $N=\frac{x}{y}$, (which might be impossible, like when $\frac{x}{y}$ is not an integer) for the distance the runner shall cover  need only exceed the length of the track at the $N$th step, or perhaps earlier.] The other popular form of the Archimedean Principle that involves not two positive real numbers, but an arbitrary real number, is related to cofinal subsets of $\R$, a notion that we shall define later. First, we introduce a sort of exotic notion.

A nonnegative real number $\varepsilon$ is an \emph{infinitesimal} if $n\in\N$ implies $\varepsilon\leq \frac{1}{n}$. Zero is an infinitesimal.

\begin{proposition}\label{dxProp} Any sum of finitely many infinitesimals is an infinitesimal.
\end{proposition}
\begin{proof} We use induction on the number $m\in\N$ of infinitesimals of which the sum is taken. If $m=1$, then the desired statement is trivially true. Suppose that for some $m\in\N$, the sum of any $m$ infinitesimals is an infinitesimal. A sum of $m+1$ infinitesimals may be denoted by $S+\varepsilon$ where $\varepsilon$ is an infinitesimal, and $S$ is a sum of $m$ infinitesimals, which, by the inductive hypothesis is an infinitesimal. Let $n\in\N$. Since both $S$ and $\varepsilon$ are infinitesimals, $S\leq \frac{1}{n}$ and $\varepsilon\leq\frac{1}{n}$, so $S+\varepsilon\leq\frac{1}{2n}$, where the right-hand side is less than $\frac{1}{n}$ because $n<2n$. That is, $S+\varepsilon\leq\frac{1}{n}$ for any $n\in\N$, or that $S+\varepsilon$ is an infinitesimal. By induction, we get the desired result.
\end{proof}

\begin{lemma}\label{APLem}
The following are equivalent.

\begin{adjustwidth}{6em}{}
\begin{itemize}
    \item[\AP] \emph{Archimedean Principle.} For any $x>0$ and any $y>0$, there exists $N\in\N$ such that $N>\frac{x}{y}$.
    \item[\APN] The sequence $\nseq$ is not bounded above. 
    \item[\API] There are no positive infinitesimals in $\R$.
\end{itemize}
\end{adjustwidth}
\end{lemma}

\begin{proof}$\AP \implies \APN$. Given $M\in\R$, we have $|M|+1>0$, and since $1>0$, by the \AP, there exists $N\in\N$, or a term in the sequence $\nseq$, such that $N>\frac{|M|+1}{1}=|M|+1>|M|\geq M$, or that $N>M$, so $N{\not\leq }M$. Therefore, the sequence $\nseq$ is not bounded above. \\

\noindent $\APN\implies\API$. We prove this by contraposition. If there exists a positive infinitesimal $\varepsilon$ in $\R$, then for any $n\in\N$, we have $\varepsilon\leq\frac{1}{n}$, where both sides are positive, so $n\leq \frac{1}{\varepsilon}$, and $\nseq$ is bounded above.\\

\noindent $\API\implies\AP$. Let $x>0$ and $y>0$. Thus, $\frac{y}{x}$ is positive. If any $N\in\N$ has the property that $N\leq \frac{x}{y}$, where both sides are positive, then $\frac{y}{x}\leq \frac{1}{N}$, so $\frac{y}{x}$ is a positive infinitesimal, contradicting \API. Hence, there exists $N\in\N$ such that $N>\frac{x}{y}$.
\end{proof}

\begin{lemma}\label{convIIILem} As a consequence of \AP, $\seqlimOp\frac{1}{2^n}=0$.
\end{lemma}
\begin{proof} We first prove by induction that $\frac{1}{2^n}\leq \frac{1}{n}$ for all $n\in\N$. At $n=1$, we simply have $\txthalf\leq 1$. If $\frac{1}{2^n}\leq \frac{1}{n}$ is true for some $n\in\N$, then $n\leq 2^n$, and $2n\leq 2^{n+1}$, where the left-hand side is equal to $n+n$, where the second term is at least $1$, so $n+1\leq n+n$. Thus, $n+1\leq 2^{n+1}$, which leads to $\frac{1}{2^{n+1}}\leq \frac{1}{n+1}$. This completes the induction.

Given $\varepsilon>0$, by the \AP, there exists $N\in\N$ such that $N>\frac{1}{\varepsilon}$, or equivalently $\frac{1}{N}<\varepsilon$. If $n\in\N$ with $n\geq N$, then $\frac{1}{n}\leq\frac{1}{N}$, where the left-hand side is, as proven earlier, at least $\frac{1}{2^n}$.

Hence, $\left|\frac{1}{2^n}-0\right|=\left|\frac{1}{2^n}\right|=\frac{1}{2^n}\leq \frac{1}{n}\leq \frac{1}{N}<\varepsilon$. Therefore, $\seqlimOp\frac{1}{2^n}=0$.
\end{proof}

\begin{lemma}\label{supseqLem} As a consequence of \AP, if $c=\sup\theSet$, then there exists a sequence in $\theSet$ that converges to $c$.
\end{lemma}
\begin{proof} For each $n\in\N$, we have $\frac{1}{n}>0$, so $c-\frac{1}{n}<c$. By Proposition~\ref{xiProp}, there exists $c_n\in\theSet$ such that $c-\frac{1}{n}<c_n\leq c<c+\frac{1}{n}$. Thus, $-\frac{1}{n}<c_n-c<\frac{1}{n}$, or that $|c_n-c|<\frac{1}{n}$. We have thus constructed a sequence $\cseq$ in $\theSet$ such that for all $n\in\N$, we have $|c_n-c|<\frac{1}{n}$.

Let $\varepsilon>0$. By the \AP, there exists a positive integer $N>\frac{1}{\varepsilon}$, or equivalently, $\frac{1}{N}<\varepsilon$. If $n\in\N$ with $n\geq N$, then $|c_n-c|<\frac{1}{n}\leq \frac{1}{N}<\varepsilon$. Therefore, $\seqlimOp c_n=c$.
\end{proof}

\begin{lemma}\label{UCTtoUASLem} As a consequence of the \AP, if $a<b$, then for each $\delta>0$, there exists a partition\linebreak $\Delta=\{x_0,x_1,\ldots,x_n\}$ of $\abClosed$ such that for each $k\in\{1,2,\ldots,n\}$, we have $x_{k}-x_{k-1}<\delta$.
\end{lemma}
\begin{proof} From $a<b$, we find that $b-a>0$. Since $\delta>0$, by the \AP, there exists an integer $n>\frac{b-a}{\delta}$, or that $\frac{b-a}{n}<\delta$. For each $k\in\{0,1,\ldots,n\}$, let $x_k:=a+k\frac{b-a}{n}$. Thus, $x_0=a+0=a$, and $x_n=a+(b-a)=b$. If $k\in\{1,2,\ldots,n\}$, then $x_{k}-x_{k-1}=\frac{b-a}{n}\lpar k -(k-1)\rpar=\frac{b-a}{n}<\delta$.
\end{proof}

A subset $\theSet$ of $\R$ is said to be \emph{cofinal} (with respect to the ordering $>$) if for each $x\in\R$, there exists $N\in\theSet$ such that $N>x$.

\begin{proposition} The following are equivalent.

\begin{adjustwidth}{6em}{}
\begin{enumerate}
    \item[\CC] \emph{Countable Cofinality of $\R$.} There exists a countably infinite cofinal subset of $\R$.
    \item[\CCN] There exists a sequence in $\R$ that is not bounded above.
\end{enumerate}
\end{adjustwidth}
\end{proposition}
\begin{proof}$\CC\implies\CCN$. Suppose $\R$ has a countably infinite cofinal subset $\theSet$. Since $\theSet$ is countably infinite, there exists a bijection $\N\into\theSet$, which we denote here by $n\mapsto c_n$. Thus, the set of all terms of the sequence $\cseq$ is equal to $\theSet$. Since $\theSet$ is cofinal, given $M\in\R$, there exists $c_N\in\theSet$ such that $c_N> M$, or that $c_N{\not\leq}M$. Therefore, $\cseq$ is not bounded above.\\

\noindent$\CCN\implies\CC$. Suppose there exists a sequence $\cseq$ in $\R$ that is not bounded above. Let\linebreak $\theSet:=\{c_n\  :\  n\in\N\}$. Given $x\in\R$, and since $\cseq$ is not bounded above, there exists $c_N\in\theSet$ such that $c_N{\not\leq} x$, or that $c_N>x$. This proves that $\theSet$ is cofinal, and since the elements of $\theSet$ may be indexed by $\N$, we find that $\theSet$ is countable. If $\theSet$ has only a finite number, say $N$, of elements, and if we denote these, WLOG, by $x_1<x_2<\cdots<x_N$, then by cofinality, there exists $x_{N+1}\in\theSet$ such that $x_1<x_2<\cdots<x_N<x_{N+1}$ so $\theSet$ has both $N$ and $N+1$ elements.$\lightning$ Therefore, $\theSet$ is countably infinite.
\end{proof}

\begin{lemma}\label{UAStoCCLem} $\AP\implies\CC$.
\end{lemma}
\begin{proof} By the \AP, given $x\in\R$, we have $|x|+1>0$, and since $1>0$, there exists $N\in\N$ such that $N>\frac{|x|+1}{1}>|x|\geq x$. Thus, $\N$ is a cofinal subset of $\R$ (which is also countably infinite).
\end{proof} 

The converse of Lemma~\ref{UAStoCCLem} is not within the scope of this exposition.\\ \\

At this point, we have supplied the reader with all the necessary definitions and lemmas to execute the proofs of the implications in the five circles of real analysis principles: convergence, connectedness, differentiability, compactness and integration. We start with convergence.

\section{The First Circle: Convergence}

The theory of convergent sequences is an important foundation of real function theory. In fact, it has been claimed that elementary real analysis may be presented  using sequences (and the closely related notion of series) as the foundation. See, for instance, \cite{lit15}. This exposition is obviously not about only the sequential approach to real analysis, and also not about using only one approach at all. Nevertheless, the importance of sequences remains undisputed and the key principles governing convergent sequences of real numbers comprise the first circle of real analysis theorems. These principles are the following.\\
\begin{adjustwidth}{6em}{}
\begin{itemize}
    \item[\CCC] \emph{Cauchy Convergence Criterion.} Every Cauchy sequence is convergent.
    \item[\sNIP] \emph{Strong Nested Intervals Principle.} If $\Iseq$ is a sequence of nested intervals, then $\bigcap_{k=1}^\infty I_k\neq\emptyset$. 
    \item[\wNIP] \emph{Weak Nested Intervals Principle.} If $\Iseq$ is a sequence of nested intervals and\linebreak $\seqlimOpIII \ell(I_k)=0$, then $\bigcap_{k=1}^\infty I_k\neq\emptyset$.
\end{itemize}
\end{adjustwidth}
Two other principles are stated in the theorem below. The reason for the distinction is that each of the above principles have to be joined with the Archimedean Principle of \AP\  for them to be included in the list of equivalences. See Section~\ref{APPrelims} for preliminary material on the \AP. The more fundamental preliminaries about sequences of real numbers may be found in Section~\ref{FirstCirclePrelims}, and the propositions and lemmata in there shall be used in the proof of Theorem~\ref{FirstCircle}.

\begin{theorem}\label{FirstCircle}
The following statements are equivalent:\\
\begin{adjustwidth}{6em}{}
\begin{itemize}
    \item[\AP\ \emph{\&} \CCC] [The conjunction of the statements of the Archimedean Principle and the Cauchy Convergence Criterion.]
    \item[\MCP] \emph{Monotone Convergence Principle.} A monotonically increasing sequence that is bounded above is convergent.
    \item[\AP\ \emph{\&} \sNIP] [The conjunction of the statements of the Archimedean Principle and the Strong Nested Intervals Principle.]
    \item[\AP\ \emph{\&} \wNIP] [The conjunction of the statements of the Archimedean Principle and the Weak Nested Intervals Principle.]
    \item[\BWP]\emph{Bolzano-Weierstrass Property of $\R$.} Every bounded sequence has a convergent subsequence.
\end{itemize}
\end{adjustwidth}
\end{theorem}

\begin{proof}$\AP\&\CCC\implies\MCP$. Suppose $\cseq$ is a monotonically increasing sequence that is bounded above. In view of the \CCC, showing that $\cseq$ is Cauchy shall suffice. Suppose otherwise. By Lemma~\ref{CCCtoMCPLem}, there exist $\varepsilon>0$ and a subsequence $\csubseq$ of $\cseq$ such that 
\begin{enumerate}
    \item[\Star] for each $k\in\N$, we have $c_{N_{k+1}}\geq c_{N_1}+k\varepsilon$.
\end{enumerate} Since $\cseq$ bounded above there exists $M>0$ such that 
\begin{enumerate}
    \item[\SStar] for all $n\in\N$, we have $c_{n}\leq M$ . 
\end{enumerate} In particular, at $n=N_1$, we have $c_{N_1}\leq M<M+1$, which implies $M+1-c_{N_1}>0$. Since $\varepsilon>0$, by the \AP, there exists $K\in\N$ such that $K>\frac{M+1-c_{N_1}}{\varepsilon}$, which implies $c_{N_1}+K\varepsilon>M+1>M$, where the left-most member, according to \Star, is at most $c_{N_{K+1}}$. Thus, $c_{N_{K+1}}>M$, but setting $n=N_{K+1}$ in \SStar\ results to $c_{N_{K+1}}\leq M$.$\lightning$\\

\noindent $\MCP\implies\AP$. If \AP\  is false, then by Lemma~\ref{APLem}, the monotonically increasing sequence $\nseq$ is bounded, and by the \MCP, convergent, which is necessarily Cauchy. Since $\txthalf>0$, this means that there exists $N\in\N$ such that for all $m,n\in\N$ with $m,n\geq N$, we have $|m-n|<\txthalf$. In particular, at $m=N+1$ and $n=N$, we have $|m-n|=|N+1-N|=1<\txthalf$.$\lightning$\\

\noindent $\MCP\implies\sNIP$. Let $\Iseq$ be a sequence of nested intervals. By Lemma~\ref{MCPtoNIPLem}, the sequence $\aseq$ of left endpoints of the intervals $I_k$ is monotonically increasing and bounded above, so by the \MCP, converges to some $c\in\R$. Also by Lemma~\ref{MCPtoNIPLem}, for any $k,n\in\N$, we have $a_k\leq c\leq b_n$. Setting $k=n$, for any $k\in\N$, we have $c\in I_k$. Hence, $c\in\bigcap_{k=1}^\infty I_k$, and this intersection is nonempty.\\

\noindent $\AP\&\sNIP\implies\AP\&\wNIP$. Showing $\sNIP\implies\wNIP$ shall suffice, but this follows from the fact these two statements have the same conclusion, and that the hypothesis of $\sNIP$ is one of the hypotheses in $\wNIP$.\\

\noindent $\AP\&\wNIP\implies\BWP$. Suppose $\cseq$ is bounded. By Lemma~\ref{NIPtoBWPLemII}, there exists a sequence $\Iseq$ of nested intervals and a subsequence $\csubseq$ of $\cseq$ such that for each $k\in\N$, we have\linebreak $c_{N_k}\in I_k$ and $\ell(I_k)=\frac{b-a}{2^{k-1}}$. By Lemma~\ref{convIIILem}, the \AP\  implies  $\seqlimOpIII\frac{1}{2^k}=0$, and by Lemma~\ref{NIPtoBWPLemInew}\ref{convII},\linebreak $\seqlimOpIII\ell(I_k)=2(b-a)\seqlimOpIII\frac{1}{2^k}=0$. By the \wNIP, there exists $x\in\bigcap_{k=1}^\infty I_k$. At this point, we have proven that for each $k\in\N$, we have $c_{N_k},x\in I_k$ and $\seqlimOpIII \ell(I_k)=0$. By the nonnegativity of absolute value and Proposition~\ref{NIPtoBWPProp}, for each $k\in\N$, we have $0\leq |c_{N_k}-x|\leq \ell(I_k)$ and $\seqlimOpIII \ell(I_k)=0$. By Lemma~\ref{NIPtoBWPLemInew}, these conditions imply $\seqlimOpIII c_{N_k}=x$. Therefore, $\cseq$ has a convergent subsequence.\\

\noindent $\BWP\implies\AP$. If \AP\  is false, then by Lemma~\ref{APLem}, $\nseq$ is bounded, and by the \BWP, has a convergent subsequence $\lpar n_{N_k}\rpar_{k\in\N}$, which is necessarily Cauchy. Since $\txthalf>0$, this means that there exists $K\in\N$ such that for all $h,k\in\N$ with $h,k\geq K$, we have $\left|n_{h}-n_{k}\right|<\txthalf$. In particular, at $h=K+1$ and $k=K$, we have $\left|n_{K+1}-n_{K}\right|<\txthalf$. Because $n_{K+1}$ and $n_{K}$ are integers, so is $\left|n_{K+1}-n_{K}\right|$. By the definition of subsequence, $K<K+1$ implies $N_K<N_{K+1}$, and since $\nseq$ is strictly increasing, $n_{K}<n_{K+1}$. Thus, we have a positive integer $n_{K+1}-n_{K}=\left|n_{K+1}-n_{K}\right|$ strictly less than $\txthalf$.$\lightning$\\

\noindent $\BWP\implies\CCC$. A Cauchy sequence $\cseq$, by Proposition~\ref{CauchyBoundProp}, is bounded, and by the \BWP, has a convergent subsequence $\csubseq$. By Lemma~\ref{BWPtoCCCLem}, $\cseq$ converges to $\seqlimOpIII c_{N_k}$.
\end{proof}

\section{The Second Circle: Connectedness}

The topological notion of connectedness is a property of a closed and bounded interval on which the Intermediate Value Theorem depends, and the Intermediate Value Theorem is one of the most important in real function theory. The second circle of real analysis principles is about connectedness. The preliminaries needed to understand the proof of Theorem~\ref{SecondCircle} below may be found in Section~\ref{SecondCirclePrelims}. It is also in Section~\ref{SecondCirclePrelims} that we introduced the standard topology on $\R$, and this reflects how involved the argumentation is concerning topology. The first and last statements in the following list of equivalence also belong to the previous circle of real analysis principles in Theorem~\ref{FirstCircle}, and this shows the equivalence of all the theorems in the first two circles. In fact, the five circles of theorems are not pairwise disjoint, and the theorems that appear in more than one circle serve the purpose of connecting all these theorems logically, where the result is that every theorem in the list is equivalent to the Dedekind Completeness Axiom (usually used as the main Completeness Axiom for elementary real analysis).

\begin{theorem}\label{SecondCircle} The following statements are equivalent.\\
\begin{adjustwidth}{6em}{}
\begin{itemize}
    \item[\AP\ \emph{\&} \wNIP] [The conjunction of the statements of the Archimedean Principle and the Weak Nested Intervals Principle.]
    \item[\ES] \emph{Existence of Suprema.} Every nonempty set of real numbers that has an upper bound has a least upper bound.
    \item[\UIC] The unit interval $[0,1]$ is connected.\footnote{In \cite{rie01}, the connectedness of $[0,1]$ was shown to be equivalent (with a short proof) to the connectedness of every path-connected metric space, and so the latter is another entry in the list of the statements that form the Second Circle. It seems that developing metric spaces in the preliminaries of this exposition is too much for the purpose of including the extra statement in the Second Circle, which has a short proof anyway of equivalence to \UIC. Hence, the Second Circle statement that ``Path connected metric spaces are connected," is not included anymore in this exposition.}
    \item[\NTD] The set $\R$ is not totally disconnected.
    \item[\AAC] All intervals are connected.
    \item[\IVT] \emph{Intermediate Value Theorem.} Given a continuous function $f:\abClosed\into\R$, if\linebreak $f(a)<k<f(b)$, then there exists $c\in\abClosed$ such that $f(c)=k$.
    \item[\RIC] The set $\R$ is connected.
    \item[\CA] \emph{Dedekind's Cut Axiom.} Every cut of $\R$ is not a gap.
        \item[\MCP] \emph{Monotone Convergence Principle.} A monotonically increasing sequence that is bounded above is convergent.
\end{itemize}
\end{adjustwidth}
\end{theorem}
\begin{proof} $\AP\&\wNIP\implies\ES$. Let $\theSet\sub\R$, and suppose that $b_1$ is an upper bound of $\theSet$. By Lemma~\ref{NIPtoESLem}, we have two cases, the first one of which, is already the desired conclusion. Thus, what remains to be shown is that we obtain the desired conclusion also from the second case, which is that there exists a sequence $\Iseq$ of nested intervals such that for any $k\in\N$, the left endpoint $a_k$ of $I_k$ is an element of $\theSet$, the right endpoint $b_k$ of $I_k$ is an upper bound but not an element of $\theSet$, and $\ell(I_k)=\frac{b_1-a_1}{2^{k-1}}$. 

By Lemma~\ref{convIIILem}, the \AP\  implies $\seqlimOpIII\frac{1}{2^k}=0$, and by Lemma~\ref{NIPtoBWPLemInew}\ref{convII}, we have:

\noindent$\seqlimOpIII\ell(I_k)=2(b_1-a_1)\seqlimOpIII\frac{1}{2^k}=0$, so by the \wNIP, there exists $c\in\bigcap_{k=1}^\infty I_k$. Since, for each $k\in\N$, the endpoints $a_k$ and $b_k$ of $I_k$ are in $I_k$, we further have $a_k,b_k,c\in I_k$ and $\seqlimOpIII \ell(I_k)=0$. By the nonnegativity of absolute value and Proposition~\ref{NIPtoBWPProp}, for each $k\in\N$, we have $0\leq |a_k-c|\leq \ell(I_k)$, $0\leq |b_k-c|\leq \ell(I_k)$ and $\seqlimOpIII \ell(I_k)=0$. By Lemma~\ref{NIPtoBWPLemInew}, these conditions imply $\seqlimOpIII a_k=c=\seqlimOpIII b_k$.

Since the intervals concerned are nested, for each $k\in\N$, we have $b_{k+1}\in\lbrak a_{k+1},b_{k+1}\rbrak\sub\lbrak a_{k},b_{k}\rbrak$, which implies $b_{k+1}\leq b_{k}$. Hence, $\bseq$ is monotonically decreasing. If $x\in\theSet$, since every term in $\bseq$ is an upper bound of $\theSet$, we have $x\leq b_k$. By Corollary~\ref{decCor}, we have $x\leq c$, and this proves that $c$ is an upper bound of $\theSet$.

Suppose there is an upper bound $b$ of $\theSet$ such that $b<c$. Thus, $c-b>0$, and since we have shown earlier that $\seqlimOpIII a_k=c$, there exists $K\in\N$ such that for all indices $k\geq K$, we have $|c-a_k|=|a_k-c|<c-b$. This is true, in particular, at $k=K$. Thus, $|c-a_K|<c-b$, where the left-hand side is at least $c-a_K$, so $c-a_K<c-b$, which implies $a_K>b$, but as established earlier, $a_K\in\theSet$. We have thus shown that $b$ is not an upper bound of $\theSet$.$\lightning$ Therefore, any upper bound $b$ of $\theSet$ satisfies $c\leq b$, and this completes the proof that $c=\sup\theSet$.\\

\noindent $\ES\implies\UIC$. Let $U$ be a subset of $[0,1]$ that contains $0$ and is both open and closed relative to $[0,1]$. Our goal is to show $U=[0,1]$ so that by Lemma~\ref{ConnectedLem}, $[0,1]$ is connected. Since $U$ is already a subset of $[0,1]$, showing $[0,1]\sub U$ shall suffice.

Define $\theSet:=\{x\in[0,1]\  :\  [0,x]\sub U\}$. Since $0\in U$, we have $[0,0]=\{0\}\sub U$, and also $0\in[0,1]$. Thus, $0\in\theSet$, and $\theSet$ is hence nonempty. Each $x\in\theSet$ satisfies $x\in[0,1]$, which implies $x\leq 1$. Thus, $1$ is an upper bound of $\theSet$. By \ES, there exists $c\in\R$ such that $c=\sup\theSet$. Every $x\in\theSet$ satisfies $[0,x]\sub U$, where the set in the left-hand side contains $x$, so $x\in U$, and this proves that $\theSet\sub U$. By Proposition~\ref{supCloseProp}, $c\in S^-\sub U^-$. Since $U$ is closed relative to $[0,1]$, there exists a closed set $F$ such that $U=F\cap [0,1]$, so $U$ is an intersection of closed sets and is hence closed. Thus, $U^-=U$, and $c\in U^-$ further becomes $c\in U$. Since $U$ is open relative to $I$, there exists an open set $G$ such that $U=G\cap [0,1]$, so $c\in U=G\cap [0,1]\sub G$. Since $G$ is open, there exists an interval $\abOpen$ such that $c\in\abOpen\sub G$. From $c\in\abOpen$, we obtain $a<c$, and by Propostion~\ref{xiProp}, there exists $\xi\in\theSet$ such that $a<\xi\leq c$. From $\xi\in\theSet$, we obtain $[0,\xi]\sub U$ and $\xi\in[0,1]$, and from $c\in\abOpen\sub G$, we obtain $\lpar a,c\rbrak\sub G$. By Proposition~\ref{supIntProp}, we also have $c\in[0,1]$.

Let $x\in\lbrak\xi,c\rbrak$, or that $\xi\leq x\leq c$. Earlier we have shown $\xi,c\in[0,1]$, so $0\leq \xi\leq x\leq c\leq 1$, which implies $x\in[0,1]$, and we have proven $\lbrak\xi,c\rbrak\sub[0,1]$. From $x\in\lbrak\xi,c\rbrak$ and $a<\xi\leq c$ and $\xi\leq x\leq c$, we find that $a<x\leq c$, so $x\in\lpar a,c\rbrak\sub G$. This proves $\lbrak\xi,c\rbrak\sub G$. Combining this with $\lbrak\xi,c\rbrak\sub[0,1]$, we have $\lbrak\xi,c\rbrak\sub G\cap[0,1]=U$. At this point, we have obtained $[0,\xi]\sub U$ and $[\xi,c]\sub U$, so $[0,c]=[0,\xi]\cup[\xi,c]\sub U$. 

Now that we have $[0,c]\sub U$, what remains to be shown is that $c=1$. Earlier, we proved $c\in[0,1]$, which implies $ c\leq 1$, so to complete the proof we show that the case $c<1$ leads to a contradiction. Earlier, we have shown that there is an interval $\abOpen$ such that $c\in\abOpen\sub G$. Since $a<c<\frac{c+b}{2}<b$, we further have $c\in\lbrak c,\frac{c+b}{2}\rbrak\sub G$. Since $c\in[0,1]$, we further have $c\in\lbrak c,\frac{c+b}{2}\rbrak\cap[0,1]\sub G\cap[0,1]=U$. If $y:=\min\left\{\frac{c+b}{2},1\right\}$, then $0\leq c<y\leq 1$ (which implies $y\in[0,1]$) and any $x\in\lbrak c,y\rbrak$ satisfies $c\leq x\leq \frac{c+b}{2}$ and $0\leq c\leq x\leq 1$. This proves $\lbrak c,y\rbrak\sub\lbrak c,\frac{c+b}{2}\rbrak\cap[0,1]\sub U$. Earlier, we showed $[0,c]\sub U$, so $[0,y]=[0,c]\cup[c,y]\sub U$, with $y\in[0,1]$ and $y>c$. That is, we have produced an element $y$ of $\theSet$ that is bigger than the upper bound $c$ of $\theSet$.$\lightning$ Therefore, $c=1$, the inclusion $[0,c]\sub U$ becomes $[0,1]\sub U$, and this completes the proof that $[0,1]$ is connected.\\

\noindent $\UIC\implies\NTD$. By \UIC, there exists a subset of $\R$, which is $[0,1]$, that is connected and is not a singleton. Therefore, $\R$ is not totally disconnected.\\

\noindent $\NTD\implies\AAC$. By \NTD, there exists a connected set $I_0\sub \R$ that contains more than one element. Thus, there exist $a_0,b_0\in I_0$, for which, WLOG, we further assume $a_0<b_0$. By Proposition~\ref{conNecProp}, $\lbrak a_0,b_0\rbrak\sub I$, and by Proposition~\ref{conNecPropII}, the interval $\lbrak a_0,b_0\rbrak$ is connected.

Let $\abClosed$ be an arbitrary closed and bounded interval. Since $a_0<b_0$, by Proposition~\ref{homeoProp}, there exists a continuous function $f:\lbrak a_0,b_0\rbrak\into \abClosed$ such that the image of $\lbrak a_0,b_0\rbrak$ under $f$ is equal to $\abClosed$, so by Proposition~\ref{contImProp}, $\abClosed$ is connected. That is, we have now proven that any closed and bounded interval is connected. If there is an interval $I$ that is not connected, then by Proposition~\ref{conIntProp}, there exists a closed and bounded interval (contained in $I$) that is not connected.$\lightning$ Therefore, any interval is connected.\\

\noindent $\AAC\implies\IVT$. Tending towards a contradiction, suppose that for any $c\in I:=\abClosed$, we have $f(c)\neq k$. By the Trichotomy Law, $f(c)<k$ or $k<f(c)$, so we now have $f[I]\sub (-\infty,k)\cup(k,\infty)$. 

Taking the intersection of both sides with $f[I]$, we obtain $f[I]\sub A\cup B$, where\linebreak $A:=(-\infty,k)\cap f[I]\sub f[I]$ and $B:=f[I]\cap (k,\infty)\sub f[I]$, so both $A$ and $B$ are open relative to $f[I]$. Also, $A\cap B=\lpar(-\infty,k)\cap  (k,\infty)\rpar\cap f[I]=\emptyset\cap f[I]=\emptyset$. Since the right-hand side of $f[I]\sub A\cup B$ is a union of subsets of $f[I]$, this union is also a subset of $f[I]$. Hence, $f[I]=A\cup B$. From $f(a)\in f[I]$ and $f(a)<k$, we find that $f(a)\in (-\infty,k)\cap f[I]=A$, while $k<f(b)$ and $f(b)\in f[I]$ imply $f(b)\in f[I]\cap(k,\infty)$. Thus, $A$ and $B$ are nonempty. At this point, we have shown that $f[I]$ is not connected. But by \AAC, $I$ is connected, and by Proposition~\ref{contImProp}, so is $f[I]$.$\lightning$ Therefore, there exists $c\in\abClosed$ such that $f(c)=k$.\\

\noindent$\IVT\implies\RIC$. Suppose $\R$ is not connected. This means that there exist nonempty disjoint open sets $X$ and $Y$ of $\R$ such that $\R=X\cup Y$. Since $X$ and $Y$ are nonempty, there exist $a\in X$ and $b\in Y$, and by the disjointness of $X$ and $Y$, we further have $a\neq b$. WLOG, we assume $a<b$.

If $I:=\abClosed$, then the sets $A:=X\cap I$ and $B:=Y\cap I$ are both open relative to $ I$. Taking the intersection with $I$ of both sides of $\R=X\cup Y$, we obtain $ I=A\cup B$, which implies $A=I\setdiff B$ and $B=I\setdiff A$.  By Corollary~\ref{closedRelCor}, $A$ and $B$ are closed relative to $I$. Also, $A\cap B=(X\cap Y)\cap I=\emptyset\cap I=\emptyset$. By Example~\ref{contEx}, the constant function $A\into\R$ defined by $x\mapsto 0$ is continuous on $A$, and the constant function $B\into\R$ given by $x\mapsto 1$ is continuous on $B$. By the Pasting Lemma, the function $f:I\into\R$ defined by 
\begin{equation}
    f(x):=\begin{cases}
        0, & x\in A,\\
        1, & x\in B,
    \end{cases}\nonumber
\end{equation}
is continuous on $A\cup B=I$. Furthermore, $f[I]=\{0,1\}$.

The conditions $a\in X$ and $a\in\abClosed=I$ imply $a\in X\cap I=A$, while $b\in Y$ and $b\in\abClosed=I$ imply $b\in B$. Thus, $f(a)=0<\txthalf<1=f(b)$. By the \IVT, there exists $c\in\abClosed$ such that $\txthalf=f(c)\in f[I]=\{0,1\}$.$\lightning$ Therefore, $\R$ is connected.\\

\noindent $\RIC\implies\CA$. Suppose \RIC\  is true but \CA\  is false, which means that $\R$ has a cut $A$ that is a gap. Thus, $A=\bigcup_{a\in A}(-\infty,a)$ and $A^c=\bigcup_{b\in A^c}(b,\infty)$, which are hence both open sets (open relative to $\R$). By the definition of cut, the cut $A$ is nonempty, and is a proper subset of $\R$, so $A^c$ is nonempty. Also, $A$ and $A^c$ are disjoint, with $\R=A\cup A^c$. Therefore, $\R$ is not connected, contradicting \RIC. Therefore, no cut of $\R$ is a gap.\\

\noindent $\CA\implies\MCP$. Suppose $\cseq$ is a monotonically increasing sequence bounded above by $M$. Let $A:=\{x\in\R\  :\  \exists n\in\N\quad x\leq c_n\}$. Since $\cseq$ is bounded above by $M$, for any $n\in\N$, we have $c_n\leq M<M+1$. This proves that $M+1\notin A$, or that $A$ is a proper subset of $\R$. Since $\cseq$ is monotonically increasing, for any $n\in\N$, we have $c_n\leq c_{n+1}$, where $n+1\in\N$. Thus, all terms of $\cseq$ are in  $A$, so $A$ is not empty. Since $A^c:=\{x\in\R\  :\  \forall n\in\N\quad x> c_n\}$, for any $a\in A$ and any $b\in A$, there exists $N\in\N$ such that $a\leq c_N<b$, which implies $a\leq c_N<b$. Thus, $A$ is a cut of $\R$. By \CA, $A$ is not a gap and hence has a cut point $c$.

Let $\varepsilon>0$. If, for any $n\in\N$, we have $c_n\leq c-\varepsilon$, then, using the fact that $c-\varepsilon<c-\varepsilon+\txthalf\varepsilon=c-\txthalf\varepsilon$, we find that $n\in\N$ implies $c_n<c-\txthalf\varepsilon$. This means that $c-\txthalf\varepsilon\in A^c$. But since $c$ is a cut point of $A$, this implies $c\leq c-\txthalf\varepsilon$, which implies $0\leq \varepsilon$. $\lightning$ Hence, there exists $N\in\N$ such that $c-\varepsilon<c_N$. Given any integer $n\geq N$, since $\cseq$ is monotonically increasing, $c-\varepsilon<c_N\leq c_n$. Since all terms of $\cseq$ are in $A$ and since $c$ is a cut point of $A$, we further have $c-\varepsilon< c_n\leq c$, where the right-most member is smaller than $c+\varepsilon$. We now have $c-\varepsilon<c_n<c+\varepsilon$, which implies $-\varepsilon<c_n-c<\varepsilon$, so $|c_n-c|<\varepsilon$. Therefore, $\seqlimOp c_n=c$, and this proves \MCP.\\

\noindent $\MCP\implies\AP\&\wNIP$. See Theorem~\ref{FirstCircle}.
\end{proof}

\section{The Third Circle: Differentiability}

In this third circle of real analysis theorems, we shall consider those that are directly about the Differential Calculus. The derivative is a limit, and in this exposition, we decided to introduce limits in terms of the topological notion of a filter base. See Section~\ref{ThirdCirclePrelims} for all the preliminaries needed for this circle of real analysis theorems. Similar to the previous circle of principles, the first and last real analysis principle in the following list also belong to previous circles, and this shows the equivalence of all theorems we have considered so far.

\begin{theorem}\label{ThirdCircle} The following are equivalent.\\
\begin{adjustwidth}{6em}{}
\begin{enumerate}
        \item[\BWP\ \emph{\&} \ES] [The conjunction of the statements of the Bolzano-Weierstrass Property of $\R$ and the Existence of Suprema.]
       \item[\EVT] \emph{Extreme Value Theorem.} Given a continuous function $f:\abClosed\into\R$, there exists $c\in\abClosed$ such that for all $x\in \abClosed$, we have $f(x)\leq f(c)$.
       \item[\RT] \emph{Rolle's Theorem.} If $f$ is continuous on $\abClosed$, differentiable on $\abOpen$, and\linebreak $f(a)=0=f(b)$, then there exists $c\in\abOpen$ such that $f'(c)=0$.
       \item[\eMVT] \emph{Extended Mean Value Theorem.} If $f,g:\abClosed\into\R$ are continuous on $\abClosed$, differentiable on $\abOpen$, and if $g'$ is nonzero on $\abOpen$, then there exists $c\in\abOpen$ such that $\frac{f'(c)}{g'(c)}=\frac{f(b)-f(a)}{g(b)-g(a)}$.
       \item[\MVT] \emph{Mean Value Theorem.} If $f:\abClosed\into\R$ is continuous on $\abClosed$ and differentiable on $\abOpen$, then there exists $c\in\abOpen$ such that $f'(c)=\frac{f(b)-f(a)}{b-a}$.
       \item[\TT] \emph{Taylor's Theorem (with Lagrange Remainder).} Let $n$ be a nonnegative integer, and let $f$ be a function $\abClosed\into\R$. Suppose all of $f^{(0)}$, $f^{(1)}$, $f^{(2)}$, $\ldots$~, $f^{(n)}$ are continuous on $\abClosed$, and that $f^{(n)}$ is differentiable on $\abOpen$. For any $x\in\lpar a,b\rbrak$, there exists $c\in\lpar a,x\rpar$ such that $$f(x)=\lbrak\sum_{k=0}^n\frac{f^{(k)}(a)}{k!}(x-a)^k\rbrak+\frac{f^{(n+1)}(c)}{(n+1)!}(x-a)^{n+1}.$$
        \item[\PCP] \emph{Polynomial Characterization Property.} Let $n$ be a nonnegative integer, and let $f$ be a function $\abClosed\into\R$. Suppose all of $f^{(0)}$, $f^{(1)}$, $f^{(2)}$, $\ldots$~, $f^{(n)}$ are continuous on $\abClosed$, and that $f^{(n)}$ is differentiable on $\abOpen$. If $f^{(n+1)}$ is zero on $\abOpen$, then $f$ is a polynomial of degree at most $n$.
        \item[\CFT] \emph{Convex Function Theorem.} If $f''$ is nonnegative on $\abOpen$, then $a<c<x<b$ implies that for all $t\in[c,x]$, we have $f(t)\leq f(c)+(t-c)\frac{f(x)-f(c)}{x-c}$ [or that $f$ is \emph{convex} on $\abOpen$].
        \item[\IFT] \emph{Increasing Function Theorem.} If $f'$ is nonnegative on $\abOpen$, then $a<c<x<b$ implies $f(c)\leq f(x)$ [or that $f$ is \emph{(monotonically) increasing} on $\abOpen$].
       \item[\CVT] \emph{Constant Value Theorem.} If $f'$ is zero on $\abOpen$, then $f$ is constant on $\abOpen$.
        \item[\UIC] The unit interval $[0,1]$ is connected.
\end{enumerate}
\end{adjustwidth}
\end{theorem}
\begin{proof} $\BWP\&\ES\implies\EVT$. Let $I:=\abClosed$. Tending towards a contradiction, suppose $f[I]$ does not have an upper bound. Thus, $1$ is not an upper bound of $f[I]$ so there exists $x_1\in f[I]$ such that $1<x_1$. Again, $\max\{2,x_1\}$ is not an upper bound of $f[I]$ so there exists $x_2\in f[I]$ such that $\max\{2,x_1\}<x_2$. Suppose that for some $n\in\N$ at least $2$, we have identified $x_n\in f[I]$ such that\linebreak $\max\{n,x_{n-1}\}<x_n$. Since $\max\{n+1,x_n\}$ is not an upper bound of $f[I]$, there exists $x_{n+1}\in f[I]$ such that $\max\{n+1,x_n\}<x_{n+1}$. By induction, we have identified a sequence $\xseq$ in $f[I]$ such that for all $n\in\N$, we have $\max\{n,x_{n-1}\}<x_n$. This property also applies at $n+1$, so $\max\{n+1,x_{n}\}<x_{n+1}$, which implies $x_n<x_{n+1}$. Thus, $\xseq$ is strictly increasing. Also, for any $n\in\N$, we have $n<x_n$. From Theorem~\ref{FirstCircle}, the \BWP\  implies \AP, which, by Lemma~\ref{APLem}, implies \APN. That is, $\nseq$ is not bounded above. Given $M>0$, there exists $n\in\N$ such that $M<n<x_n$. Thus, $\xseq$ is also not bounded above. 

Since $\xseq$ is a sequence in $f[I]$, for each $n\in\N$, there exists $t_n\in I$ such that $x_n=f(t_n)$. The sequence $\tseq$ is a sequence in $I=\abClosed$ or is a bounded sequence. By the \BWP, it has a convergent subsequence $\tsubseq$. This is also a sequence in $I=\abClosed$, so $\tsubseq$ is bounded above by $b$. Also, the fact that $\tsubseq$ is a sequence in the closed set $\abClosed$, by Proposition~\ref{limCloseProp}, implies that\linebreak $t:=\seqlimOpIII t_{N_k}\in\abClosed^-=\abClosed$. Since $f$ is continuous on $I$, by Proposition~\ref{contSeqProp}, the sequence\linebreak $\lpar f\lpar t_{N_k}\rpar\rpar_{k\in\N}=\xsubseq$ is convergent, is necessarily bounded, and is, furthermore, bounded above. Since $\tsubseq$ is a subsequence of $\tseq$, the conditions on the indices $N_k$ from the definition of subsequence also apply when these indices are used in the sequence $\xsubseq$. Thus, $\xsubseq$ is a subsequence of $\xseq$. At this point, we have produced a subsequence of $\xseq$ that is bounded above. However, earlier, we proved that $\xseq$ is strictly increasing and not bounded above, so by Proposition~\ref{strictsubProp}, any subsequence of $\xseq$ is not bounded above.$\lightning$

Henceforth, the set $f[I]$ has an upper bound. It is nonempty because it contains $f(a)$. By \ES, there exists $y\in\R$ such that $y=\sup f[I]$. By Lemma~\ref{supseqLem}, there exists a sequence $\sseq$ in $f[I]$ such that $\seqlimOp s_n=y$. Since $\sseq$ is a sequence in $f[I]$, for each $n\in\N$, there exists $c_n\in I=\abClosed$ such that $s_n=f(c_n)$. Thus, $\cseq$ is bounded, and by the \BWP, has a convergent subsequence $\csubseq$. Let $c=\seqlimOpIII c_{N_k}$. Since $\csubseq$ is a sequence in $\abClosed$, which is a closed set, by Proposition~\ref{limCloseProp},\linebreak $c\in\abClosed^-=\abClosed$. By Proposition~\ref{contSeqProp}, we have $\seqlimOpIII f(c_{N_k})=f(c)$. Since, for all $k\in\N$, we have\linebreak $s_{N_k}=f(c_{N_k})$, we have just proven that $\ssubseq$ is also convergent, with:

\noindent$\seqlimOpIII s_{N_k}=\seqlimOpIII f(c_{N_k})=f(c)$. The conditions (from the definition of subsequence) on the indices $N_k$ apply to both $\ssubseq$ and $\csubseq$, so $\ssubseq$ is a subsequence of $\sseq$, and by Lemma~\ref{limsubLem}, $f(c)=\seqlimOpIII s_{N_k}=\seqlimOp s_n=y=\sup f[I]$. At this point, we have produced $c\in I$ such that $f(c)=\sup f[I]$, which is an upper bound of $f[I]$, or that for any $x\in I$, we have $f(x)\leq f(c)$.\\

\noindent $\EVT\implies\RT$. If $f$ is zero (and hence constant) on $\abClosed$, then by Proposition~\ref{AlgDiffProp}\ref{AlgConst}, $f'$ is zero on $\abClosed$, and in particular on $\abOpen\sub\abClosed$, which implies the desired conclusion.

For the case when $f$ has a nonzero value at some $\xi\in\abClosed$, we cannot have $\xi=a$ or $\xi =b$. Otherwise, since $f$ is a function, we have $f(\xi)=f(a)$ or $f(\xi)=f(b)$, where the right-hand sides, by assumption are zero, contradicting $f(\xi)\neq 0$. Henceforth, $\xi\in\abOpen$.

Since $f$ is continuous on $\abClosed$, by the \EVT, there exists $c\in\abClosed$ such that 
\begin{enumerate}
    \item[\Star] for any $x\in\abClosed$, we have $f(x)\leq f(c)$.
\end{enumerate}

For the case $f(\xi)> 0$, we have $0< f(\xi)\leq f(c)$, so $f(c)>0$. If $c=a$ or $c=b$, then:

\noindent$0<f(c)=f(a)=0\lightning$ or $0<f(c)=f(b)=0\lightning$. Thus, $c$ cannot be the endpoints of $\abClosed$, and we further have $c\in\abOpen$, and on this interval $f$ is assumed to be differentiable. That is, $f$ is differentiable at $c$. From \Star, we find that any $x\in\abOpen\sub\abClosed$ satisfies $f(x)\leq f(c)$, or equivalently, $f[\abOpen]\sub\lpar-\infty,f(c)\rbrak$. By Lemma~\ref{EVTtoRTLem}, $f'(c)=0$.

The remaining case is $f(\xi)<0$, for which $(-f)(\xi)>0$, where $-f$, by Lemma~\ref{negfunctionLem}, is continuous on $\abClosed$ and differentiable on $\abOpen$. Thus, all premises for the previous case are satisfied by the function $-f$, so $(-f)'(c)=0$. Using Proposition~\ref{AlgDiffProp}\ref{DiffLinear}, $-1\cdot f'(c)=0$, which implies $f'(c)=0$.\\

\noindent $\RT\implies\eMVT$. Define $F:\abClosed\into\R$ by 
\begin{eqnarray}
    F(x) &=& \lbrak g(b)-g(a)\rbrak\lbrak f(x)-f(a)\rbrak-\lbrak g(x)-g(a)\rbrak\lbrak f(b)-f(a)\rbrak,\label{eMVTeq1}\\
    &=&c_1f(x)+c_2g(x)+c_3,\label{eMVTeq2}
\end{eqnarray}
where $c_1=g(b)-g(a)$, $c_2=f(a)-f(b)$ and $c_3=f(b)g(a)-f(a)g(b)$. Since $f$, $g$ and the constant function $\mapsto c_3$ are all continuous on $\abClosed$ and differentiable on $\abOpen$, by Propositions~\ref{AlgLimProp}\ref{AlgConst},\ref{AlgLinear} and \ref{AlgDiffProp}\ref{AlgConti},\ref{DiffConst},\ref{DiffLinear}, so is $F$, with $F'$, using \eqref{eMVTeq2}, is given by $F'=c_1f'+c_2g'$. In \eqref{eMVTeq1}, setting $x=a$ makes the right-hand side $0-0$, while setting $x=b$ makes the right-hand side the difference of two equal quantities. Thus, $F(a)=0=F(b)$. By \RT, there exists $c\in\abOpen$ such that $F'(c)=0$, or that $c_1f'(c)+c_2g'(c)=0$. By assumption, $g'(c)\neq 0$, so we further have $c_1\frac{f'(c)}{g'(c)}=-c_2$. Substituting the definitions of $c_1$ and $c_2$, we obtain $[g(b)-g(a)]\frac{f'(c)}{g'(c)}=f(b)-f(a)$. To complete the proof, we show why we can divide both sides of this equation by $g(b)-g(a)$, or equivalently, that $g(b)-g(a)\neq 0$.

Tending towards a contradiction, suppose $g(b)-g(a)= 0$. Define $G:\abClosed\into\R$ by the equation $G(x)=g(x)-g(a)=g(x)-g(b)$. Since $g$ and the constant function $x\mapsto g(a)$ are continuous on $\abClosed$ and differentiable on $\abOpen$, by Propositions~\ref{AlgLimProp}\ref{AlgConst},\ref{AlgLinear} and \ref{AlgDiffProp}\ref{AlgConti},\ref{DiffConst},\ref{DiffLinear}, so is $G$, with the derivative as $G'(x)=g'(x)-0=g'(x)$, or that $G'=g'$. Also, $G(a)=0=G(b)$. By \RT, there exists $k\in\abOpen$ such that $g'(k)=G'(k)=0$, contradicting the assumption that $g'$ is nonzero on $\abOpen$. This completes the proof.\\

\noindent $\eMVT\implies\MVT$. By Propositions~\ref{AlgLimProp}\ref{AlgId} and \ref{AlgDiffProp}\ref{DiffId}, the function $g$ defined by $g(x)=x$ is continuous on $\abClosed$, differentiable on $\abOpen$, and with derivative given by $g'(x)=1$. By the \eMVT, there exists $c\in\abOpen$ such that $\frac{f'(c)}{1}=\frac{f(b)-f(a)}{g(b)-g(a)}$, or that $f'(c)=\frac{f(b)-f(a)}{b-a}$.\\

\noindent $\MVT\implies\TT$ By Lemma~\ref{MVTtoTTLem}, given $x\in\lpar a,b\rbrak$, there exists a unique $\rho(x)\in\R$ such that the function $F:\lbrak a,x\rbrak\into\R$ defined by
\begin{eqnarray}
    F(t)=f(x)-\sum_{k=0}^n\frac{f^{(k)}(t)}{k!}(x-t)^k- \rho(x) \frac{(x-t)^{n+1}}{(n+1)!},\label{TAYLOReqDEF}
\end{eqnarray}
is continuous on $[a,x]$ and differentiable on $(a,x)$. Furthermore, $F(a)=0$ and 
\begin{eqnarray}
    F'(t) = \frac{(x-t)^n}{n!} \lbrak \rho(x) - f^{(n+1)}(t) \rbrak.\label{TAYLOReq}
\end{eqnarray}
From \eqref{TAYLOReqDEF}, we also find that at $x=t$, we have $F(x)=f(x)-f(x)+0=0$. By the \MVT, there exists $c\in(a,x)$ such that $F'(c)=\frac{F(x)-F(a)}{x-a}=\frac{0-0}{x-a}=0$. Substituting $0$ to the left-hand side of \eqref{TAYLOReq} and $c$ in the right-hand side, $0=(x-c)^n[\rho(x)-f^{(n+1)(c)}]$, where, since $c\in(a,x)$, we have $c<x$, so $x-c\neq 0$, and $0=(x-c)^n[\rho(x)-f^{(n+1)(c)}]$ further becomes $\rho(x)=f^{(n+1)}(c)$, which we substitute to \eqref{TAYLOReqDEF} together with $t=a$. The resulting left-hand side becomes $F(a)$, which, as said earlier, is zero. By rearranging terms,

$$f(x)=\lbrak\sum_{k=0}^n\frac{f^{(k)}(a)}{k!}(x-a)^k\rbrak+\frac{f^{(n+1)}(c)}{(n+1)!}(x-a)^{n+1}.$$\\

\noindent$\TT\implies\PCP\&\CFT\&\IFT$. In the statement of \TT, if we add the assumption that $f^{(n+1)}$ is zero, then we may substitute $f^{(n+1)}(c)=0$ in the equation to obtain $f(x)=\sum_{k=0}^n\frac{f^{(k)}(a)}{k!}(x-a)^k$, so $f$ is a polynomial of degree at most $n$, and this proves the \PCP.

For the proofs of the \CFT\ and \IFT, in the statement of \TT, let $n\in\{0,1\}$. The hypotheses of the \CFT\  and \IFT, may be generalized as ``$f^{(n+1)}$ is nonnegative on $\abOpen$,'' where $n=0$ for the \IFT, and $n=1$ for the \CFT. Consequently, for all $k\in\{0,\ldots,n\}$, the function $f^{(k)}$ is differentiable on $\abOpen$, and by Proposition~\ref{AlgDiffProp}\ref{DiffConti}, is continuous on $\abOpen$, and in particular, if $a<c<\tau<x<b$, then $f^{(k)}$ is continuous on the subsets $[c,x]$, $[c,\tau]$ and $[\tau,x]$ of $\abOpen$. Also, the differentiability of $f^{(k)}$ on $\abOpen$ also implies differentiability on the subsets $(c,x)$, $(c,\tau)$ and $(\tau,x)$ of $\abOpen$. By Proposition~\ref{AlgDiffProp}\ref{DiffLinear}, for the case $n=1$, the functions $(-f)^{(0)}=-f$, $(-f)^{(1)}=-f'$ are continuous on $[c,x]$, $[c,\tau]$ and $[\tau,x]$ and are differentiable on $(c,x)$, $(c,\tau)$ and $(\tau,x)$. We now use \TT\  as described in the following.
\begin{enumerate}
    \item[\Star] If $n=0$, then by \TT, for any $t\in(c,x]$, there exists $\xi\in(c,t)$ such that $\frac{f(t)-f(c)}{t-c} =f'(\xi)$.
    \item[\SStar] If $n=1$, then using \TT\  on the function $-f$, for any $t\in(c,\tau]$, there exists $\xi\in(c,\tau)$ such that $(-f)(t)-(-f)(c)-(t-c)(-f)'(c) =\frac{(-f)''(\xi)}{2}(t-c)^2=\frac{f''(\xi)}{2}(t-c)^2$.
    \item[\SSStar] Also for $n=1$, using \TT\  on $f$, for any $t\in(\tau,x]$, there exists $\zeta\in(\tau,x)$ such that\linebreak $f(t)-f(\tau)-(t-\tau)f'(\tau) =\frac{f''(\zeta)}{2}(t-\tau)^2$.
\end{enumerate}

If $n=0$, then the nonnegativity of $f^{(n+1)}=f'$ applied to \Star\ means that $\frac{f(t)-f(c)}{t-c}\geq 0$, which is true, in particular, at $t=x\in(c,x]$. That is, $\frac{f(x)-f(c)}{x-c}\geq 0$, both sides of which, we multiply by $x-c$, which is positive because $a<c<x<b$. We obtain $f(x)-f(c)\geq 0$, or that $f(c)\leq f(x)$, and this proves the \IFT.

For the case $n=1$, we set $t=\tau$ in \SStar\ and $t=x$ in \SSStar\ to deduce that there exist $\xi\in(c,\tau)$ and $\zeta\in(\tau,x)$ such that
\begin{eqnarray}
    (-f)(\tau)-(-f)(c)-(\tau-c)(-f)'(c) &=&\frac{f''(\xi)}{2}(\tau-c)^2,\label{newCFTeq1}\\
    f(x)-f(\tau)-(x-\tau)f'(\tau) &=&\frac{f''(\zeta)}{2}(x-\tau)^2.\label{newCFTeq2}
\end{eqnarray}
Since $c<\tau<x$, we find that $\frac{(\tau-c)^2}{2}$ and $\frac{(x-\tau)^2}{2}$ are positive. In conjunction with the hypothesis of the \CFT\  that $f''$ is nonnegative, which implies that $f''(\xi)$ and $f''(\zeta)$ are nonnegative, we find that the right-hand sides of \eqref{newCFTeq1}--\eqref{newCFTeq2} are nonnegative. As a result,
\begin{eqnarray}
    (-f)(\tau)-(-f)(c)-(\tau-c)(-f)'(c) &\geq & 0,\label{newCFTeq3}\\
    f(x)-f(\tau)-(x-\tau)f'(\tau) &\geq &0.\label{newCFTeq4}
\end{eqnarray}
We multiply both sides of \eqref{newCFTeq3} by $-1$, and isolate $f'(c)$. We also isolate $f'(\tau)$ in \eqref{newCFTeq4}. These result to
\begin{eqnarray}
    \frac{f(\tau)-f(c)}{\tau-c} &\leq & f'(c),\label{newCFTeq5}\\
    \frac{f(x)-f(c)}{x-c} &\geq &f'(\tau).\label{newCFTeq6}
\end{eqnarray}
Since the \IFT\  has been proven true, the condition that $f''$ is nonnegative means that $f'$ is monotonically increasing, so from $c<\tau<x$, we obtain $f'(c)\leq f'(\tau)$, which, in conjunction with \eqref{newCFTeq5}--\eqref{newCFTeq6}, gives us
\begin{flalign}
    && \frac{f(\tau)-f(c)}{\tau-c} &\leq \frac{f(x)-f(c)}{x-c}, &\mbox{if }c<\tau<x,\label{newCFTeq7}
\end{flalign}
where the condition $c<\tau<x$ means that $\tau-c$ is positive, which, when multiplied to both sides of the main inequality in \eqref{newCFTeq7}, results to the desired inequality. 

The complete the proof of the \CFT, the remaining cases are $\tau=c$ and $\tau=x$. If $\tau=c$, the trivial equation $f(c)=f(c)$ leads to $f(c)\leq f(c)+0=f(c)+(c-c)\frac{f(x)-f(c)}{x-c}$, and we get the desired inequality. If $\tau=x$, then $f(x)=f(x)=f(c)+f(x)-f(c)$ may be used instead to obtain $f(x)\leq f(c)+(x-c)\frac{f(x)-f(c)}{x-c}$. This completes the proof of the \CFT.\\

\noindent$\PCP\implies\CVT$. If $f'=f^{(1)}$ is zero on $\abOpen$, then setting $n=0$ in the statement of the \PCP, $f$ is a polynomial of degree at most zero, which is a constant.\\

\noindent$\CFT\implies\CVT$. Suppose $f'$ is zero on $\abOpen$. By Proposition~\ref{AlgDiffProp}, $f''$ and $(-f)''=-1\cdot f''$ are also zero on $\abOpen$. Consequently, $f''$ and $(-f)''$ are nonnegative on $\abOpen$. By the \CFT,
\begin{flalign}
    && f(t) &\leq f(c)+(t-c)\frac{f(x)-f(c)}{x-c},\label{CFTeq1}\\
    && -f(t) &\leq -f(c)+(t-c)\frac{-f(x)-[-f(c)]}{x-c}, &\mbox{if }a<c<x<b,\  t\in[c,x].\label{CFTeq2}
\end{flalign}
Given arbitrary $c,x\in\abOpen$, if $c=x$, since $f$ is a function $f(c)=f(x)$. If $c\neq x$, we assume, WLOG, that $c<x$, so we may use \eqref{CFTeq1}--\eqref{CFTeq2} for any $t\in[c,x]$. Multiplying both sides of \eqref{CFTeq2} by $-1$ and combining with \eqref{CFTeq1}, we have $f(t)=f(c)+(t-c)\frac{f(x)-f(c)}{x-c}$. Consequently, $f(t)=\frac{f(x)-f(c)}{x-c}t+\frac{xf(c)-cf(x)}{x-c}$. Using Proposition~\ref{AlgDiffProp}, $f'(t)=\frac{f(x)-f(c)}{x-c}$, where the left-hand side, by assumption, is zero. That is, $0=\frac{f(x)-f(c)}{x-c}$. Solving for $f(c)$, we obtain $f(c)=f(x)$. That is, for any $c,x\in\abOpen$, we have shown $f(c)=f(x)$. Therefore, $f$ is constant on $\abOpen$.\\

\noindent$\IFT\implies\CVT$. If $f'$ is zero on $\abOpen$, then by Proposition~\ref{AlgDiffProp}, so are $f'$ and $(-f)'=-1\cdot f'$, and are consequently nonnegative on $\abOpen$, and by the \IFT,
\begin{flalign}
    && f(x) &\leq f(c),\label{CFTeq3}\\
    && -f(x) &\leq -f(c), &\mbox{if }a<c<x<b.\label{CFTeq4}
\end{flalign}
Given arbitrary $c,x\in\abOpen$, the fact that $f$ is a function means that $c=x$ implies $f(c)=f(x)$, while if $c\neq x$, which, WLOG, may be further assumed as $c<x$, by \eqref{CFTeq3}--\eqref{CFTeq4}, we have $f(x)\leq f(c)$ and $f(x)\geq f(c)$, so $f(x)=f(c)$. Therefore, $f$ is constant on $\abOpen$.\\

\noindent$\CVT\implies\UIC$. Tending towards a contradiction, suppose $[0,1]$ is not connected. Since\linebreak $[0,1]\sub(-1,2)$, there exists a closed and bounded interval contained in $(-1,2)$ that is not connected. By contraposition of Proposition~\ref{conNecPropII}, $(-1,2)$ is not connected. Thus, there exist nonempty disjoint subsets $A$ and $B$ of $(-1,2)$, both open relative to $(-1,2)$, such that $(-1,2)=A\cup B$. Since $A$ and $B$ are nonempty, there exist $a\in A$ and $b\in B$, and the function $f:(-1,2)\into\R$ defined by 
\begin{equation}
    f(x):=\begin{cases}
        0, & x\in A,\\
        1, & x\in B,
    \end{cases}\nonumber
\end{equation} 
has the property that $f(a)=0$ and $f(b)=1$. Since $A$ and $B$ are both open relative to $(-1,2)$, there exist open sets $G_1$ and $G_2$ such that $A=G_1\cap(-1,2)$ and $B=G_2\cap(-1,2)$. But these finite unions of open sets are open, or that $A$ and $B$ are both open sets, so $A=A^o$ and $B=B^o$. These imply that $a$ [respectively, $b$] and any other element of $A$ [respectively, $B$] are both elements and interior points of $A$ [respectively, $B$], so by our definition of differentiability, it is possible to define $f'$ on $A$ and $B$, for as long as the required limit exists, which does exist because $f$ is constant on $A$ and also on $B$. Furthermore, by Proposition~\ref{AlgDiffProp}\ref{DiffConst}, $f'$ is zero on $A$ and on $B$, or is zero on $A\cup B=(-1,2)$. By the \CVT, the function $f$ is constant on $(-1,2)$, so it must have the same value at $a$ and at $b$. That is, $0=f(a)=f(b)=1$.$\lightning$ Therefore, $[0,1]$ is connected.\\

\noindent$\UIC\implies\BWP\&\ES$. From Theorem~\ref{SecondCircle}, $\UIC\iff\ES$ and $\UIC\implies\MCP$, but from Theorem~\ref{FirstCircle}, $\MCP\implies\BWP$. This completes the proof.
\end{proof}

\section{The Fourth Circle: Compactness}\label{FourthCircleSec}

The topological notion of compactness is said to suggest no intuitive image, but is fruitful \cite[p.~36]{dix84}. Maybe this is because the proofs of theorems that are avoided in calculus, but are required to be done in a first analysis course, call for the use of ``compactness arguments'' \cite[p.~470]{tho07}, which are arguments that invoke a ``compactness theorem'' such as those from previous circles: \ES, \CA, \BWP\  and \wNIP; or the Heine-Borel Theorem that masquerades below as the statement of the compactness of the unit interval $[0,1]$. Because all theorems in the five circles in this exposition are eventually equivalent to Dedekind completeness, every such theorem is essentially a compactness theorem, and the use of any of them in proving a real analysis theorem is a compactness argument. Hence is the fruitfulness of the notion of compactness. However, compactness may also be referred to as a ``gate-keeper topic,'' in the sense that if an undergraduate mathematics student does not understand compactness it is unlikely that such a student will be able to do higher-level mathematics \cite[p.~619]{ram15}. 

The circle of theorems for this section shall involve the following.\\
\begin{adjustwidth}{6em}{}
\begin{itemize}
\item[\LCL] \emph{Lebesgue Covering Lemma.} Every open cover of a closed and bounded interval has a Lebesgue number.
\item[\UCT] \emph{Uniform Continuity Theorem.} If $f$ is continuous on $\abClosed$, then $f$ is uniformly continuous. 
    \item[\BVT] \emph{Bounded Value Theorem.} If $f$ is continuous on $\abClosed$, then there exists $M\in\R$ such that $a\leq x\leq b$ implies $f(x)\leq M$. 
\end{itemize}
\end{adjustwidth}
Similar to the first circle in Theorem~\ref{FirstCircle}, there are two other new theorems in the list below, and the distinction is because each one of the has to be joined with the \AP\  or \CC\  to be equivalent to any other statement in the circle. Also, the first and last statements in the list in Theorem~\ref{FourthCircle} belong to other circles, and again, this implies the equivalence of all theorems so far. The preliminaries for this fourth circle of theorems may be found in Section~\ref{FourthCirclePrelims}.

\begin{theorem}\label{FourthCircle}
The following statements are equivalent:\\
\begin{adjustwidth}{6em}{}
\begin{enumerate}
    \item[\AP\ \emph{\&} \wNIP] [The conjunction of the statements of the Archimedean Principle and the Weak Nested Intervals Principle.]
    \item[\UIK] The unit interval $[0,1]$ is compact.
        \item[\AP\ \emph{\&} \LCL] [The conjunction of the statements of the Archimedean Principle and the Lebesgue Covering Lemma.]
            \item[\AP\ \emph{\&} \UCT] [The conjunction of the statements of the Archimedean Principle and the Uniform Continuity Theorem.]
            \item[\UAS] \emph{Uniform Approximation by Step Functions.} If $f$ is continuous on $\abClosed$, then for each $\varepsilon>0$, there exists a step function $\varphi:\abClosed\into\R$ such that for any $x\in\abClosed$, we have $|f(x)-\varphi(x)|<\varepsilon$.
        \item[\CC\ \emph{\&} \BVT] [The conjunction of the statements of the Countable Cofinality of $\R$ and the Bounded Value Theorem.]
    \item[\MCP] \emph{Monotone Convergence Principle.} A monotonically increasing sequence that is bounded above is convergent.
\end{enumerate}
\end{adjustwidth}
\end{theorem}
\begin{proof} $\AP\&\wNIP\implies\UIK$. Let $\Cover$ be an open cover of $[0,1]$. Tending towards a contradiction, suppose $[0,1]$ is not compact. Thus, $I_1:=[0,1]$ cannot be covered by finitely many sets from $\Cover$. Suppose that for some $k\in\N$, closed and bounded intervals $\lbrak a_{k-1},b_{k-1}\rbrak= I_{k-1}\sub I_{k-2}\sub\cdots\sub I_1$ have been found such that for all positive integers $n<k$, the set $I_n$ cannot be covered by finitely many sets from $\Cover$ and that $\ell(I_n)=\frac{1}{2^{n-1}}$.  If $L_{k-1}:=\lbrak a_{k-1},\frac{a_{k-1}+b_{k-1}}{2}\rbrak$ and $R_{k-1}:=\lbrak\frac{a_{k-1}+b_{k-1}}{2},b_{k-1}\rbrak$, then $I_{k-1}=L_{k-1}\cup R_{k-1}$. If each of $L_{k-1}$ and $R_{k-1}$ may be covered by finitely many sets from $\Cover$, then by Proposition~\ref{UnionFiniteProp}, so is $L_{k-1}\cup R_{k-1}=I_{k-1}$.$\lightning$ Hence, there exists $I_{k}=\lbrak a_{k},b_{k}\rbrak\in\{L_{k-1},R_{k-1}\}$\linebreak (with $I_{k}\sub I_{k-1}$) such that $I_{k}$ cannot be covered by finitely many sets from $\Cover$, with the length given by: $\ell(I_{k})=\frac{b_k-a_k}{2}=\txthalf\ell(I_{k-1})=\frac{1}{2^{k}}$. By induction, we have produced a sequence $\Iseq$ of nested intervals such that for each $k\in\N$, we have $\ell(I_k)=\frac{1}{2^{k-1}}$ and that $I_k$ cannot be covered by finitely many sets from $\Cover$. By Lemma~\ref{convIIILem}, the \AP\  implies  $\seqlimOpIII\frac{1}{2^k}=0$, and by Lemma~\ref{NIPtoBWPLemInew}\ref{convII}, $\seqlimOpIII\ell(I_k)=2\cdot\seqlimOpIII\frac{1}{2^k}=0$. By the \wNIP, there exists $c\in\bigcap_{k=1}^\infty I_k$, where the set in the right-hand side is a subset of any interval in the sequence $\Iseq$, which is a subset of $I_1=[0,1]\sub\bigcup_{U\in\Cover}U$. That is, $c\in \bigcup_{U\in\Cover}U$. Since $\Cover$ consists only of open sets, there exists and open set $G$ such that $c\in G\in\Cover$. Since $G$ is open, there exists $\delta>0$ such that
\begin{enumerate}
    \item[\Star] $|x-c|<\delta$ implies $x\in G$.
\end{enumerate} Since $\delta>0$ and $\seqlimOpIII\ell(I_k)=0$, there exists $K\in\N$ such that for all $k\geq K$, and in particular at $k=K$, we have $\ell(I_K)=\frac{1}{2^{K-1}}=\left|\frac{1}{2^{K-1}}\right|=\left|\frac{1}{2^{K-1}}-0\right|<\delta$. That is, $\ell(I_K)<\delta$. Suppose $x\in [a_K,b_K]:=I_K$. Since, $c\in \bigcap_{k=1}^\infty I_k\sub I_K=[a_K,b_K]$, by Proposition~\ref{NIPtoBWPProp}, $|x-c|\leq \ell(I_K)<\delta$, so by \Star, $x\in G$, and we have proven that $I_K=[a_K,b_K]\sub G\in\Cover$. That is, $I_K$ can be covered by one set (so by finitely many sets) from $\Cover$, contradicting the definition of $\Iseq$. Therefore, $[0,1]$ is compact.\\

\noindent$\UIK\implies\AP$. Tending towards a contradiction, suppose \AP\  is false. By Lemma~\ref{APLem}, the set\linebreak $\Inf:=\left\{\varepsilon>0\  :\  \forall n\in\N\quad \varepsilon\leq\frac{1}{n} \right\}$ of all positive infinitesimals has an element $y$.

Let $c\in[0,1]$. If $c\in\Inf$, then, since $y\in\Inf$, by Proposition~\ref{dxProp}, we have $\delta:=c+y\in\Inf$. The elements $c$, $y$ and $\delta$ of $\Inf$ are all positive, so $-\delta<0<c<c+y=\delta$, and this implies $c\in(-\delta,\delta)\sub\bigcup_{\varepsilon\in\Inf}(-\varepsilon,\varepsilon)$. This also covers the case $c=0$, which, in particular, is in $(-\delta,\delta)$ even if it is not in $\Inf$. If $c\notin\Inf$, then there exists $N\in\N$ such that $c{\not\leq}\frac{1}{N}$. Since $c\in[0,1]$, we further have $\frac{1}{N}<c\leq 1<2$, so $c\in\lpar\frac{1}{N},2\rpar\sub\bigcup_{n=1}^\infty\lpar\frac{1}{n},2\rpar$. At this point, we have proven that
\begin{eqnarray}
    [0,1]\sub\lbrak\bigcup_{\varepsilon\in\Inf}(-\varepsilon,\varepsilon)\rbrak\cup\lbrak\bigcup_{n=1}^\infty\lpar\frac{1}{n},2\rpar\rbrak,
\end{eqnarray}
where the intervals that form the union in the right-hand side of the set inclusion are all open sets. Thus, these intervals form an open cover $\Cover$ of $[0,1]$. By \UIK, $[0,1]$ is compact, so there exist finitely many intervals from $\Cover$ that cover $[0,1]$. Equivalently, there exist finitely many positive infinitesimals $\varepsilon$ and finitely many positive integers $n$ such that the corresponding intervals $(-\varepsilon,\varepsilon)$ and $\lpar\frac{1}{n},2\rpar$ cover $[0,1]$. WLOG, we may denote these as the $M+K$ intervals
\begin{eqnarray}
    \lpar-\varepsilon_1,\varepsilon_1\rpar,\qquad\lpar-\varepsilon_2,\varepsilon_2\rpar, \qquad&\ldots& \qquad\lpar-\varepsilon_M,\varepsilon_M\rpar,\label{New01SubCover0}\\
    \lpar\frac{1}{N_1},2\rpar,\qquad\lpar\frac{1}{N_2},2\rpar,\qquad &\ldots &\qquad\lpar\frac{1}{N_K},2\rpar,\qquad\label{New01SubCover}
\end{eqnarray}
where, for each $m\in\{1,2,\ldots, M\}$ and each $k\in\{1,2,\ldots,K\}$, we have $\varepsilon_m\in\Inf$ and $N_k\in\N$. This is because if less than $M+K$ sets among \eqref{New01SubCover0}--\eqref{New01SubCover} cover $[0,1]$, then the union of these sets is still a subset of the union of all the sets in \eqref{New01SubCover0}--\eqref{New01SubCover}. We also assume, WLOG, that $N_1<N_2<\cdots<N_K$, or that $\frac{1}{N_K}<\cdots<\frac{1}{N_2}<\frac{1}{N_1}$. Otherwise, we only need to re-index the aforementioned positive integers. Because of this ordering, the positive integer $1+N_K$ is bigger than all of $N_1<N_2<\cdots<N_K$, and consequently, the number $\frac{1}{1+N_K}$ is strictly less than the left endpoints of the intervals in \eqref{New01SubCover}, so $\frac{1}{1+N_K}$ is not in any of the intervals in \eqref{New01SubCover}. Given $m\in\{1,2,\ldots, M\}$, since $\varepsilon_m$ is a positive infinitesimal and $1+N_K$ is a positive integer, we have $\varepsilon_m\leq \frac{1}{1+N_K}$, so $\frac{1}{1+N_K}\notin(-\varepsilon_m,\varepsilon_m)$. But since $1+N_K$ is positive and $1+N_K\geq 1$, we have $0<\frac{1}{1+N_K}\leq 1$, which implies $\frac{1}{1+N_K}\in[0,1]$. At this point, we have produced an element of $[0,1]$ that is not in the union of the sets in \eqref{New01SubCover0}--\eqref{New01SubCover}. This contradicts the condition that the sets in \eqref{New01SubCover0}--\eqref{New01SubCover} cover $[0,1]$. This completes the proof.\\

\noindent$\UIK\implies\LCL$. Let $\abClosed$ be an arbitrary closed and bounded interval. By \UIK, the interval $[0,1]$ is compact, and by Lemma~\ref{conNecPropIII}, so is $\abClosed$. If $\abClosed$ is empty, then we use $\delta=1>0$, and the desired conclusion is vacuously true. Henceforth, assume that $\abClosed$ is nonempty. By Lemma~\ref{UIKtoLCLLem}, any open cover of $\abClosed$ has a Lebesgue number.\\

\noindent$\AP\&\LCL\implies\AP\&\UCT$. The implication $\AP\&\LCL\implies\AP$ is immediate, so we prove that $\AP\&\LCL$ implies $\UCT$.

Suppose $f$ is continuous on $\abClosed$, and let $\varepsilon>0$. By Lemma~\ref{LCLtoUCTLem}, there exists an open cover $\Cover$ of $\abClosed$ such that for each $G\in\Cover$, if $x,c\in G$, then $|f(x)-f(c)|<\varepsilon$. By the \LCL, the open cover $\Cover$ has a Lebesgue number $\delta>0$, which means that for any $x,c\in\abClosed$, if $|x-c|<\delta$, then there exists $G\in\Cover$ such that $x,c\in G$, which, by \Star, implies $|f(x)-f(c)|<\varepsilon$. Therefore, $f$ is uniformly continuous.\\

\noindent$\AP\&\UCT\implies\UAS$. Suppose $f$ is continuous on $\abClosed$, and let $\varepsilon>0$. By the \UCT, $f$ is uniformly continuous, so there exists $\delta>0$ such that 
\begin{enumerate}
    \item[\Star] if $x,c\in\abClosed$ and $|x-c|<\delta$, then $|f(x)-f(c)|<\varepsilon$.
\end{enumerate}
By Lemma~\ref{UCTtoUASLem}, there exists a partition there exists a partition $\Delta=\{x_0,x_1,\ldots,x_n\}$ of $\abClosed$ such that
\begin{flalign}
&&    x_{k}-x_{k-1} &<\delta, &\mbox{for each }k\in\{1,2,\ldots,n\}.\label{UASeq}
\end{flalign}
We define $\varphi:\abClosed\into\R$ by $\varphi(b)=\varphi(x_n):=f(x_n)=f(b)$, and for each $k\in\{1,2,\ldots,n\}$, if $x\in [x_{k-1},x_k)$, then $\varphi(x):=f(x_{k-1})$. Thus, for each $k\in\{1,2,\ldots,n\}$, the function $\varphi$ is constant on $[x_{k-1},x_k)$, and in particular, on $(x_{k-1},x_k)\sub[x_{k-1},x_k)$. This means that $\varphi$ is a step function. Let $x\in\abClosed$. By Proposition~\ref{PartitionCupProp}, $x\in\bigcup_{k=1}^n[x_{k-1},x_{k}]$, so there exists $K\in\{1,2,\ldots,n\}$ such that $x\in[x_{K-1},x_{K}]$. Thus, $x_{K-1}\leq x\leq x_K$, which implies $0\leq x-x_{K-1}\leq x_K-x_{K-1}$, so by \eqref{UASeq}, $0\leq x-x_{K-1}<\delta$, where the first inequality means that $x-x_{K-1}$ is nonnegative and is hence equal to its absolute value. That is, $|x-x_{K-1}|<\delta$. Since $x,x_{K-1}\in\abClosed$, by \Star, we have $|f(x)-f(x_{K-1})|<\varepsilon$, where, by the definition of $\varphi$, we may replace $f(x_{K-1})$ by $\varphi(x)$ if $x\in\lbrak x_{K-1},x_K\rpar$. That is, $|f(x)-\varphi(x)|<\varepsilon$. If $x=x_K$, then either $x=b$ [if $K=n$] or [if $K<n$, then] $x$ is the left endpoint of the next interval $[x_{K},x_{K+1})$. In both of these cases, $f(x)=f(x_K)=\varphi(x_K)=\varphi(x)$, so $|f(x)-\varphi(x)|=0<\varepsilon$. In all cases, $|f(x)-\varphi(x)|<\varepsilon$, and this completes the proof. \\

\noindent$\UAS\implies\CC$. Tending towards a contradiction, suppose \CC\  is false. By contraposition of Lemma~\ref{UAStoCCLem}, \AP\  is false, and by Lemma~\ref{APLem}, there exists an infinitesimal $\varepsilon>0$. By Propositions~\ref{AlgLimProp}\ref{AlgId} and \ref{AlgDiffProp}\ref{AlgConti}, the function $f:x\mapsto x$ is continuous on $[0,1]$, and by \UAS, there exists a step function $\varphi:\abClosed\into\R$ such that for any $x\in[0,1]$, we have $|f(x)-\varphi(x)|<\varepsilon$, and since $f:x\mapsto x$, this further becomes $|x-\varphi(x)|<\varepsilon$. That is,
\begin{flalign}
    && |x-\varphi(x)|&<\varepsilon, &\mbox{for all }x\in[0,1].\label{UAStoCCeq}
\end{flalign}
By Proposition~\ref{StepImageProp}, there exist $c_1$, $c_2$, \ldots , $c_N\in\R$ such that
\begin{eqnarray}
    \{\varphi(x)\  :\  x\in[0,1]\}=\{c_1,c_2,\ldots,c_N\}.\label{UAStoCCeq2}
\end{eqnarray}
Given $x\in[0,1]$, by \eqref{UAStoCCeq}, we have $|x-\varphi(x)|<\varepsilon$, which impies $x\in(\varphi(x)-\varepsilon,\varphi(x)+\varepsilon)$, and by \eqref{UAStoCCeq2}, there exists $K\in\{1,2,\ldots,N\}$ such that $\varphi(x)=c_K$, and so $x\in(c_K-\varepsilon,c_K+\varepsilon)$. At this point, we have proven that
\begin{eqnarray}
    [0,1]\sub\bigcup_{k=1}^n(c_k-\varepsilon,c_k+\varepsilon),\nonumber
\end{eqnarray}
which, by Proposition~\ref{CoverLengthProp}, implies that $1=1-0<\sum_{k=1}^n2\varepsilon=2n\varepsilon$, which, by Proposition~\ref{dxProp} is an infinitesimal, so for any $n\in\N$, we have $1<2n\varepsilon\leq\frac{1}{n}$, which implies $1<\frac{1}{n}$, which is false at $n=1$. Therefore, \CC\  is true.\\

\noindent$\UAS\implies\BVT$. If $f$ is continuous on $\abClosed$, then for $1>0$, by \UAS, there exists a step function $\varphi:\abClosed\into\R$ such that for each $x\in\abClosed$, we have $|f(x)-\varphi(x)|<1$. By Proposition~\ref{StepImageProp}, there exist $c_1$, $c_2$, \ldots , $c_N\in\R$ such that $\{\varphi(x)\  :\  x\in[0,1]\}=\{c_1,c_2,\ldots,c_N\}$. Thus, given $x\in\abClosed$, there exists $K\in\{1,2,\ldots,N\}$ such that $\varphi(x)=c_K$, so $|f(x)-\varphi(x)|<1$ becomes $|f(x)-c_K|<1$, or equivalently, $f(x)\in(c_K-1,c_K+1)$. Consequently, $f(x)<c_K+1\leq M:=\max\{c_k+1\  :\  k\in\{1,2,\ldots,N\}\}$.\\

\noindent$\CC\&\BVT\implies \MCP$. Tending towards a contradiction, suppose there exists a monotonically increasing sequence $\aseq$ in $\R$ that is bounded above by $b\in\R$ but does not converge. By \CC, there exists a sequence $\cseq$ in $\R$ that is not bounded above. If $I_1:=\lpar-\infty,a_1\rbrak$, and if, for each integer $n\geq 2$, we define $I_n:=\lbrak a_{n-1},a_n\rbrak$, with $A:=\bigcup_{n=1}^\infty I_n$, then by Lemma~\ref{BVTtoMCPLem}, the function $f:\R\into\R$ defined by
\begin{eqnarray}
    f(x)=\begin{cases} c_1, & x \in I_1, \\
    c_{n-1} + \frac{x-a_{n-1}}{a_n-a_{n-1}}(c_n - c_{n-1}), & x \in I_n, n \geq 2, \\ 0, & x \notin A, \end{cases}\nonumber
\end{eqnarray}
is continuous on $\R$, with $\{c_n\  :\  \in\N\}\sub f[\{a_n\  :\  n\in\N\}]$. Since $\aseq$ is monotonically increasing and bounded above by $b$, for each $n\in\N$, we have $a:=a_1\leq a_n\leq b$. Thus, $\{a_n\  :\  n\in\N\}\sub I:=\abClosed$, and so, $\{c_n\  :\  \in\N\}\sub f[\{a_n\  :\  n\in\N\}]\sub f[I]$. Since $\cseq$ is not bounded above, for each $M\in\R$, there exists $n\in\N$, and hence there exists $c_n\in f[I]$, such that $c_n>M$. Consequently, there exists $x\in I=\abClosed$ such that $f(x)=c_n>M$. That is, we have shown that for any $M\in\R$, there exists $x$ such that $a\leq x\leq b$ but $f(x){\not\leq}M$, where $f$, which is continuous on $\R$, is continuous on $\abClosed\sub\R$. This contradicts the \BVT. Therefore, any monotonically increasing sequence that is bounded above converges. \\

\noindent$\MCP\implies\AP\&\wNIP$. See Theorem~\ref{FirstCircle}.
\end{proof}

\section{The Fifth Circle: Integration}

The more precise setting for studying integration theorems equivalent to Dedekind completeness is that in an arbitrary ordered field, just like that done in the paper \cite{olm73} or in \cite[Appendix~A2]{tei13}. One of the main technical issues is the difference between two notions that are trivially equivalent in $\R$ but the difference does matter in an arbitrary ordered field: integrability of a function (how close ``upper and lower step function integrals'' are) versus the existence of the actual number that serves as the integral. Since we are using a naive approach, the integrability theorems here are not stated like that done in the standard references \cite{dev14,pro13,tei13} for this topic. [Also, in the statement of the Darboux Integral Theorem, ``upper integral'' was used in \cite{rie01}, but we have ``lower integral.'' This is mainly because we avoided infima.] Our goal is to preserve as much as possible the logical structure in the proving done in \cite{rie01}, and hence is the appearance and structure of this section. However, from the original four integration theorems in \cite{rie01}, we have here six. The additional two are because, in this author's opinion, the logical structure of the proving in \cite{rie01} for integration theorems is made clearer this way. Another technical issue concerns one form of the Fundamental Theorem of Calculus (that which is about differentiation of an integral). A modified form of this was used in \cite[p.~273]{dev14} just for the theorem to qualify for being in the list. A notion called ``uniform differentiability'' was introduced to accomplish this goal, but this does not seem compatible with our naive approach. Also, a closer look at the logic done in \cite{rie01} shows that the said form of the Fundamental Theorem of Calculus need only be joined by conjunction with another theorem [which is neither \AP\  nor \CC] for it to be part of a circle of real analysis principles. Whether this is unsavory or not is something we disregard, simply because the logic works.

The real analysis principles concerning integration that we shall be interested in are the following.\\
\begin{adjustwidth}{6em}{}
\begin{enumerate}
    \item[\DIT] \emph{Darboux Integral Theorem.} If $f$ is continuous on $\abClosed$, then the lower Darboux integral $\abIntLow f$ exists.
    \item[\RIT] \emph{Riemann Integrability Theorem.} If $f$ is continuous on $\abClosed$, then $f$ is Riemann integrable over $\abClosed$.
    \item[\FTCi] \emph{Fundamental Theorem of Calculus, First Form.} If $f$ is continuous on $\abClosed$, then the function $\abClosed\into\R$ denoted by $x\mapsto\int_a^xf$ is an antiderivative of $f$. 
            \item[\IAT] \emph{Increasing Antiderivative Theorem.} If $f$ is continuous and nonnegative on $\abClosed$, then any antiderivative of $f$ is monotonically increasing on $\abClosed$.\\
\end{enumerate}
\end{adjustwidth}
Similar to previous circles of theorems, the above had to be mentioned first before showing the full list of equivalences because each of the above principles has to be joined with the \AP\  or \CC\  [in this case, only \CC] for it to be included in the list of equivalences. The preliminaries for integration theory are in Section~\ref{FifthCirclePrelims}. 

The proof of Theorem~\ref{FifthCircle} has far more dependence on previously proven equivalences than in the previous four circles, and this is part of the reason why the \BVT\  and the \MVT\  are sort of ``inserted'' in the middle of the list of equivalences.

\begin{theorem}\label{FifthCircle} The following are equivalent.\\
\begin{adjustwidth}{6em}{}
\begin{enumerate}
 \item[\UAS] \emph{Uniform Approximation by Step Functions.} If $f$ is continuous on $\abClosed$, then for each $\varepsilon>0$, there exists a step function $\varphi:\abClosed\into\R$ such that for any $x\in\abClosed$, we have $|f(x)-\varphi(x)|<\varepsilon$.

    \item[\CC\ \emph{\&} \DIT] [The conjunction of the statements of the Countable Cofinality of $\R$ and the Darboux Integral Theorem.]
        \item[\CC\ \emph{\&} \RIT] [The conjunction of the statements of the Countable Cofinality of $\R$ and the Riemann Integral Theorem.]
        \item[\CC\ \emph{\&} \BVT] [The conjunction of the statements of the Countable Cofinality of $\R$ and the Bounded Value Theorem.]
                   \item[\MVT] \emph{Mean Value Theorem.} If $f:\abClosed\into\R$ is continuous on $\abClosed$ and differentiable on $\abOpen$, then there exists $c\in\abOpen$ such that $f'(c)=\frac{f(b)-f(a)}{b-a}$.
            \item[\FTCi\ \emph{\&} \IAT] [The conjunction of the statements of the Fundamental Theorem of Calculus, First Form\footnote{Fundamental Theorem of Calculus, First Form: ``differentiation of an integral''}, and the Increasing Antiderivative Theorem.]
        \item[\FTCii] \emph{Fundamental Theorem of Calculus, Second Form\footnote{Fundamental Theorem of Calculus, Second Form: ``integration of a derivative''}.} If $f$ is continuous on $\abClosed$ and $F$ is an antiderivative of $f$, then $\abInt f=F(b)-F(a)$.
    \item[\ADT] \emph{Antiderivative Difference Theorem.} If $f$ is continuous on $\abClosed$ and if $F$ and $G$ are antiderivatives of $f$, then $F-G$ is a constant function.  
                      \item[\CVT] \emph{Constant Value Theorem.} If $f'$ is zero on $\abOpen$, then $f$ is constant on $\abOpen$.
\end{enumerate}
\end{adjustwidth}
\end{theorem}
\begin{proof}$\UAS\implies\CC\&\DIT$. Using Theorem~\ref{FourthCircle} and then Theorem~\ref{SecondCircle}, we have
\begin{eqnarray}
    \UAS\implies\CC\&\BVT\implies\MCP\implies\ES.\label{UAStoCCandBVT}
\end{eqnarray}
In particular, \UAS\  implies \CC\  (as desired), and also \ES, which will be used in the following proof of the \DIT. Suppose $f$ is continuous on $\abClosed$. By Lemma~\ref{negfunctionLem}, so is $-f$. Since $1>0$, by \UAS, there exist $\phi,\theta\in\abSteps$ such that $\phi\leq f$ and $\theta\leq-f$ [which implies $f\leq -\theta$] such that $|f-\phi|<1$ and $|f-\theta|<1$. The conditions $\phi\in\abSteps$ and $\phi\leq f$ mean that $\int_a^b\phi\in\theSet:=\left\{\int_a^b\varphi\  :\  \varphi\in\abSteps,\  \varphi\leq f\right\}$, so $\theSet$ is nonempty. By Proposition~\ref{negStepIntProp}, $-\theta\in\abSteps$. If $\varphi\in\abSteps$ such that $\varphi\leq f$, then, because $f\leq -\theta$, we have $\varphi\leq -\theta$, and by Proposition~\ref{monotoneProp}, $\int_a^b\varphi\leq \int_a^b(-\theta)$, where the left-hand side may be considered as an arbitrary element of $\theSet$. Thus, the real number $\int_a^b(-\theta)$ is an upper bound of the nonempty set $\theSet$. By \ES, $\abIntLow f=\sup\theSet$ exists. This proves the \DIT.\\

\noindent$\UAS\implies\CC\&\RIT$. From \eqref{UAStoCCandBVT}, we find that \UAS\  implies \CC, as desired. Also, we have proven that \UAS\  implies \DIT, so the assumption that $f$ is continuous on $\abClosed$, which, by Lemma~\ref{negfunctionLem}, implies the continuity of $-f$ on $\abClosed$, by the \DIT, implies that $\abIntLow f$ and $r:=\abIntLow (-f)$ exist in $\R$. By \eqref{DarbouxNeg}, $-r=-\abIntLow(-f)=\abIntUp f$ exists in $\R$. That is, $\UAS$ implies the existence of both $\abIntLow f$ and $\abIntUp f$, which shall be used in the following argument.

Let $\varepsilon>0$. If $\abClosed$ is degenerate [$a=b$], then by Proposition~\ref{DegenerateIntProp}, the statement \RIT\  is trivially true. Henceforth, suppose $a<b$, so $\frac{\varepsilon}{4(b-a)}>0$. We assume that $f$ is continuous on $\abClosed$, and by Lemma~\ref{negfunctionLem}, so is $-f$. By the \UAS, there exist $\varphi,\psi\in\abSteps$ such that $|f-\varphi|<\frac{\varepsilon}{4(b-a)}$ and $|(-f)-\psi|<\frac{\varepsilon}{4(b-a)}$, which imply
\begin{eqnarray}
    \varphi-\frac{\varepsilon}{4(b-a)} <&f&<\varphi+\frac{\varepsilon}{4(b-a)},\label{RITineq1}\\
    \psi-\frac{\varepsilon}{4(b-a)} <&-f&<\psi+\frac{\varepsilon}{4(b-a)}.\label{RITineq2}
\end{eqnarray}
The first inequalities in \eqref{RITineq1}--\eqref{RITineq2} imply $\varphi-\frac{\varepsilon}{4(b-a)}\leq f$ and $\psi-\frac{\varepsilon}{4(b-a)} \leq-f$, where the left-hand sides, according to Proposition~\ref{negStepIntProp} are in $\abSteps$, and the integrals of these step functions may also be computed using Proposition~\ref{negStepIntProp}. At this point, we have shown that
\begin{eqnarray}
    \int_a^b\varphi-\frac{\varepsilon}{4}&\in&\theSet_1:=\left\{\int_a^b\phi\  :\  \phi\in\abSteps,\  \phi\leq f\right\},\nonumber\\
        \int_a^b\psi-\frac{\varepsilon}{4}&\in&\theSet_2:=\left\{\int_a^b\phi\  :\  \phi\in\abSteps,\  \phi\leq -f\right\}.\nonumber
\end{eqnarray}
But since $\abIntLow f=\sup\theSet_1$ and $-\abIntUp f=\sup\theSet_2$ are upper bounds of $\theSet_1$ and $\theSet_2$, respectively, we have\linebreak $\int_a^b\varphi-\frac{\varepsilon}{4}\leq \abIntLow f$ and $\int_a^b\psi-\frac{\varepsilon}{4}\leq -\abIntUp f$. We multiply both sides of both inequalities by $-1$ and add the results to obtain
\begin{eqnarray}
    \abIntUp f-\abIntLow f\leq -\lbrak\int_a^b\varphi+\int_a^b\psi\rbrak+\frac{\varepsilon}{2}.\label{RITeqfinal1}
\end{eqnarray}
We multiply \eqref{RITineq2} by $-1$ and rearrange so that we have the new system
\begin{eqnarray}
    \varphi-\frac{\varepsilon}{4(b-a)} <&f&<\varphi+\frac{\varepsilon}{4(b-a)},\label{RITineq3}\\
    -\psi-\frac{\varepsilon}{4(b-a)} <&f&<-\psi+\frac{\varepsilon}{4(b-a)}.\label{RITineq4}
\end{eqnarray}
From the first inequality in \eqref{RITineq3} and the second inequality in \eqref{RITineq4}, we obtain $\varphi+\psi<\frac{\varepsilon}{2(b-a)}$, while from the first inequality in \eqref{RITineq4} and the second inequality in \eqref{RITineq3}, we obtain $-\frac{\varepsilon}{2(b-a)}<\varphi+\psi$. That is,
\begin{eqnarray}
   -\frac{\varepsilon}{2(b-a)}<\varphi+\psi <\frac{\varepsilon}{2(b-a)}.\label{RITsuperIneq}
\end{eqnarray}
According to Proposition~\ref{negStepIntProp}, the left, middle and right members of \eqref{RITsuperIneq} are all in $\abSteps$. We use Proposition~\ref{negStepIntProp} to compute their integrals, and by Proposition~\ref{monotoneProp}, we obtain $-\frac{\varepsilon}{2}<\int_a^b\varphi+\int_a^b\psi<\frac{\varepsilon}{2}$, or that
\begin{eqnarray}
\left|\int_a^b\varphi+\int_a^b\psi\right|<\frac{\varepsilon}{2}.    \label{RITeqfinal2}
\end{eqnarray}
From Proposition~\ref{SumtoStepProp}, $\abIntLow f\leq \abIntUp f$, which implies $0\leq \abIntUp f-\abIntLow f$, and by \eqref{RITeqfinal1},\eqref{RITeqfinal2},
$$0\leq \abIntUp f-\abIntLow f\leq -\lbrak\int_a^b\varphi+\int_a^b\psi\rbrak+\frac{\varepsilon}{2}\leq  \left|\int_a^b\varphi+\int_a^b\psi\right|+\frac{\varepsilon}{2}<\frac{\varepsilon}{2}+\frac{\varepsilon}{2}=\varepsilon.$$
That is,
\begin{eqnarray}
    0\leq \abIntUp f-\abIntLow f<\varepsilon.\label{RITfinaleq3}
\end{eqnarray}
If $\abIntUp f-\abIntLow f$ is not zero, then the first inequality in \eqref{RITfinaleq3} means that $\abIntUp f-\abIntLow f$ is positive, so we may substitute it to $\varepsilon$ to obtain $\abIntUp f-\abIntLow f<\abIntUp f-\abIntLow f$, contradicting the Trichotomy Law. Hence, $\abIntUp f-\abIntLow f$ is zero, or that $\abIntUp f=\abIntLow f$. Therefore, $f$ is Riemann integrable over $\abClosed$.\\

\noindent$\CC\&\Star\implies\CC\&\BVT$, where \Star\  is the \DIT\  or the \RIT. Showing that \Star\  implies \BVT\  shall suffice. The hypothesis in \Star\  is that $f$ is continuous on $\abClosed$, so suppose that this is indeed true. 

For the case when \Star\  is the \DIT, the continuity of $f$ on $\abClosed$, together with Lemma~\ref{negfunctionLem}, which implies the continuity of $-f$ on $\abClosed$, the identity \eqref{DarbouxNeg} and also \DIT\  itself, implies that $-\abIntLow(-f)=\abIntUp f$ exists in $\R$, so we may use Lemma~\ref{DITtoBVTLem} to show that $f$ is bounded above on $\abClosed$. If \Star\  is the \RIT, use Lemma~\ref{UAStoRITLem}, and then Lemma~\ref{RITtoBVTLem}.\\

\noindent$\CC\&\BVT\implies\MVT$. We have the implications
\begin{eqnarray}
   \CC\&\BVT\implies\MCP\implies\UIC\iff\MVT,\nonumber
\end{eqnarray}
using Theorems~\ref{FourthCircle},\ref{SecondCircle},\ref{ThirdCircle}, respectively.\\

\noindent$\MVT\implies\FTCi$. Using Theorem~\ref{ThirdCircle}, then Theorem~\ref{SecondCircle} and then Theorem~\ref{FourthCircle}, we have the implications
\begin{eqnarray}
    \MVT\iff\EVT\iff\UIC\implies\IVT\implies\MCP\iff\UAS.
\end{eqnarray}
That is, \MVT\  implies the \EVT, the \IVT\  and \UAS, all of which, we shall use in the proof of the following.\\
\begin{adjustwidth}{6em}{}
\begin{enumerate}
    \item[\iMVT] \emph{Mean Value Theorem for Integrals.} If $f$ is continuous on $\abClosed$, then there exists\linebreak $\xi\in\abClosed$ such that $(b-a)f(\xi)=\int_a^bf$.\\
\end{enumerate}
\end{adjustwidth}

If $a=b$, then the desired equation is simply $0=0$. Thus, we assume henceforth that $a<b$, or that $b-a$ is positive. If $f$ is continuous on $\abClosed$, then by Lemma~\ref{negfunctionLem}, so is $-f$, and by the \EVT, there exist $c_1,c_2\in\abClosed$ such that $f\leq f(c_1)$ and $-f\leq -f(c_1)$ [where the right-hand sides are constants] on $\abClosed$. Consequently, $f(c_1)\leq f\leq f(c_2)$ on $\abClosed$. By \UAS, $f$ is Riemann integrable over $\abClosed$, so by Proposition~\ref{mfMIntProp}, $f(c_1)\  (b-a)\leq \int_a^bf\leq f(c_2)\  (b-a)$, every member of which, we divide by the positive number $b-a$ to obtain $f(c_1)\leq \frac{1}{b-a}\int_a^bf\leq f(c_2)$. If $f(c_1)= \frac{1}{b-a}\int_a^bf$ or $\frac{1}{b-a}\int_a^bf= f(c_2)$, then we may take $\xi$ to be either $c_1$ or $c_2$ and we are done. For the remaining case $f(c_1)< \frac{1}{b-a}\int_a^bf< f(c_2)$, we look at two subcases: $c_1\leq c_2$ or $c_2<c_1$. Since $c_1,c_2\in\abClosed$, we further have $a\leq c_1\leq c_2\leq b$ or $a\leq c_2<c_1\leq b$. In the former case, the continuity of $f$ on $\abClosed$ implies the continuity of $f$ on $[c_1,c_2]\sub\abClosed$, so using $f(c_1)< \frac{1}{b-a}\int_a^bf< f(c_2)$ and the \IVT, there exists $\xi\in[c_1,c_2]\sub\abClosed$ such that $f(\xi)=\frac{1}{b-a}\int_a^bf$. If $a\leq c_2<c_1\leq b$, then we use the continuity of $-f$ on $\abClosed$, which implies the continuity of $-f$ on $[c_2,c_1]\sub\abClosed$, and also turn the previous inequalities into $-f(c_2)< \frac{-1}{b-a}\int_a^bf< -f(c_1)$, to use the \IVT\  to obtain $\xi\in[c_2,c_1]\sub\abClosed$ for which $-f(\xi)=\frac{-1}{b-a}\int_a^bf$, or that $f(\xi)=\frac{1}{b-a}\int_a^bf$. This completes the proof of the \iMVT.

We now prove \FTCi. Suppose $f$ is continuous on $\abClosed$, let $c\in\abClosed$ and let $x\in\abClosed$ that is distinct from $c$. We show that regardless of whether $x<c$ or $c<x$ [which imply $a\leq x<c\leq b$ or $a\leq c<x\leq b$], we have the equation
\begin{flalign}
    &&f(tx+(1-t)c) & =\frac{\int_a^xf-\int_a^cf}{x-c}, & (\mbox{for some }t\in\R).\label{iMVTtoFTCiEQ}
\end{flalign}
The continuity of $f$ on $\abClosed$ implies the continuity of $f$ on either $[x,c]\sub\abClosed$ or $[c,x]\sub\abClosed$. By the \iMVT, either there exists $\xi\in[x,c]$ or $\zeta\in[c,x]$ such that $f(\xi)=\frac{\int_x^cf}{c-x}$ or $f(\zeta)=\frac{\int_c^xf}{x-c}$. We use $t=\frac{c-\xi}{c-x}$ if $x<c$, or $t=\frac{\zeta-c}{x-c}$ if $c<x$, and so, either $f(tx+(1-t)c)=\frac{\int_x^cf}{c-x}$ if $x<c$, or $f(tx+(1-t)c)=\frac{\int_c^xf}{x-c}$ if $c<x$. Using Lemma~\ref{IntAdditivityLem}, either $\int_a^cf=\int_a^xf+\int_x^cf$ if $a\leq x<c$, or $\int_a^xf=\int_a^cf+\int_c^xf$ if $a\leq x<c$. Thus, either $\int_a^cf-\int_a^xf=\int_x^cf$ if $a\leq x<c$, or $\int_a^xf-\int_a^cf=\int_c^xf$ if $a\leq x<c$. At this point, we have either $f(tx+(1-t)c)=\frac{\int_a^cf-\int_a^xf}{c-x}$ or $f(tx+(1-t)c)=\frac{\int_a^xf-\int_a^cf}{x-c}$, where the right-hand sides are related by $\frac{\int_a^xf-\int_a^cf}{x-c}=\frac{-1}{-1}\cdot \frac{\int_a^xf-\int_a^cf}{x-c}=\frac{\int_a^cf-\int_a^xf}{c-x}$. This proves \eqref{iMVTtoFTCiEQ}. For any filter base $\Base$ that approaches $c$, the function $\varphi:x\mapsto tx+(1-t)c$, by Proposition~\ref{AlgLimProp} satisfies $\BcLim \varphi=c$, and at this limit value, $f$ is continuous, so by Proposition~\ref{AlgDiffProp}\ref{AlgConti}, $f\lpar \BcLim \varphi\rpar=f(c)=\xcLim f(x)$. By Proposition~\ref{twoSidedLimProp}, we have $f\lpar\BcLim \varphi\rpar=f(c)=\xcLimL f(x)$ and $f\lpar\BcLim \varphi\rpar=f(c)=\xcLimR f(x)$. In the former, we make $\Base$ be the filter base for $x\into c-$, while for the latter, let $\Base$ be the filter base for $x\into c+$. That is, the limit of the left-hand side of \eqref{iMVTtoFTCiEQ} as $x\into c-$ and as $x\into c+$ are both $f(c)$, and we now have $\xcLimL\frac{\int_a^xf-\int_a^cf}{x-c}=f(c)=\xcLimR\frac{\int_a^xf-\int_a^cf}{x-c}$. By Proposition~\ref{twoSidedLimProp}, $\xcLim\frac{\int_a^xf-\int_a^cf}{x-c}=f(c)$. Hence, the function $x\mapsto \int_a^xf$ is differentiable at $c$, and the derivative is $c\mapsto f(c)$, or is $f$. Therefore, $x\mapsto \int_a^xf$ is an antiderivative of $f$.\\

\noindent$\MVT\implies\IAT$. Suppose $a\leq c<x\leq b$. If $F$ is an antiderivative of $f$ on $\abClosed$, then $F$ is differentiable on $\abClosed$, and in particular, on $(c,x)\sub\abClosed$. By Proposition~\ref{AlgDiffProp}\ref{DiffConti}, $F$ is continuous on $\abClosed$, and in particular, on $[c,x]\sub\abOpen$. By the \MVT, there exists $\xi\in(c,x)$ such that $F'(\xi)=\frac{F(x)-F(c)}{x-c}$, where the left-hand side is equal to $f(\xi)$, and is, by assumption, nonnegative. Thus, $0\leq \frac{F(x)-F(c)}{x-c}$. From $a\leq c<x\leq b$, we find that $x-c$ is positive, so we further have $0\leq F(x)-F(c)$, or that $F(c)\leq F(x)$. Hence, $a\leq c<x\leq b$ implies $F(c)\leq F(x)$, or that $F$ is monotonically increasing on $\abClosed$.\\

\noindent$\MVT\implies\FTCii$. Let $\varphi,\psi:\abClosed\into\R$ be step functions such that $\varphi\leq f\leq \psi$. By Proposition~\ref{anypartitionProp}, there exists a partition $\Delta = \{x_0,x_1,\ldots,x_n\}$ of $\abClosed$ and there exist real numbers $m_0$, $m_1$, $\ldots$ , $m_n$, $M_1$, $M_2$, $\ldots$ , $M_n$ such that $\int_a^b\varphi=\sum_{k=1}^n m_k(x_k-x_{k-1})$ and $\int_a^b\psi=\sum_{k=1}^n M_k(x_k-x_{k-1})$. Let $k\in\{1,2,\ldots,n\}$. Since $\varphi\leq f\leq \psi$,
\begin{eqnarray}
m_k=\varphi(x)\leq f(x)\leq \psi(x) = M_k,\qquad\qquad\mbox{if }x\in\lpar x_{k-1},x_k\rpar,\nonumber
\end{eqnarray}
and since, by the notation for partitions, $x_k>x_{k-1}$, or that $x_k-x_{k-1}>0$, we further have
\begin{eqnarray}
m_k(x_k-x_{k-1})\leq (x_k-x_{k-1})f(x)\leq M_k(x_k-x_{k-1}),\qquad\qquad\mbox{if }x\in\lpar x_{k-1},x_k\rpar.\label{FTC1}
\end{eqnarray}
Since $F$ is an antiderivative of $f$ on $\abClosed$, we find that $F$ is differentiable on $\abClosed$, and in particular, on $\lpar x_{k-1},x_k\rpar\sub\abClosed$ for any $k\in\{1,2,\ldots,n\}$. By Proposition~\ref{AlgDiffProp}\ref{DiffConti}, $F$ is continuous on $\abClosed$, and in particular, on $\lpar x_{k-1},x_k\rpar\sub\abClosed$ for any $k\in\{1,2,\ldots,n\}$. By the \MVT, there exists $\xi_k\in \lpar x_{k-1},x_k\rpar$ such that $F(x_k)-F(x_{k-1})=(x_k-x_{k-1})F'(\xi_k)=(x_k-x_{k-1})f(\xi_k)$, and so, if we set $x=\xi_k$ in \eqref{FTC1},
\begin{eqnarray}
m_k(x_k-x_{k-1})\quad\leq & F(x_k)-F(x_{k-1}) & \leq\quad M_k(x_k-x_{k-1}),\nonumber\\
\sum_{k=1}^n m_k(x_k-x_{k-1})\quad\leq & \displaystyle\sum_{k=1}^n F(x_k)- \sum_{k=1}^n F(x_{k-1}) & \leq\quad \sum_{k=1}^n M_k(x_k-x_{k-1}),\nonumber\\
\int_a^b\varphi \quad\leq& \displaystyle\sum_{k=1}^{n} F(x_k)- \sum_{k=0}^{n-1} F(x_{k})  &\leq\quad\int_a^b\psi,\nonumber\\
\int_a^b\varphi \quad\leq& F(x_n)+\displaystyle\sum_{k=1}^{n-1} F(x_k)- \sum_{k=1}^{n-1} F(x_{k}) - F(x_0)  &\leq\quad\int_a^b\psi,\nonumber\\
\int_a^b\varphi \quad\leq& F(b)-F(a)  &\leq\quad\int_a^b\psi.\label{FTC2}
\end{eqnarray}
At this point, we have proven that the real number $F(b)-F(a)$ has the property that, for any step functions $\varphi,\psi:\abClosed\into\R$, if $\varphi\leq f\leq \psi$ then $\int_a^b\varphi\leq F(b)-F(a)\leq\int_a^b\psi$. According to Proposition~\ref{RiemannUniqueProp}, the only real number with this property is $\int_a^bf$. Therefore $\int_a^bf=F(b)-F(a)$.\\

\noindent$\FTCi\&\IAT\implies\ADT$. Since $F$ and $G$ are both antiderivatives of $f$, we have $F'=f=G'$. [At least one such antiderivative exists according to \FTCi.] Using Proposition~\ref{AlgDiffProp}, we have\linebreak $(F-G)'=F'-G'=f-f=0$ and $(G-F)'=G'-F'=f-f=0$. Thus, $F-G$ and $G-F$ are antiderivatives of the zero function $0:\abClosed\into\R$, which is nonnegative on $\abClosed$, and by Example~\ref{contEx}, is continuous on $\abClosed$. By the \IAT, both $F-G$ and $G-F$ are monotonically increasing on $\abClosed$. That is, $a\leq c<x\leq b$ implies $(F-G)(c)\leq (F-G)(x)$ and $(G-F)(c)\leq (G-F)(x)$, where to both sides of the latter, we multiply $-1$ to obtain $(F-G)(c)\geq (F-G)(x)$. At this point, we have $(F-G)(c)=(F-G)(x)$ whenever $a\leq c<x\leq b$. Therefore, $F-G$ is a constant function.\\

\noindent$\FTCii\implies\ADT$. Suppose $a\leq c<x\leq b$. The continuity of $f$ on $\abClosed$ implies the continuity of $f$ on $[c,x]\sub\abClosed$. Since $F$ and $G$ are antiderivatives of $f$, by \FTCii, we obtain\linebreak $G(x)-G(c)=\int_c^xf=F(x)-F(c)$, so $G(x)-F(x)=G(c)-F(c)$, or that $(G-F)(x)=(G-F)(c)$. Therefore, $F-G$ is a constant function.\\

\noindent$\ADT\implies \CVT$. Suppose $a<c<x<b$. On the interval $[c,x]\sub\abOpen$, the derivatives of $f$ and the zero function $0$ are $f'$ and $0$, respectively, where for the former, by assumption, $f'=0$. That is, $0$ is the derivative of both $f$ and $0$. Equivalently, $f$ and $0$ are both antiderivatives of $0$, which, by Example~\ref{contEx}, is continuous on $[c,x]$. By the \ADT, the difference $f-0=f$ is constant on $[c,x]$. In particular, $f(c)=f(x)$. At this point, we have proven that $a<c<x<b$ implies $f(c)=f(x)$. Therefore, $f$ is constant on $\abOpen$. \\

\noindent$\CVT\implies\UAS$. We have the implications

$$\CVT\implies\UIC\implies\MCP\iff\UAS,$$
from Theorems~\ref{ThirdCircle},\ref{SecondCircle},\ref{FourthCircle}, respectively.
\end{proof} 

\  \\ \\

\section*{\centering Logical Relationships of Real Analysis Principles}

We now collect all the implications in the proofs of Theorems~\ref{FirstCircle}--\ref{FifthCircle}, and we summarize the logical connections in the following diagram.\\ \\

\begin{tikzpicture}[
scale=0.7,             
    transform shape,       
    double arrow/.style={
        double, 
        double distance=1.5pt, 
        draw=#1, 
        ->, 
        >={Latex[length=4pt, width=5pt, color=#1]}, 
        shorten >= 3pt, 
        shorten <= 3pt,
        rounded corners=10pt 
    },
    vertex/.style={
        draw, 
        rectangle, 
        rounded corners=5pt, 
        minimum width=40pt, 
        minimum height=20pt
    }
]

\node[vertex, draw=violet, fill=violet!10] (ADT) at (0,1.5) {\ADT};
\node[vertex, draw=violet, fill=violet!10] (FTC1) at (-3,3) {\FTCi\&\IAT};
\node[vertex, draw=violet, fill=violet!10] (FTC2) at (0,3) {\FTCii};
\node[vertex, draw=teal, fill=green!10] (MVT) at (0,4.5) {\MVT};

\node[vertex, draw=teal, fill=green!10] (TT) at (5.75,4.5) {\TT}; 
\node[vertex, draw=teal, fill=green!10] (CFT) at (11.5,3.5) {\CFT}; 
\node[vertex, draw=teal, fill=green!10] (IFT) at (11.5,4.5) {\IFT};
\node[vertex, draw=teal, fill=green!10] (PCP) at (11.5,5.5) {\PCP};
\node[vertex, draw=teal, fill=green!10] (CVT) at (17.25,4.5) {\CVT};
        
        \node[vertex, draw=teal, fill=green!10] (eMVT) at (0,6) {\eMVT};
        \node[vertex, draw=teal, fill=green!10] (RT) at (0,7.5) {\RT};
        \node[vertex, draw=teal, fill=green!10] (EVT) at (0,9) {\EVT};
        \node[vertex, draw=red, fill=red!10] (BWP) at (0,10.5) {\BWP};
        \node[vertex, draw=red, fill=red!10] (CCC) at (0,13.5) {\AP\&\CCC};

        \node[vertex, draw=brown, fill=yellow!10] (ES) at (3.875,9) {\ES};
        \node[vertex, draw=red, fill=red!10] (wNIP) at (3.875,10.5) {\AP\&\wNIP};
        \node[vertex, draw=red, fill=red!10] (sNIP) at (3.875,12) {\AP\&\sNIP};
        \node[vertex, draw=red, fill=red!10] (MCP) at (3.875,13.5) {\MCP};
        \node[vertex, draw=brown, fill=yellow!10] (CA) at (3.875,15) {\CA};
        \node[vertex, draw=brown, fill=yellow!10] (I4) at (3.875,16.5) {\RIC};

        \node[vertex, draw=blue, fill=blue!10] (BVT) at (7.75,13.5) {\CC\&\BVT};

        \node[vertex, draw=violet, fill=violet!10] (DIT) at (46.5/4,12) {\CC\&\DIT};
        \node[vertex, draw=blue, fill=blue!10] (UAS) at (46.5/4,13.5) {\UAS};
        \node[vertex, draw=violet, fill=violet!10] (RIT) at (46.5/4,15) {\CC\&\RIT};
       \node[vertex, draw=brown, fill=yellow!10] (IVT) at (8.3333,16.5) {\IVT}; 
\node[vertex, draw=brown, fill=yellow!10] (I3) at (12.7917,16.5) {\AAC};

        \node[vertex, draw=blue, fill=blue!10] (UCT) at (15.5,13.5) {\AP\&\UCT};
        \node[vertex, draw=blue, fill=blue!10] (LCL) at (15.5,12) {\AP\&\LCL};

        \node[vertex, draw=blue, fill=blue!10] (I5) at (15.5,10.5) {\UIK};


\node[vertex, draw=brown, fill=yellow!10] (I2) at (17.25,16.5) {\NTD};
        
        \node[vertex, draw=brown, fill=yellow!10] (I1) at (17.25,9) {\UIC};

        \draw[double arrow=violet] (ADT) -| (CVT);
        \draw[double arrow=violet] (FTC1) |- (ADT);
        \draw[double arrow=violet] (MVT) -| (FTC1);
        \draw[double arrow=violet] (MVT) -- (FTC2);
        \draw[double arrow=violet] (FTC2) -- (ADT);

\draw[double arrow=teal] (MVT) -- (TT);

        \draw[double arrow=teal] (TT) -- (CFT);
        \draw[double arrow=teal] (TT) -- (PCP);
        
        \draw[double arrow=teal] (TT) -- (IFT);
        \draw[double arrow=teal] (CFT)--(CVT);
        \draw[double arrow=teal] (PCP)--(CVT);
        \draw[double arrow=teal] (IFT)--(CVT);
        \draw[double arrow=teal] (CVT)--(I1);
        \draw[double arrow=teal] (BWP)--(EVT);
        \draw[double arrow=teal] (EVT)--(RT);
        \draw[double arrow=teal] (RT)--(eMVT);
        \draw[double arrow=teal] (eMVT)--(MVT);
        \draw[double arrow=teal] (ES)--(EVT);

        \draw[double arrow=violet] (UAS)-- (RIT);
        \draw[double arrow=violet] (UAS)-- (DIT);
        \draw[double arrow=violet] (RIT)-| (BVT);
        \draw[double arrow=violet] (DIT)-| (BVT);

        \draw[double arrow=blue] (wNIP) -- (I5);
        \draw[double arrow=blue] (I5) -- (LCL);
        \draw[double arrow=blue] (LCL) -- (UCT);
        \draw[double arrow=blue] (UCT) -- (UAS);
        \draw[double arrow=blue] (UAS) -- (BVT);
        \draw[double arrow=blue] (BVT) -- (MCP);

        \draw[double arrow=brown] (wNIP) -- (ES);
        \draw[double arrow=brown] (ES) -- (I1);
        \draw[double arrow=brown] (I1) -- (I2);
        \draw[double arrow=brown] (I2) -- (I3);
        \draw[double arrow=brown] (I3) -- (IVT);
        \draw[double arrow=brown] (IVT) -- (I4);
        \draw[double arrow=brown] (I4) -- (CA);
        \draw[double arrow=brown] (CA) -- (MCP);

        \draw[double arrow=red] (CCC) -- (MCP);
        \draw[double arrow=red] (MCP) -- (sNIP);
        \draw[double arrow=red] (sNIP) -- (wNIP);
        \draw[double arrow=red] (wNIP) -- (BWP);
        \draw[double arrow=red] (BWP) -- (CCC);
\end{tikzpicture}\\ \\

The First, Second and Fourth Circles, as proven in Theorems~\ref{FirstCircle},\ref{SecondCircle},\ref{FourthCircle}, may easily be traced on the above diagram. The Third Circle from Theorem~\ref{ThirdCircle} has some more involvement with the First and Second Circles: start at the \EVT, go down to the \MVT, go to the right until the \CVT, but then we need to go up to \NTD, go to the left to \RIC, and trace from there downward paths to the \BWP\  and \ES\  to get back to the \EVT. The situation is less simple regarding the Fifth Circle, for as we pointed out regarding the proof of Theorem~\ref{FifthCircle}, there is much more dependence on previously established equivalences from the other four circles. For the Fifth Circle, we just have to think of \UAS\  and \UIC\ as equivalent, and this has been established in previous circles of real analysis principles. That is, even if there is no bidirectional arrow in the above diagram connecting \UAS\  and \UIC, such exists based on previous equivalences. All we need now is a path from the \UAS\  to \UIC. Ignoring the direct path from the $\UAS$ to $\CC\&\BVT$, there are two paths from $\UAS$ to $\CC\&\BVT$, one passing through $\CC\&\RIT$ and the other through $\CC\&\DIT$. [In the proof of Theorem~\ref{FifthCircle} for these implications, we needed the equivalence of the \UAS\  to the \ES, which is true based on previous circles, but the diagram does not reflect this anymore, for simplicity.] From $\CC\&\BVT$, trace a path that goes down to the \MVT, from which, two paths branch out but both go to the \ADT. [As may be seen in the proof of Theorem~\ref{FifthCircle}, to prove that the \MVT\  implies the \FTCi, we needed the equivalence of the \MVT\  to the \IVT\  and the \EVT. This is also not shown anymore in the above diagram for simplicity.] From the \ADT, we go to the \CVT, then to \UIC, as desired. That is the Fifth Circle. 

Completing the five circles of real analysis theorems seems to be a good tour of topology, and this may give a concrete impression on how topology serves as the foundation for the calculus that we have learned in the undergraduate level. The claim in \cite{rie01} that the arrangement of the theorems in the five circles is in such a way that the equivalence to \ES\  [or \CA] may be seen ``without much extra work,'' cannot be proven, or at least cannot be proven easily, but the conceptual framework provided by the five circles gives us a good organizing principle for the jungle that real analysis is. We hope to have provided a good supplement for anyone intending to read the main references \cite{dev14,pro13,rie01,tei13}, and also to anyone studying real analysis.


\end{document}